\documentclass{amsart}
\usepackage{amsopn}
\usepackage{amssymb, amscd, amsmath}
\usepackage{array}
\usepackage{color}

\newcommand{\nc}{\newcommand}

\DeclareRobustCommand{\SkipTocEntry}[4]{}

\nc{\detb}{\det_{\Beta_\mg}}
\nc{\Slb}{\Sl_{\Beta}}
\nc{\Slbm}{\Sl_{\Beta_\mg}^H}
\nc{\slbm}{\slg_{\Beta_\mg}^H}
\nc{\ggl}{\mathfrak{g}}
\nc{\ppm}{\mathfrak{p}}
\nc{\GG}{G}
\nc{\Gm}{\Gl(\mg)}
\nc{\Gg}{\Gl(\ggl)}
\nc{\glgg}{{\mathfrak{gl}(\mathfrak{g})}}
\nc{\GH}{\Gl^{\!H}}
\nc{\GHm}{\Gl^{\!H} (\mg)}
\nc{\GHmk}{\Gl^H(\mg_k)}
\nc{\gHm}{{\mathfrak{gl}^H(\mathfrak{m})}}
\nc{\OHm}{\Or^H(\mg)}
\nc{\sogHm}{ {\mathfrak{so}^H(\mathfrak{m})} }
\nc{\Vm}{V(\mg)}
\nc{\Vg}{V(\ggl)}
\nc{\Vmi}{V(\mg_\infty)}
\nc{\Om}{\Or(\mg)}
\nc{\Og}{\Or(\ggl)}
\nc{\Symm}{{\rm Sym}(\mg)}
\nc{\Symg}{{\rm Sym}(\ggl)}
\nc{\gm}{\mathfrak{gl}(\mg)}
\nc{\som}{\sog(\mg)}
\nc{\sogg}{{\mathfrak{so}(\mathfrak{g})}}
\nc{\hml}{{\mu^{\ggl}}}
\nc{\ml}{{\mu^{\ggl}_{\mg}}}
\nc{\Ol}{{\mathcal{O}_{\ggl}}}
\nc{\OlH}{{\mathcal{O}^H_{\ggl}}}
\nc{\Vn}{V(n)}
\nc{\SymV}{{\rm Sym}(V)}
\nc{\mub}{{\bar \mu}}
\nc{\mumg}{{\mu_\mg}}
\nc{\Betam}{{\Beta_\mg}}
\nc{\Betag}{{\Beta_\ggo}}
\nc{\Betagp}{{\Beta_\ggo^+}}
\nc{\ii}{{\mathrm{i}}}
\nc{\Nrm}{{\mathrm{N}}}
\nc{\Srm}{{\mathrm{S}}}
\nc{\spec}{{\operatorname{spec}}}
\nc{\iR}{{\ii\RR}}

\nc{\Vzerog}{V_{\Beta_\ggo^+}^{0}}
\nc{\Vnng}{V_{\Beta_\ggo^+}^{\geq 0}}
\nc{\Vnnssg}{U_{\Beta_\ggo^+}^{\geq 0}}
\nc{\Vzerossg}{U_{\Beta_\ggo^+}^{0}}

\nc{\GGs}{S}
\nc{\ggs}{\mathfrak{s}}

 \nc{\ggo}{\mathfrak{g}}
 \nc{\ggob}{\overline{\mathfrak{g}}}
\nc{\lamg}{\Lambda^2\ggo^*\otimes\ggo}
\nc{\gkp}{(\ggo=\kg\oplus\pg,\ip)} \nc{\ukh}{(\ug=\kg\oplus\hg,\ip)}
\nc{\tgkp}{(\tilde{\ggo}=\kg\oplus\pg,\ip)}

\nc{\fg}{\mathfrak{f}}  \nc{\vg}{\mathfrak{v}} \nc{\wg}{\mathfrak{w}} \nc{\zg}{\mathfrak{z}} \nc{\ngo}{\mathfrak{n}} \nc{\kg}{\mathfrak{k}} \nc{\mg}{\mathfrak{m}} \nc{\bg}{\mathfrak{b}}  \nc{\sog}{\mathfrak{so}} \nc{\sug}{\mathfrak{su}} \nc{\spg}{\mathfrak{sp}} \nc{\slg}{\mathfrak{sl}} \nc{\glg}{\mathfrak{gl}} \nc{\cg}{\mathfrak{c}} \nc{\rg}{\mathfrak{r}}  \nc{\hg}{\mathfrak{h}} \nc{\tgo}{\mathfrak{t}} \nc{\ug}{\mathfrak{u}} \nc{\dg}{\mathfrak{d}} \nc{\ag}{\mathfrak{a}} \nc{\pg}{\mathfrak{p}} \nc{\sg}{\mathfrak{s}} \nc{\affg}{\mathfrak{aff}} \nc{\qg}{\mathfrak{q}}
\nc{\Xg}{\mathfrak{X}} \nc{\lgo}{\mathfrak{l}} \nc{\tg}{\mathfrak{t} }

\nc{\pca}{\mathcal{P}} \nc{\nca}{\mathcal{N}} \nc{\lca}{\mathcal{L}} \nc{\oca}{\mathcal{O}} \nc{\mca}{\mathcal{M}} \nc{\tca}{\mathcal{T}} \nc{\aca}{\mathcal{A}} \nc{\cca}{\mathcal{C}} \nc{\gca}{\mathcal{G}} \nc{\sca}{\mathcal{S}} \nc{\hca}{\mathcal{H}} \nc{\bca}{\mathcal{B}} \nc{\dca}{\mathcal{D}} \nc{\fca}{\mathcal{F}} \nc{\Qca}{\mathcal{Q}}

\nc{\dd}{{\rm d}}  \nc{\ddt}{\tfrac{{\rm d}}{{\rm d}t}}        \nc{\dds}{\tfrac{{\rm d}}{{\rm d}s}} 
\nc{\ddtbig}{\frac{{\rm d}}{{\rm d}t}}      \nc{\dpar}{\tfrac{\partial}{\partial t}}    

\nc{\im}{\mathtt{i}}    \renewcommand{\Im}{{\rm Im}}

\nc{\SO}{\mathrm{SO}}   \nc{\Spe}{\mathrm{Sp}}      \nc{\Sl}{\mathrm{SL}}       \nc{\Gl}{\mathrm{GL}}
\nc{\SU}{\mathrm{SU}}   \nc{\Or}{\mathrm{O}}        \nc{\U}{\mathrm{U}} 
\nc{\Se}{\mathrm{S}}    \nc{\Cl}{\mathrm{Cl}}       \nc{\Spein}{\mathrm{Spin}}
\nc{\Pin}{\mathrm{Pin}} \nc{\G}{\mathrm{GL}_n(\RR)} \nc{\g}{\mathfrak{gl}_n(\RR)}

\nc{\RR}{{\mathbb R}} \nc{\HH}{{\mathbb H}} \nc{\CC}{{\mathbb C}} \nc{\ZZ}{{\mathbb Z}}
\nc{\FF}{{\mathbb F}} \nc{\NN}{{\mathbb N}} \nc{\QQ}{{\mathbb Q}} \nc{\PP}{{\mathbb P}}
\nc{\KK}{{\mathbb K}}

\nc{\vs}{\vspace{.2cm}} \nc{\vsp}{\vspace{1cm}} 
\nc{\ip}{{\langle \,\cdot \,,\cdot \,\rangle }}
 \nc{\la}{\langle} \nc{\ra}{\rangle} \nc{\unm}{\tfrac{1}{2}}
\nc{\unc}{\tfrac{1}{4}} \nc{\und}{\tfrac{1}{16}} \nc{\no}{\vs\noindent}
\nc{\lam}{\Lambda^2(\RR^n)^*\otimes\RR^n} \nc{\tangz}{{\rm T}^{\rm Zar}}

\nc{\nor}{{\sf n}}  \nc{\mum}{/\!\!/} \nc{\kir}{/\!\!/\!\!/}
\nc{\Ri}{\tfrac{4\Ric_{\mu}}{||\mu||^2}} \nc{\ds}{\displaystyle}
\nc{\ben}{\begin{enumerate}} \nc{\een}{\end{enumerate}} \nc{\f}{\frac}
\nc{\lb}{[\cdot,\cdot]} \nc{\isn}{\tfrac{1}{||v||^2}}

\nc{\wt}{\widetilde}
\nc{\raw}{\rightarrow} \nc{\lraw}{\longrightarrow} \nc{\hqn}{\mathcal{H}_{q,n}}
\nc{\minimatrix}[4]{\left[\begin{smallmatrix} {#1} & {#2} \\ {#3} & {#4} \end{smallmatrix}\right]}
\nc{\twomatrix}[4]{\left[\begin{array}{cc} {#1} & {#2} \\ {#3} & {#4} \end{array} \right]}
\nc{\threematrix}[9]{\left[\begin{array}{ccc} {#1} & {#2} & {#3} \\ {#4} & {#5} & {#6}\\ {#7} & {#8} & {#9} \end{array} \right]}
\nc{\mut}{\tilde{\mu}} \nc{\mur}{{\mu_r}} \nc{\mutr}{{\tilde{\mu}_r}}

\nc{\alert}{\color{blue}}

\nc{\glgan}{\minimatrix{0}{0}{\star}{0}} \nc{\glgna}{\minimatrix{0}{\star}{0}{0}}  \nc{\glgnn}{\minimatrix{0}{0}{0}{\star}}  \nc{\glgaa}{\minimatrix{\star}{0}{0}{0}}
\nc{\Vaan}{{\left(\ag \wedge \ag\right)^* \otimes \ngo}} \nc{\Vann}{{\left(\ag \otimes \ngo \right)^* \otimes \ngo}} \nc{\Vnnn}{{\left(\ngo \wedge \ngo \right)^* \otimes \ngo}}

\nc{\ad}{\operatorname{ad}}  \nc{\Aut}{\operatorname{Aut}}   \nc{\Inn}{\operatorname{Inn}}   \nc{\Lie}{\operatorname{Lie}} \nc{\Ad}{\operatorname{Ad}} \nc{\Der}{\operatorname{Der}} \nc{\rad}{\operatorname{rad}} \nc{\kf}{\operatorname{B}}

\nc{\End}{\operatorname{End}} \nc{\rank}{\operatorname{rank}} \nc{\Ker}{\operatorname{Ker}} \nc{\tr}{\operatorname{tr}} \nc{\Sym}{\operatorname{Sym}} \nc{\diag}{\operatorname{diag}} \nc{\proy}{\operatorname{pr}} \nc{\Adj}{\operatorname{Adj}} \nc{\proj}{\operatorname{pr}} \nc{\Id}{{\operatorname{Id}}} \nc{\Span}{\operatorname{span}}

\nc{\Hess}{\operatorname{Hess}}  \nc{\dif}{\operatorname{d}} \nc{\sen}{\operatorname{sen}} \nc{\grad}{\operatorname{grad}} \nc{\Order}{\operatorname{O}} \nc{\divg}{\operatorname{div}}

\nc{\Iso}{\operatorname{Iso}} \nc{\Diff}{\operatorname{Diff}} \nc{\Rc}{\operatorname{Rc}} \nc{\Ricci}{\operatorname{Ric}}
\nc{\ric}{\operatorname{ric}} 
\nc{\Riem}{\operatorname{Rm}} \nc{\scal}{\operatorname{scal}} \nc{\scalm}{\operatorname{scal}^\star} \nc{\Riccim}{\operatorname{Ric}^{\star}} \nc{\tang}{\operatorname{T}} \nc{\vol}{\operatorname{vol}} \nc{\mcv}{\operatorname{H}} \nc{\inj}{\operatorname{inj}}
\nc{\isog}{\mathfrak{iso}}

\nc{\mm}{\operatorname{M}} \nc{\CH}{\operatorname{CH}} \nc{\Irr}{\operatorname{Irr}} \nc{\mcc}{\operatorname{mcc}} \nc{\Sb}{\mathcal{S}_\Beta} \nc{\mmm}{\operatorname{m}} 

\nc{\Symn}{{{\rm Sym}(n)}}
\nc{\Beta}{{\beta}}
\nc{\Alpha}{A}
\nc{\Vr}{V_{\Beta^+}^{r}}
\nc{\Vzero}{V_{\Beta^+}^{0}}
\nc{\Vnn}{V_{\Beta^+}^{\geq 0}}
\nc{\Vnnss}{U_{\Beta^+}^{\geq 0}}
\nc{\Vzeross}{U_{\Beta^+}^{0}}
\nc{\Vnnt}{V_{\tilde\Beta^+}^{\geq 0}}
\nc{\Betap}{ {\Beta + \Vert{\Beta}\Vert^2 \Id} }
\nc{\Ap}{ {A + \Vert{A}\Vert^2 \Id} }
\nc{\zero}{ {\backslash \{0\} } }
\nc{\normmm}{{\rm E}}
\nc{\ipp}{\la\,\cdot \,,\cdot\,\ra^*_g}
\nc{\ippk}{\la\,\cdot \,,\cdot\,\ra^*_{g_k}}
\nc{\ippi}{\la\,\cdot \,,\cdot\,\ra^*_{g_\infty}}
\nc{\ipnew}{\la \la \cdot , \cdot \ra\ra}
\nc{\der}{\mathfrak{der}}
\nc{\kfm}{\widetilde{\kf}} 
\nc{\KFm}{\widetilde{\mathcal{B}}}
\nc{\KF}{\mathcal{B}}
\nc{\II}{{\mathbb I}}
\nc{\spa}{\operatorname{span}}

\newtheorem{mainthm}{Theorem}[]
\newtheorem{maincor}[mainthm]{Corollary}
\newtheorem*{problem}{Problem}

\theoremstyle{plain}
\newtheorem{theorem}{Theorem}[section]
\newtheorem{proposition}[theorem]{Proposition}
\newtheorem{corollary}[theorem]{Corollary}
\newtheorem{lemma}[theorem]{Lemma}

\theoremstyle{definition}
\newtheorem{definition}[theorem]{Definition}
\newtheorem{notation}[theorem]{Notation}

\theoremstyle{remark}
\newtheorem{remark}[theorem]{Remark}

\newtheorem{example}[theorem]{Example}

\begin{document}
\begin{titlepage}

\title{Immortal homogeneous Ricci flows}

\author{Christoph B\"ohm}	
\address{University of M\"unster, Einsteinstra{\ss}e 62, 48149 M\"unster, Germany}
\email{cboehm@math.uni-muenster.de}

\author{Ramiro A.~ Lafuente} 
\address{University of M\"unster, Einsteinstra{\ss}e 62, 48149 M\"unster, Germany}
\email{lafuente@uni-muenster.de}

\thanks{The second author was supported by the Alexander von Humboldt Foundation.}

\begin{abstract}
We show that for an  immortal homogeneous Ricci flow solution any sequence of parabolic blow-downs
subconverges to a homogeneous expanding Ricci soliton. 
This is established by constructing a new Lyapunov function based on
curvature estimates which come from real geometric invariant theory.
\end{abstract}

\end{titlepage}


\maketitle
\setcounter{page}{1}
\setcounter{tocdepth}{0}

In 1982 Hamilton showed
that  on a compact $3$-dimensional manifold
 any metric of positive Ricci curvature can be deformed to the round metric 
  via volume normalized Ricci flow \cite{Ham82}. Similar convergence results were obtained afterwards in higher dimensions, see for instance
\cite{Ham86}, \cite{BW} and \cite{BS}.

About 20 years later, Perelman was able to describe new geometric quantities, which 
are monotone along Ricci flow solutions  \cite{Per1}, \cite{Per2}. 
 In the case of finite extinction time,
assuming a \mbox{Type-I} behavior of the evolved curvature tensors,
 it could then be shown
 that any essential blow-up sequence subconverges to a non-flat, gradient Ricci soliton   \cite{Nab10}, \cite{EMT}.
 
By contrast, for \emph{immortal} Ricci flow solutions, assuming a Type-III behavior of the evolved curvature tensors, subconvergence of an essential blow-down sequence to an expanding limit soliton is unknown in general. The low-dimensional situation is however
much better understood: in dimension two the existence of a
 locally homogeneous limit solution of constant curvature could be established \cite{Ham88}, \cite{Chow91}, \cite{JMS09}, and in dimension three subconvergence to an expanding  
 locally homogeneous limit soliton could be shown in several cases \cite{Lott}, \cite{Bam14}.

The main result of this paper is the following

\begin{mainthm}\label{mainconvergence}
For an immortal homogeneous Ricci flow solution
any sequence of  blow-downs subconverges to an expanding homogeneous Ricci soliton.
\end{mainthm}

Recall that a Ricci flow solution is called \emph{homogeneous}
if it is homogeneous at any time, and immortal if it exists for all positive times.
Theorem \ref{mainconvergence}
 was known for low-dimensional homogeneous spaces \cite{IJL06}, \cite{Lot07},
for nilpotent Lie groups \cite{nilRF} and for a special
class of solvable Lie groups \cite{Arroyo2013}. 
Notice that in the immortal case the limit soliton of a non-flat solution can be flat: see \cite{IJ92}.

For an immortal homogeneous Ricci flow solution $\big(M^n,(g(t))_{t \in [0,\infty)} \big)$ and 
any sequence of times $(s_k)_{k \in \NN}$  converging to infinity 
we set $g_k(t):=\tfrac{1}{s_k}\cdot g(s_k\cdot t)$. If the 
sequence $\big(M^n,(g_k(t))_{t \in [0,\infty)}\big)_{k \in \NN}$ of \emph{blow-downs} 
is non-collapsed, 
convergence means pointed smooth Chee\-ger-Gromov convergence to a space locally isometric to the stated limit;
otherwise convergence has to be understood in the sense of Riemannian groupoids: see  \mbox{Section \ref{sec:conv}} for details. Remarkably,
even in this case, the a priori only locally defined limit soliton can 
be extended to a globally homogeneous Ricci soliton by Theorem \ref{thm_lochomsolitons}.

A \emph{Ricci soliton} is a Ricci flow solution which is self-similar up to scaling. 
Expanding Ricci solitons can be characterized by the fact that they
emerge from a  metric $g$ (also called Ricci soliton) satisfying
$\ric(g) +\mathcal{L}_X g = - g$.
Here, $\mathcal{L}_X g$ denotes the Lie derivative of $g$
in the direction of a smooth vector field $X$ on $M^n$.
 
For expanding  homogeneous Ricci solitons there
exists already a well-developed structure theory: see \cite{Heber1998}, \cite{standard}, \cite{Nik11}, \cite{Jbl13b}, \cite{LafuenteLauret2014a}, \cite{LafuenteLauret2014b}, \cite{Jbl2015}, \cite{HePtrWyl}, \cite{GJ15}, \cite{JblPet14}, \cite{AL16}.
By  \cite{PW}, these solitons 
are either the product of a negative Einstein metric  
and a flat Euclidean factor, or  \emph{non-gradient} 
solitons. The latter provide the main difficulty for proving Theorem 1, since it seems that the monotone quantities described by Perelman cannot detect them.

Any homogeneous Ricci flow solution on a homogeneous space covered by $\RR^n$ is immortal \cite{scalar}, and these in fact  are  all known examples. Notice also
that already $\RR^3$ is not only a solvable Lie group
in uncountably many algebraically  distinct ways \cite{Bian1898}, but
 also a semisimple Lie group. By taking semi-direct products one obtains many algebraically 
 different Lie group structures on $\RR^n$.
Since  by \mbox{Theorem \ref{thm_noncollapsing}}
homogeneous Ricci flow solutions 
on  \mbox{$\RR^n$} are non-collapsed, and by Theorem \ref{thm_chgr} the limit is again a solution on $\RR^n$,
from Theorem \ref{mainconvergence} we  immediately deduce

\begin{maincor}\label{mainconvRn}
For  a homogeneous Ricci flow
solution on $\RR^n$,
any sequence of blow-downs subconverges in pointed Cheeger-Gromov topology to an
  expanding  homogeneous Ricci soliton on $\RR^n$.
\end{maincor}

The homogeneous structure will in general change in the limit: see Example \ref{ex_algcollapse}.
Let us also mention that Corollary \ref{mainconvRn} would cover all immortal homogeneous Ricci flow solutions, provided there is an affirmative answer to the following

\begin{problem}
Show that the universal covering space of an immortal homogeneous Ricci flow
solution is diffeomorphic to $\RR^n$.
\end{problem}

Since any homogeneous Einstein metric with negative Einstein constant yields
an immortal Ricci flow solution, the simply-connected case of the Alekseevskii conjecture for non-compact 
homogeneous Einstein spaces would follow \cite[7.57]{Bss}.

In some cases,
next to the  existence of a homogeneous limit soliton its uniqueness could be established 
 \cite{IJL06}, \cite{Lot07}, \cite{nilRF}, \cite{Jbl11}. 
Using Theorem \ref{mainconvergence}, in \cite{BL2} we generalize these uniqueness results
to a large class of solvable Lie groups containing all Einstein solvmanifolds and solvsolitons.
We  prove there that on a solvable Lie group of real type, the limit soliton does not depend on the initial metric. For Einstein solvmanifolds, the convergence can be improved to $C^\infty$, and as a consequence one gets a dynamical proof of a result of Jablonski and Gordon on the maximal symmetry of Einstein solvmanifolds \cite{GJ15}.

We turn now to the proof of 
\mbox{Theorem \ref{mainconvergence}}, assuming that the underlying homogeneous space is  a simply-connected Lie group $\GG$, and that the sequence of blow-downs is non-collapsed,
 that is  ${\rm inj}(M^n,g_k(1))\geq \delta>0$ for all $k \in \NN$.  
 For a homogeneous Ricci flow solution
\begin{eqnarray}
    \frac{\dd  g(t)}{\dd t}  = - 2 \ric(g(t)), \qquad g(0) = g_0, \label{eqn_ricciflow}
\end{eqnarray}
we have that $G \subset \Iso(M^n,g(t))$ for all $t$. As a consequence, 
 the Lie algebra $\ggo$ of $G$ can be identified with a Lie algebra of Killing fields on $M^n$, and the Lie bracket $\hml:\ggl \wedge \ggl \to \ggl$ is given by the usual Lie bracket  of smooth vector fields on $M^n$.

 By applying Uhlenbeck's trick of moving  frames, we pull-back $\hml$ and obtain a smooth curve in the space of brackets  $\Vg :=\Lambda^2(\ggl^*) \otimes \ggl$. 
Next, we break symmetry by choosing a left-invariant background   metric $\bar g$ on $\GG$, 
which gives rise to a scalar product $\ip$ on $\ggl$ and hence also on $\Vg$ and $\End(\ggl)$.
Then, we show  that the \emph{bracket flow}  
\[
    \frac{ \dd \mu(t)}{\dd t} = - \pi\left( \Ricci_{\mu(t)}\right) 
         \mu(t) \,, \qquad \mu(0) = \mu_0,
\]
on the orbit $\Ol :=\Gg \cdot \hml$ is equivalent to the homogeneous Ricci flow:
 see Theorem \ref{thm:Riccibracket}. Here,
the linear action of $\Gg$ on $\Vg$, defined in (\ref{def_muaction}), induces
the Lie algebra representation $\pi:T_{\Id} \Gg  \to \End(\Vg)$, defined in (\ref{eqn_defpib}).

The bracket flow was introduced by Lauret in \cite{homRF} for globally homogeneous spaces, where also the precise correspondence between homogeneous metrics and brackets was described: see Proposition \ref{prop:metricbracket}. Notice that the bracket of a homogeneous space determines its
Levi-Civita connection, and vice versa.

About ten years earlier, Lauret observed in \cite{Lau2003}
 that the (scale-invariant)
 \emph{moment map} $\mmm :\Vg\backslash \{0\} \to \Symg$ 
associated to the above $\Gg$-action on $\Vg$ is a natural part of the Ricci endomorphism 
\[
\Ricci_\mu =   
     \mm_\mu
      - \unm \kf_\mu 
   - S(\ad_\mu \mcv_\mu)= \Riccim_\mu  - S(\ad_\mu \mcv_\mu)
\] 
of a bracket $\mu \in \Vg$, since $\mm_\mu = \tfrac{1}{4}\cdot \mmm(\mu) \cdot \Vert \mu \Vert^2 $; 
 see (\ref{eqn_FormulaRicci}) and Lemma 
\ref{lem_formulabigmm}. The  \emph{modified
Ricci curvature} $\Riccim_\mu$, introduced in \cite{Heber1998},
has the following property:
If $g$ is the homogeneous metric corresponding to $\mu$, then
 $\Riccim(g)$ is orthogonal to the $\Aut(G)$-orbit of $g$ in the
 space  $\mca_G$ of left-invariant metrics on $\GG$. Recall here, that $\Aut(G)$
consists precisely of those diffeomorphisms of
$\GG$, which preserve  $\mca_G$.
As expected, it can be shown that the bracket flow and the \emph{unimodular} bracket flow 
are equivalent: see Corollary \ref{cor_uhl1}.

We turn now to  real geometric invariant theory.
 A first important fact is the existence of a Morse-type
stratification of $\Vg$ into finitely many strata, introduced in \cite{Krw1},
\cite{Ness} in the complex case. Each stratum is the unstable manifold
of a critical set of the negative gradient flow of the \emph{energy function}
$\normmm(\mu)=\Vert\mmm(\mu)\Vert^2$. Remarkably, these strata turn out to be 
$\Gg$-invariant. As a consequence,  the orbit $\Ol$ is  contained in a single 
stratum $\sca_\Beta$,  labeled by a self-adjoint endomorphism 
$\Beta \in \Symg$: see  \mbox{Theorem \ref{thm_stratifb}}.
Real geometric invariant theory is  well-understood thanks to \cite{RS90}, \cite{Mar2001}, \cite{HSS}, \cite{EbJbl09}, \cite{standard}, among others: see also  \cite{GIT} 
for a fully self-contained proof of many of these results in the case of linear actions.

The $\Og$-equivariance of the unimodular 
bracket flow allows us to introduce the \emph{gauged} unimodular bracket flow, which is 
again equivalent to the bracket flow, but lives
on the much smaller orbit $Q_\Beta \cdot \hml \subset \Ol$: see Corollary \ref{cor_gauge}.
The parabolic subgroup $Q_\Beta \subset \Gg$ is 
naturally associated to the stratum $\sca_\Beta$: 
see \mbox{Definition \ref{def_groups1b}}. This should be interpreted as 
new first integrals for the bracket flow 
coming from a refinement of Uhlenbeck's trick by choosing the moving orthonormal frames 
in a clever way.
The huge advantage of  the smaller orbit $Q_\Beta \cdot \hml$ is that
geometric invariant theory provides us with the new curvature estimate
\[
 0 \leq  \big\la \mm_\mu , \Betap \big\ra =
   \big\la \Riccim_\mu,\Betap \big\ra =  \big\la \Riccim_\mu,\Beta \big\ra + \scalm(\mu)\cdot \Vert \Beta \Vert^2\,.
\]
Here we used $ \la  \kf_\mu, \Betap \ra=0$: see Corollary \ref{cor_refinedbetaplus}.
If $\scalm(\mu) \leq 0$ (which holds along an immortal solution by Lemma \ref{lem_scalmev}) then Cauchy-Schwarz inequality yields   
\begin{eqnarray}
	 \Vert \Riccim_\mu\Vert \geq  \vert {\scalm (\mu)} \vert \cdot \Vert \Beta \Vert  \,,\label{eqn_estricbetaintro}
\end{eqnarray}
with equality if and only if 
$\Riccim_{\mu} = -\scalm(\mu) \cdot \Beta$ and $\mu \in \sca_\Beta$: see  Corollary \ref{cor_mainestimate}  (recall that $\tr \Beta = -1$ by Remark \ref{rem_git1}). We show in \mbox{Theorem \ref{thm_lochomsolitons}} that these conditions characterize expanding homogeneous Ricci solitons.




In Section \ref{sec:lyap} we introduce the \emph{$\Beta$-volume-normalized} modified scalar curvature 
\[
  F_\Beta (\mu) =(v_\Beta(\mu))^2 \cdot \scalm( \mu )\,,
\]
which is scale-invariant, and along a non-flat immortal solution $(\bar \mu(t))_{t \in [0,\infty)}$ to the gauged  unimodular bracket flow evolves by
\begin{equation*}
    \frac{ \dd \, F_\Beta(\bar \mu(t))}{\dd t}  =
     2 \cdot \vert F_\Beta(\bar \mu(t)) \vert \cdot \Big( 
     \tfrac{ \Vert { \Riccim_{\bar \mu(t)}  } \Vert ^2}{\vert \scalm({\bar \mu(t)})\vert} 
      -  \big\la \Riccim_{\bar \mu(t)}, \Beta \, \big\ra \Big).
\end{equation*}
Using again Cauchy-Schwarz, together with the above formula and \eqref{eqn_estricbetaintro}, we  prove

\begin{mainthm}\label{thm_mainlyap}
For a non-flat, immortal homogeneous Ricci flow $(M^n, (g(t))_{t\in [0,\infty)})$ there exists an associated flow of brackets $(\bar \mu(t))_{t\in [0,\infty)}$ in a stratum $\sca_\Beta$ along which $F_\Beta$ is non-decreasing. Moreover, $F_\Beta$ is strictly increasing, unless $(M^n \! ,g(0))$ is locally isometric to an expanding homogeneous Ricci soliton with $\Riccim_{\mu} = -\scalm(\mu) \cdot \Beta$.
\end{mainthm}

Let us mention here that the $\Beta$-volume $v_\Beta$ can be computed explicitly: if $\mu$ corresponds to the left-invariant metric $g$, we 
may write $g(\cdot, \cdot) = \bar g(P \cdot,  \cdot)$ in terms of the background metric $\bar g$,
with $P = h^t h > 0$, $h\in \Gg$ lower-triangular with eigenvalues $h_1, \ldots, h_n > 0$.
We can assume in addition, that $\Beta\in \Sym(\ggo)$ is diagonal,
 with eigenvalues $\Beta_1, \ldots, \Beta_n$.  Then, $v_\Beta$ is given by
\[
	v_\Beta(\mu) = h_1^{-\Beta_1} \cdot \,\cdots \,  \cdot h_n^{-\Beta_n}\,.
\]
For instance, 
for a semisimple $\GG$ we have $\Beta =-1/n \cdot \Id_\ggo$, and in this case $F_\Beta$ is
nothing but the volume-normalized scalar curvature functional. 



After Theorem \ref{thm_mainlyap}, the hope would be to show that the sequence $(\mu_k)_{k \in \NN}$ of brackets corresponding to
$(M^n,g_k(1))_{k \in \NN}$ subconverges to a limit bracket, which corresponds to an expanding soliton.  Since immortal homogeneous Ricci flow solutions are of Type-III  \cite{Bhm2014},
any sequence of blow-downs subconverges to an immortal homogeneous limit solution
by 
Hamilton's compactness theorem. 
Moreover, there exists always a sequence of spanning, appropriately rescaled, 
Killing vector fields along
$(M^n,g_k(1))_{k \in \NN}$ that subconverge to limit Killing vector fields \cite{Heber1998}.
However, it may now happen, that these limit Killing fields
do not span the tangent space of the limit manifold anymore, a phenomenon we
call {\it algebraic collapse}.
A first geometric consequence is that 
the dimension of the isometry group ``jumps'' in the limit.
On the algebraic side, algebraic collapse turns out
to be equivalent to unboundedness of the sequence $(\mu_k)_{k \in \NN}$ in $\Vg$: 
see  \mbox{Section \ref{sec_algcon}}.
Notice that
algebraic collapse occurs already on $\Sl(2,\RR)$: see Example \ref{ex_algcollapse}. 

The way around algebraic collapse is to argue in two steps: for a given sequence $(s_k)_{k \in \NN}$ we
first extract
a subsequence of blow-downs, such that the dimension of the limit isometry group is maximal.
In a second step, for such a limit Ricci flow
it can be shown that the sequence of corresponding brackets is bounded, and that the blow-downs subconverge in Cheeger-Gromov topology to an expanding homogeneous Ricci soliton.
The crucial results are provided in  \mbox{Theorem \ref{thm:Liebracketconvergence}},
Corollary \ref{cor:algcollapse} and  \mbox{Lemma \ref{lem_solitonbracket}}.
They relate the geometric convergence of Killing fields
to the convergence of abstract brackets, roughly saying that algebraically non-collapsed equivariant Cheeger-Gromov convergence implies convergence of the brackets.
Finally, since the limit of a limit is again a limit, Theorem \ref{mainconvergence} follows in this case.

In the general case,
if in the first step the sequence of blow-downs is geometrically collapsed,
 we cannot apply Hamilton's compactness theorem.
The way around this is to pull back the blow-downs to a disc of fixed radius in the tangent space, using the Type-III behavior: see Section \ref{sec:conv}.
In this way one obtains an incomplete, locally homogeneous
limit Ricci flow solution on that disc. 
We show in Appendix \ref{app_unique}, that even such incomplete
locally homogeneous Ricci flows are uniquely determined by their initial metric, which implies
that their local isometry group does not change over time,
and proceed as above. 

Finally, we would like to mention that the case of non-trivial isotropy 
group is considerably more difficult than the Lie group case, and that
also in the second step in the non-collapsed Lie group case discussed above one has to deal with that.

The article is organized as follows.
In Section \ref{sec_homricflow} we discuss the Ricci flow on locally homogeneous spaces, and in Section \ref{sec_BF} we introduce the bracket flow. 
The refinements  of Uhlenbeck's trick will be described in Section \ref{sec:uhlenbeck}.
In \mbox{Section \ref{sec_stratif}}, we collect important results in real geometric invariant theory.  
In \mbox{Section \ref{sec_strathom}} we prove special properties of the stratum of a homogeneous space.
The new curvature estimates for the Ricci curvature of
homogeneous spaces will be discussed in Section \ref{sec_newcurv}, and these will be used
in \mbox{Section \ref{sec:lyap}} to construct a scale-invariant Lyapunov function for the bracket flow, which
is constant precisely on expanding solitons, thus proving Theorem \ref{thm_mainlyap}. 
The proofs of Theorem \ref{mainconvergence} and Corollary \ref{mainconvRn} will be given in Section \ref{sec:conv}.
In \mbox{Section \ref{sec_algcon}}
we introduce equivariant Cheeger-Gromov convergence and discuss algebraic collapsing. 
In Appendix \ref{app_lochomog} we review the theory of locally homogeneous
spaces and prove algebraic structure results. Then in Appendix \ref{app_unique} we show uniqueness of locally homogeneous Ricci flows. In Appendix \ref{app_stratum} we determine the stratum corresponding to a Lie algebra in terms of the stratum of its nilradical. Finally, in Appendix \ref{app_chgr} we prove that a Cheeger-Gromov limit of homogeneous metrics on $\RR^n$ must be again a homogeneous metric on $\RR^n$.

\vs \noindent {\it Acknowledgements.} 
 It is a pleasure to thank Robert Bryant for sharing with us a beautiful proof of Lie's second theorem, and John Lott for his useful suggestions. We would also like to thank the referee for very helpful comments.


\section{Locally homogeneous Ricci flows}\label{sec_homricflow}

In this section we describe homogeneous Ricci flow solutions on 
locally homogeneous spaces. Unless otherwise stated, all Ricci flow solutions considered are \emph{maximal}, i.e.~ defined on a maximal time interval $[0,T)$. Our main result in this direction is the following

\begin{theorem}\label{thm:locRicflow}
Let $(M^n,g_0)$ be a locally homogeneous space, which is either complete or
diffeomorphic to a disk $D^n$. Then, there exists
a unique Ricci flow solution of locally homogeneous metrics on 
$M^n$ with initial metric $g_0$.
\end{theorem}

For complete initial metrics, existence and uniqueness follow from well-known general results on complete Ricci flow solutions with bounded geometry \cite{Kot}. 
In the incomplete case this follows from Theorem \ref{thm:gRicflow}, where
we also show that the homogeneous structure does not change.


\begin{definition}\label{def_lochom}
A connected Riemannian manifold  $(M^n, g)$ is called 
 a \emph{locally homogeneous space}, if for all  
 $p,q \in M^n$ there exists a local isometry $f:B_\epsilon(p)\to B_\epsilon(q)$ for some $\epsilon=\epsilon_{p,q} >0$ 
 with $B_\epsilon(p),B_\epsilon(q) \subset M^n$ being open distance balls in $M^n$. 
 A locally homogeneous space is called \emph{globally homogeneous}, or \emph{homogeneous} for short, if for all 
\mbox{$p,q \in M^n$} there exists an isometry $f:M^n \to M^n$ with $f(p)=q$. 
\end{definition}


Let $(M^n,g)$ be a locally homogeneous space and $p \in M^n$. 
Then by \cite{Nomizu1960} there exists a Lie algebra $\ggl$ of Killing vector fields defined on an open neighbourhood of $p$, which span $T_p M^n$.

\begin{definition}\label{def_transitive}
Let $(M^n, g)$ be a locally homogeneous space, which is either complete
or diffeomorphic to a disk $D^n$. Then $(M^n,g)$ is called a  \emph{$\ggl$-homogeneous space}, if there exists a Lie algebra $\ggl$ of Killing vector fields on $M^n$, 
which span the tangent space of $M$ at any point. Such a Lie algebra $\ggl$ is called \emph{transitive on $M^n$}. 
\end{definition}

We denote a $\ggl$-homogeneous space by $(M^n,g,\ggl)$ and
a \emph{pointed $\ggl$-homogeneous space} by  $(M^n,g,p,\ggl)$,
where $p \in M^n$.  In case we do not want to specify a 
homogeneous metric $g$ on $M^n$, we simply write $(M^n,\ggl)$ or $(M^n,p,\ggl)$.
Any globally homogeneous space is of course a $\ggl$-homogeneous space.
Moreover,  if $(M^n,g)$ is incomplete but diffeomorphic to a disc,
then this is also true by \cite[Thm.~1]{Nomizu1960}.

\begin{remark}
A compact hyperbolic manifold is locally homogeneous but not $\ggo$-ho\-mo\-geneous, since by Bochner's theorem a Riemannian manifold with negative Ricci curvature
cannot admit globally non-trivial Killing fields.
\end{remark}

The second fundamental Theorem of Lie yields on $(M^n,g,\ggl)$
a local action of the connected, simply-connected Lie group $\GG$ with
Lie algebra $\ggl$: see Theorem \ref{thm_Lie}.
That is, there exists an open
neighbourhood $U$ of $\{e\} \times M^n$ in $\GG \times M^n$ and a smooth local
multiplication map $\Phi:U \to M^n$, such that $\Phi(hg,m)=\Phi(h,\Phi(g,m))$ whenever 
this makes sense. Moreover, the local flows of right invariant vector fields
on $\GG$ and the local flows of the smooth Killing vector fields in $\ggl$ on $M^n$  
correspond to each other. As a consequence, $\GG$ acts locally transitively 
and isometrically on $(M^n,g)$. 
Let us mention, that if $M^n$ is simply connected and complete,
then $\GG$ acts transitively and isometrically on $M^n$: 
see \mbox{Theorem 7} in \cite{Nomizu1960}.

We now fix a point $p \in M^n$. If $(M^n,g)$ is globally homogeneous,
then 
\[
  H:=\{g \in \GG: g \cdot p=p\}
\] 
denotes the isotropy subgroup at $p$, which is
closed in $G$. We denote
by 
\[
  \hg:=T_e H=\{ X \in \ggl: X(p)=0\}
\] 
the isotropy subalgebra at $p$.  As is well-known, $M^n$ and $\GG/H$ are
diffeomorphic.

If $(M^n,g)$ is not globally homogeneous, we also define $\hg$ as above. It is clear that in this case $\hg$ is still a subalgebra of $\ggl$.
We denote by $H$ the corresponding connected Lie subgroup of $\GG$.
As a consequence, we are either in the globally homogeneous
case where $H$ is closed, or $H$ is a non-closed, connected Lie
subgroup of $\GG$.

In the next step we are going to describe the space of $\GG$-invariant metrics on a
$\ggl$-homogeneous space $(M^n,\ggl)$. These are by definition those metrics
for which $\GG$ acts by local isometries. 
To this end, we denote by 
\[
	\hml : \ggl \wedge \ggl \to \ggl \, ;  \qquad X \wedge Y \mapsto [X,Y],
\] 
the usual Lie bracket in $\ggl$ of smooth Killing vector fields on $M^n$. The \emph{Killing form} of $\ggl$ is the bilinear form defined by
\begin{equation}\label{eqn_defKF}
  \kf_{\hml}:\ggl \times \ggl \to \RR\,\,; \quad
  (X,Y)\mapsto \tr_\ggl \big(\ad_{\hml}\!X \cdot \ad_{\hml}\!Y\big)\,,
\end{equation}
where $(\ad_\hml \!X) ( \, \cdot \,) := \hml(X,\, \cdot \, )$ is the usual adjoint map.

\begin{definition}[Canonical decomposition]\label{def_can}
Let $(M^n,p,\ggl)$ be a pointed 
$\ggl$-homo\-geneous space and let $\mg$ denote the orthogonal complement of 
$\hg$ in $\ggl$ with 
respect to the Killing form $\kf_\hml$. Then $\ggl=\hg \oplus \mg$
is called the \emph{canonical reductive decomposition} of $\ggl$ (or \emph{canonical decomposition} for short).
\end{definition}

By Lemma \ref{lem_Kill} the Killing form $\kf_{\hml}$ is negative definite
on $\hg$, thus $\mg$ is well-defined. 
Since $\kf_\hml$ is invariant under automorphisms of
$\ggo$, and since  $\Ad(h) (\hg) \subset \hg$ 
for $h \in H$, the canonical complement $\mg$ is also $\Ad(H)$-invariant.
It follows that the canonical decomposition is indeed reductive: $[\hg,\mg] \subset \mg$.


\begin{lemma}\label{lem_correspGinvmetrics}
Let $(M^n,p,\ggl)$ be a $\ggl$-homogeneous space.
Then, there is a one-to-one correspondence between the set of $\GG$-invariant metrics on $M^n$ and
the set of $\Ad(H)$-invariant scalar products on $\mg$.
\end{lemma}

\begin{proof}
For globally homogeneous spaces the result is well-known.
For locally homogeneous spaces $M^n$ it is also clear, that a $\GG$-invariant metric  induces
a scalar product on $T_pM^n$, which is invariant under the
local action of the isotropy group $H$ on $T_pM^n$. Vice versa, it is not 
difficult to show, that any such scalar product can be extended to 
a $\GG$-invariant metric on $M^n$. Here we are using that
 locally homogeneous metrics are real analytic by \cite{BLS}. As a consequence,
 one can extend then the given scalar product to larger and larger 
coordinate balls in the disk $M^n=D^n$.
Finally, under the identification $\mg \to T_pM^n$, $X \mapsto X(p)$,
the local action of the isotropy group $H$ on $T_pM^n$ becomes the adjoint action of 
$H$ on the $\Ad(H)$-invariant complement $\mg$ as in the globally homogeneous case.
This shows the claim.
\end{proof}

The set $P(\mg)^{\Ad(H)}$ of ${\rm Ad}(H)$-invariant scalar products on $\mg$ is an open cone in the finite-dimensional vector space  $S^2(\mg)^{\Ad(H)}$ of 
${\rm Ad}(H)$-invariant symmetric bilinear forms on $\mg$. 

\begin{notation}\label{not_ipg}
For a $G$-invariant metric $g$ on a $\ggo$-homogeneous space $(M^n,p,\ggo)$ we denote by $\ip_g \in P(\mg)^{\Ad(H)}$ the corresponding scalar product on $\mg$.
\end{notation}

Next, let $\GHm$ be the centralizer of $\Ad(H)\vert_\mg \subset \Om$ in $\Gm$.
To describe $\GHm$ explicitly, 
we denote by $K:=\overline{\Ad(H)\vert_\mg}$ the closure of
$\Ad(H)\vert_\mg$ in $\Om$. Then $K$ is compact, hence 
there exists a unique decomposition $\mg=\ppm_1 \oplus \cdots \oplus \ppm_k$
of $\mg$ into $K$-isotypical summands: see \cite{BTD}, II, Proposition 6.9. 
Each of the isotypical summands $\ppm_i$ is itself a direct sum of $n_i\geq 1$ equivalent 
$K$-irreducible summands, say of real dimension $d_i$, $1 \leq i \leq k$: notice that the latter decomposition is not unique if $n_i>1$. The irreducible
summands of $\ppm_i$ are real representations of real, complex or quaternionic type: see \cite[Ch.~II, Thm.~6.7]{BTD}. In each of the cases we set $\KK_i:=\RR, \CC$ or $\HH$, respectively.
Then by Schur's Lemma we have
\begin{equation}\label{eqn_decompGHm}
  \GHm \simeq \Gl(n_1,\KK_1) \times \cdots  \times \Gl(n_k,\KK_k)\,,
\end{equation}
where the embedding of $\Gl(n_i,\KK_i)$ into $\Gl(\ppm_i)$ is induced by the decomposition of $\ppm_i$ into $K$-irreducible summands, and is  
described explicitly in \cite[pp.~100--101]{Bhm04}.

There is a natural left-action of $\GHm$ on the set $P(\mg)^{\Ad(H)}$, given by
\begin{equation}\label{eqn_actionip}
	h \cdot \ip := \la h^{-1} \, \cdot , h^{-1} \, \cdot \ra.
\end{equation}
In the following we
consider each $h \in \GHm$ also as an element in $\Gg$ by extending it trivially on
$\hg$: that is, $h|_\hg = \Id_\hg$ (see \eqref{eqn_Gminclusion}).

\begin{lemma}\label{lem_autisometric}
Suppose that for $h\in \GHm$ its trivial extension to $\Gg$ is an automorphism of the Lie algebra $\ggo$. Then the $G$-invariant metrics $g$, $g'$ on $M^n$, corresponding 
to $\ip$ and $\ip' := h\cdot \ip$, respectively, are locally isometric.
\end{lemma}

\begin{proof} Let $\psi \in \Aut(\ggo)$ denote the trivial extension of $h$.
Since $\GG$ is simply-connected, there exists a unique $\varphi \in \Aut(\GG)$ with $(d\varphi)_e= \psi$. Using that $\psi$ fixes $\hg$, it follows that $\varphi$ fixes the identity component $H_0$ of $H$.

In the globally homogeneous case, $\tilde M^n=G/H_0$ is the universal covering
space of $M^n=G/H$. The map $\varphi$ induces a
diffeomorphism of $\tilde M^n$ fixing $\tilde p=eH_0$. Thus,
the $G$-invariant metrics $\tilde g$ and $\tilde g'$ on $\tilde M^n$,
corresponding respectively to  $\ip$ and $\ip' = h\cdot \ip$, are isometric.
The claim follows in this case.

In the locally homogeneous incomplete case, $H$ is connected but not closed in $\GG$.
As a consequence the quotient space $G/H$ is not a smooth manifold anymore. However $H$ is
still an immersed submanifold of $G$ and consequently there exists an open neighbourhood
$U_H$ of $e$ in $H$, which is a submanifold of $G$. Exactly as in the globally homogeneous
case, using the normal exponential map of $H$ 
one can construct a section $S \subset G$, such that an open neighbourhood
$U_G$ of $e$ in $G$ is diffeomorphic to $S \times U_H$ and $H$ acts locally 
only on the second factor. Therefore we may build a local quotient $S=U_G/U_H$.
As in the globally homogeneous case it follows that smooth functions on
$S$ are in one-to-one correspondence with smooth functions on $U_G$ which
are constant on local $H$-orbits. Precisely as in the globally homogeneous case 
it follows, that $\varphi$ induces a local diffemorphism of $M^n$ fixing the point $p$.
This shows the claim.
\end{proof}

In the globally homogeneous, simply-connected case, the metrics in Lemma \ref{lem_autisometric} are of course isometric. Two non-isometric flat metrics on torii show that in general this is not the case. It is not hard to construct examples
with $\pi_1(M)$ finite.

We turn now our attention to the Ricci tensor $\ric (g)_p$ of a $\GG$-invariant metric $g$ on $M^n$, which we can consider as an element of $S^2(\mg)^{\Ad(H)}$,
since it is invariant under the local isometries in $H$. In what follows we recall a formula for the 
$\ip_g$-self-adjoint Ricci endomorphism $\Ricci^g:\mg \to \mg$ associated to $\ric (g)_p$:
\[
   \ric(g)_p \big( X(p),Y(p) \big) =  \big\langle {\Ricci^g}  X, Y \big\rangle_g,
\]
for all $X,Y \in \mg$. See \cite[7.38]{Bss} and \cite[\S 2]{LafuenteLauret2014b} for further details. Notice
that $\Ricci^g$ is $\Ad(H)|_\mg$-equi\-variant.

 Given the canonical decomposition $\ggl = \hg \oplus \mg$, the restriction of the Lie bracket of vector fields $\hml$  to $\mg$ decomposes into two maps $\hml = \mu_\hg^\ggo + \ml$
\begin{equation}\label{eqn_defbracket}
	\mu_\hg^\ggo := \lb_\hg : \mg \wedge \mg \to \hg, \qquad \ml := \lb_\mg : \mg \wedge \mg \to \mg\,.
\end{equation}
Notice that $\ml \in V(\mg) = \Lambda^2(\mg^*) \otimes \mg$, although it is not necessarily a Lie bracket since the Jacobi condition may fail to hold. For $X\in \mg$ 
we denote by
\[
	\ad_\ml X := \proj_\mg \circ (\ad_\hml X)|_\mg : \mg \to \mg
\]
the corresponding adjoint map.
 Moreover, we denote by $\kf^g\in \Symm$ the self-adjoint endomorphism 
of $\big(\mg,\ip_g\big)$ defined by the identity
\begin{equation}\label{eqn_defKFend}
 	\big\langle {\kf^g} X, Y \big\rangle_g  = \kf_{\hml}(X,Y)
\end{equation}
for $X, Y\in \mg$. 
The self-adjoint endomorphism $\mm^g:\mg \to \mg$ is defined by
\begin{equation}
   \big \la \mm^g X, X\big\ra_g =
    -\frac12 \sum_{i,j} \big\langle \ml(X, E_i), E_j \big\rangle^2_g
     + \frac14 \sum_{i,j} \big\langle \ml(E_i,E_j), X \big\rangle ^2_g\,,\label{eqn_formulamm} 
\end{equation}
where $\{ E_1,\ldots,E_n \}$ is any orthonormal basis  of $\big(\mg,\ip_g\big)$. 

Consider the linear map 
$T:\ggo \to \RR,\,\, X \mapsto \tr_{\ggl}(\ad_{\hml} X)$. Observe that $T(\hg) = 0$, since $\overline{\Ad(H)} \subset \Gg$ is compact. Using that $\hml(\hg,\mg) \subset \mg$ we also get $T(X) = \tr_\mg (\ad_\ml X)$ for $X\in \mg$. The \emph{mean curvature vector}  $\mcv^g \in \mg$ is defined by
 \begin{equation}\label{eqn_defmcv}
  	\big\la \! \mcv^g, X  \big\ra_g = \tr_{\ggo} (\ad_\hml X),
 \end{equation}
for all $X\in \mg$.  
Notice that $\mcv^g=0$ if and only if $\ggo$ is \emph{unimodular}.

 Finally, if $A^t$ denotes the $\ip_g$-transpose of $A$, we have the projection
\[
  S^g: {\rm End}(\mg) \to \Symm \,\,;\,\,\,A \mapsto \unm\left(A + A^t\right),
\] 

\begin{lemma}\label{lem_Ricfor}
Let $(M^n,g,p,\ggl)$ be a pointed $\ggl$-homogeneous space.
Then we have
\begin{equation}
    \Ricci^g =\mm^g - \unm \kf^g - S^g\big(\ad_\ml (\mcv^g) \big)\,.
    \label{eqn_Riccimg}
\end{equation}
\end{lemma}


It is worth pointing out that $\Ricci^g$ can be computed from
the canonical decomposition $\ggl=\hg \oplus \mg$,
the Lie bracket $\hml$,
and the scalar product $\ip_g$ on $\mg$.

\begin{theorem}\label{thm:gRicflow}
Let $(M^n,p,\ggl)$ be a pointed $\ggl$-ho\-mo\-geneous space.
Then for any $\GG$-in\-va\-riant initial metric on $M^n$ there exists a Ricci flow solution of $\GG$-invariant metrics, which is unique among all locally homogeneous Ricci flow solutions.
\end{theorem}

\begin{proof}
For a pointed $\ggl$-ho\-mo\-geneous space $(M^n,p,\ggl)$ 
and a $\GG$-homo\-geneous initial metric $g_0$ the Ricci flow equation 
(\ref{eqn_ricciflow}) is equivalent to
\begin{eqnarray}
    \frac{d}{d t}\ip_{g(t)}
    = -\big\la \! \Ricci^{g(t)}\,\cdot \,,\cdot \, \big\ra_{g(t)}
     -\big\la\,\cdot \,,  \Ricci^{g(t)}\cdot \, \big\ra_{g(t)} \,.\label{eqn_ricciflowm}
\end{eqnarray}
By Lemma \ref{lem_Ricfor} this is
an ordinary differential equation on  $P(\mg)^{\Ad(H)}$. 
This shows the existence part.
For uniqueness we refer the reader to Appendix \ref{app_unique}.
\end{proof}

\begin{remark}\label{rem:generalizedRF}
Notice that formula (\ref{eqn_Riccimg}) is defined for
any scalar product $\ip'$ on $\mg$ (not necessarily $\Ad(H)$-invariant). 
Consequently, (\ref{eqn_ricciflowm}) can be considered as an ordinary differential equation on the set  of scalar products $P(\mg)$ on $\mg$.
Then, as we have shown above, the set $P(\mg)^{\Ad(H)}$ is invariant
under this generalized Ricci flow equation. 
\end{remark}


\section{The bracket flow}\label{sec_BF}

Our aim in this section is to describe the setting of varying brackets instead of scalar products. This gives rise to an equation in the space of brackets which is equivalent to the Ricci flow, the so called \emph{bracket flow}. The bracket flow was introduced by Guzhvina 
for nilpotent Lie groups \cite{Guz}, and 
by Lauret in general for globally homogeneous spaces \cite{homRF}. We will show in Theorem \ref{thm:Riccibracket} that the equivalence between both flows holds also in the locally homogeneous case.

Let $(M^n,p,\ggl)$ be a $\ggl$-homogeneous space (Def.~ \ref{def_transitive}), with canonical reductive decomposition $\ggl = \hg \oplus \mg$.

\begin{notation}
We fix once and for all a $G$-invariant background metric $\bar g$ on $M$, corresponding to an $\Ad(H)$-invariant scalar product $\ip$ on $\mg$, and
denote by $\Om = \Or(\mg,\ip)$ the corresponding orthogonal group.
\end{notation}

Recall that we denote by $\hml$ the usual Lie bracket of vector fields in $\ggl$, and by $\ml$ its restriction to $\mg$.
Consider the real vector space
\[
  \Vg:=\Lambda^2(\ggo^*)\otimes \ggo,
\] 
called the \emph{space of brackets} on $\ggo$. The group $\Gl(\ggo)$ acts linearly on $\Vg$ by 
\begin{eqnarray}
  (h\cdot \mu)(\, \cdot \,,\cdot \,):=h \, \mu(h^{-1}\,\cdot ,h^{-1}\, \cdot )\,,
    \label{def_muaction}
\end{eqnarray}
where $h\in \Gg$, $\mu \in \Vg$. The canonical decomposition $\ggo = \hg \oplus \mg$ induces a natural inclusion 
\begin{equation}\label{eqn_Gminclusion}
	\Gm \simeq \minimatrix{\Id_\hg}{}{}{\Gm} \subset \Gg\,.
\end{equation}
Under that identification, the group $\Gm$ also acts on $\Vg$ via \eqref{def_muaction}.

Recall that $\GHm$ denotes the centralizer of
$\Ad(H)\vert_\mg $ in $\Gm$. We set 
\begin{equation}\label{eqn_defOHm}
	\OHm:= \GHm \cap \Om\,. 
\end{equation}
Consider the following orbit associated to the homogeneous space $(M^n, p, \ggo)$: 
\begin{eqnarray}
  \OlH:=\GHm \cdot \hml\, \, \subset \, \Vg.
\end{eqnarray}

\begin{definition}\label{def_bracket}
Given a $G$-invariant metric $g$ on $(M^n,p,\ggo)$ determined by the scalar product $\ip_g $ on $\mg$, we may write $\ip_g = \la h \, \cdot , h \, \cdot \ra$ for $h\in \GHm$. Then, the \emph{bracket associated with $g$} is defined by 
\[
	\mu := h \cdot \hml \in \OlH \subset \Vg.
\]
Here we are thinking $h\in \Gg$ using \eqref{eqn_Gminclusion}.
\end{definition}

Notice that in the previous definition the map $h$ is only well-defined up to left multiplication by $\OHm$. Thus, $\mu$ is only well-defined up to the action of $\OHm$ on brackets. 
Conversely, given a bracket $\mu \in \OlH$, after chosing $h\in \GHm$ so that $\mu = h \cdot \hml$, we can associate to $\mu$ the unique $G$-invariant metric $g$ on $M$ determined by the scalar product $\ip_g = \la h \,\cdot, h\, \cdot \ra \in P(\mg)^{\Ad(H)}$. Also in 
this construction there is an ambiguity for choosing  $h$. Namely, for any $\phi \in \Aut^H_\mg(\ggo)$ we have $\varphi\cdot \hml = \hml$, thus $h$ could be replaced by 
 $h \phi$.  Here, 
\[
	\Aut^H_\mg(\ggo) = \Aut(\ggo) \cap \GHm\, ,
\]
with $\Aut(\ggo) = \{ h\in \Gg : h \cdot \hml = \hml \}$ and $\GHm \subset \Gg$ as in \eqref{eqn_Gminclusion}. By  Lemma \ref{lem_autisometric},
all $G$-invariant metrics determined by scalar products in $\{ \la  \phi \, \cdot, \phi \, \cdot \ra_g : \phi \in \Aut^H_\mg(\hml) \}$, are locally isometric to $g$. 
Thus, all metrics associated to a given bracket $\mu \in \OlH$ are equivariantly locally isometric. Finally, notice that for  $k\in \OHm$, the set of metrics associated to $k\cdot \mu$ coincides with those associated to $\mu$. 

The following is now clear from the previous discussion.

\begin{proposition}\label{prop:metricbracket}
Let $(M^n, p, \ggo)$ be a $\ggo$-homogeneous case. Then, there is a one-to-one correspondence between $\Aut^H_\mg(\ggo)$-orbits of $G$-invariant metrics on $M^n$ and 
$\OHm$-orbits of brackets in $\OlH$.
\end{proposition}

In the above proposition, the action on scalar products (and hence on $G$-invariant metrics) is defined in \eqref{eqn_actionip}, and the action on brackets in \eqref{def_muaction}. Notice that $\Aut^H_\mg(\ggo)$ might be non-compact, whereas
$\OHm$ is always compact. In this way, we have replaced a non-compact \emph{gauge group}
by a compact one, by breaking symmetry.


We now generalize the definition of $\Ricci^g=\Ricci_\hml^g$ 
in Lemma $\ref{lem_Ricfor}$ by allowing the bracket to change. 
For a bracket $\mu$ consider its restriction to $\mg$, $\mu|_{\mg\wedge \mg} : \mg \wedge \mg \to \ggo$, which in turn decomposes as
\begin{equation}\label{eqn_muhgmg}
	\mu \,|_{\mg\wedge \mg} = \mu_\hg + \mumg, \qquad  \mu_\hg : \mg \wedge \mg \to \hg, \qquad \mumg : \mg \wedge \mg \to \mg.
\end{equation}

\begin{definition}
For a pointed $\ggl$-homogeneous space $(M^n,p,\ggl)$ 
with canonical decomposition $\ggl = \hg \oplus \mg$ 
and Lie bracket $\hml$, let $\mu \in \OlH$ and
$\ip_g$ be an arbitrary scalar product on $\mg$. Then we set
\begin{equation}\label{eqn_FormulaRiccig}
    \Ricci_\mu^g = \mm_\mumg^g - \unm \kf_\mu^g - S^g(\ad_\mumg \mcv_\mu^g)\,.
\end{equation}
For the background metric $\bar g$ we drop the super script ${}^{\bar g}$ and  write simply
\begin{equation}\label{eqn_FormulaRicci}
    \Ricci_\mu = \mm_\mumg - \unm \kf_\mu - S(\ad_\mumg \mcv_\mu)\,.
\end{equation}
\end{definition}

In \eqref{eqn_FormulaRiccig}, \eqref{eqn_FormulaRicci} the terms on the right hand side are computed as explained in Section \ref{sec_homricflow} (see \eqref{eqn_defKF}, \eqref{eqn_defKFend}, \eqref{eqn_formulamm}, \eqref{eqn_defmcv}), by replacing $\hml$ by $\mu$, and $\ml$ by $\mumg$.


It is natural to ask how does the Ricci flow \eqref{eqn_ricciflowm} look like on the space of brackets. The following ordinary differential equation on $\OlH$ is called
the \emph{bracket flow}:
\begin{equation}\label{eqn_BracketFlow}
    \frac{ \dd \mu}{\dd t} = - \pi\left( \Ricci_{\mu}\right) 
         \mu \,, \qquad \mu(0) = \mu_0 \in \OlH.
\end{equation}
Here, the Lie algebra representation $\pi : \glg(\ggo) \to \End(\Vg)$, defined by
\begin{eqnarray}\label{eqn_defpib}
  (\pi(A)\mu)(X,Y):= A\, \mu(X,Y)-\mu(AX,Y)-\mu(X,AY),
\end{eqnarray}
is the derivative of the action \eqref{def_muaction},
where  $A\in \End(\ggo)$ and  $X,Y \in \ggo$. As in \eqref{eqn_Gminclusion}, we identify 
\begin{equation}\label{eqn_glgminclusion}
	\glg(\mg) \simeq \minimatrix{0}{}{}{\glg(\mg)} \subset \glg(\ggo),
\end{equation}
and denote also by $\pi : \glg(\mg) \to \End(\Vg)$ the Lie algebra representation of $\glg(\mg)$ obtained by the above inclusion and \eqref{eqn_defpib}.

The main result of this section is the following

\begin{theorem}\label{thm:Riccibracket}
Let $(M^n, p, \ggo)$ be a $\ggo$-homogeneous space. Then, the Ricci flow of 
$\GG$-invariant metrics on $M^n$ and the bracket flow on $\OlH$ are equivalent.
\end{theorem}

For a given  $G$-invariant initial metric $g_0$ let 
$(g(t))_{t\in [0,T)}$ denote the Ricci flow solution from Theorem \ref{thm:gRicflow}.
Then by Proposition \ref{prop:metricbracket} we can associate to $g_0$ 
an initial bracket $\mu_0\in \Vg$. Theorem \ref{thm:Riccibracket}
says then, that for this initial bracket there exists on the same time interval
a maximal solution $(\mu(t))_{t\in [0,T)}$ to the bracket flow,
such that $\mu(t)$ is a bracket associated to the metric $g(t)$
for each $t\in [0,T)$. In the same manner we can associate to any bracket flow
solution on $\OlH$ a Ricci flow solution of $\GG$-invariant metrics.
In fact, it follows from the proof of Theorem \ref{thm:Riccibracket} that there exists a smooth curve $(h(t))_{t\in [0,T)} \subset \GHm$ such that $\mu(t) = h(t)\cdot \hml$ and $\ip_{g(t)} = \la h(t) \, \cdot, h(t) \, \cdot\ra$. Most importantly, we have that
\[
	\ric (g(t))_p (X, Y) = \big\la   \Ricci_{\mu(t)} (h(t) X), \, h(t) Y  \big\ra,
\]
for all $X, Y \in \mg$.

\begin{proof}[Proof of Theorem \ref{thm:Riccibracket}]
By definition, it is clear that $\Ricci_\mu$ depends smoothly on $\mu \in \OlH$.
Note also that 
$\pi(\Ricci_\mu)\mu \in T_\mu \OlH$, since $\Ricci_\mu$ is $\Ad(H)$-equivariant by Corollary \ref{cor_AdHequiv}. As a consequence the bracket flow is a smooth ordinary differential equation on  $\OlH$.

Let $(g(t))_{[0,T)}$ be the $\GG$-invariant Ricci flow solution  with $g(0)=g_0$ given by Theorem \ref{thm:gRicflow},
where we consider $g(t)$ as an $\Ad(H)$-invariant scalar product $\ip_{g(t)}$
on $\mg$.
Recall that the homogeneous Ricci flow equation is equivalent to (\ref{eqn_ricciflowm}).
Let now $h(t) \in \Gl(\mg)$ solve the linear equation
\begin{eqnarray}
    h'(t) = - h(t) \cdot \Ricci^{g(t)}, \qquad h(0) = \Id_{\mg}\label{eqn_defh}
\end{eqnarray}
on $[0,T)$. It then follows from (\ref{eqn_defh}) by differentiating, that
\[
 \tilde g(t)(\,\cdot  \,,\,\cdot \,):=g_0(h(t)\, \cdot\, ,h(t)\, \cdot\,)
\]
satisfies (\ref{eqn_ricciflowm}) as well. Thus $\tilde g(t)=g(t)$. Notice that
\begin{eqnarray}
 h(t):  \big( \mg,\ip_{g(t)} \big) \to \big( \mg,\ip_{g_0} \big) \label{eqn_interh}
\end{eqnarray}
is an isometry. 

Next, we extend $h(t)$ to an endomorphism 
$ h(t):\ggl\to \ggl$ as in \eqref{eqn_Gminclusion}. Thus
\begin{eqnarray}
   h(t):(\ggl,\hml) \to (\ggl,  \mu(t):= h(t) \cdot \hml) \label{eqn_interhh}
\end{eqnarray}
is an isomorphism of Lie algebras, respecting the canonical decomposition
$\ggl = \hg \oplus \mg$. 
 By (\ref{eqn_interh}) and (\ref{eqn_interhh}) we are in a position to apply 
Lemma \ref{lem_Ricciintertw} and deduce that
\begin{eqnarray}
   \Ricci^{g_0}_{\mu(t)} = h(t)\, \Ricci_{\hml}^{g(t)} \, h(t)^{-1} ,\label{eqn_Ricinter}
\end{eqnarray}
where $\mu(t):=h(t) \cdot \hml$. Differentiating $\mu(t)=h(t) \,  \hml(h(t)^{-1}\,\cdot, 
h(t)^{-1} \,\cdot )$ yields now
\begin{eqnarray*}
  \mu'(t)
    =  -\pi \big( h(t)\Ricci_\hml^{g(t)}h(t)^{-1} \big) \, \mu(t)
    =  - \pi \big( \Ricci_{\mu(t)}^{g_0}) \mu(t)\,,
\end{eqnarray*}
where we have used (\ref{eqn_defh}) and (\ref{eqn_Ricinter}).
Finally writing $\ip_{g_0}=\langle h_0 \,\cdot \,, h_0 \,\cdot \rangle$
for some $h_0 \in \GHm$, we 
set $\tilde \mu(t):=h_0 \cdot \mu(t)$. Then precisely as above it follows that
$\tilde \mu(t)$ is a solution to the bracket flow (\ref{eqn_BracketFlow})
with $\tilde \mu(0)=h_0 \cdot \hml$.

Vice versa,
let $(\tilde \mu(t))_{t \in [0,T)}$ denote a solution to the bracket flow
with initial metric $\tilde \mu(0)=h_0 \cdot \hml$ for some $h_0 \in \GHm$. 
We set $\ip_{g_0}:=\langle h_0 \,\cdot \,, h_0 \,\cdot \rangle$ and
obtain as above a solution $\mu(t):=h_0^{-1}\cdot \tilde \mu(t)$ of 
\begin{eqnarray}
  \mu'(t)
    =  - \pi \big( {\Ricci_{\mu(t)}^{g_0} } \big) \, \mu(t)\label{eqn_Riccig0}
\end{eqnarray}
with $\mu(0)=\hml$.
Next, we let $h(t) \in \Gl(\mg)$ solve the linear equation
\begin{eqnarray}
    h'(t) = -  \Ricci_{\mu(t)}^{g_0} \cdot \, h(t), \qquad h(0) = \Id_{\mg}
    \label{eqn_defbh}
\end{eqnarray}
on $[0,T)$ and set $\mu(t):=h^{-1}(t) \tilde \mu(t)$. Differentiation yields
$\mu(t) \equiv \hml$ for all $t \in [0,T)$: see (\ref{eqn_glnequi}). Moreover, setting
$\ip_{g(t)}=\langle h(t) \,\cdot \,, h(t) \,\cdot \rangle_{g_0}$
and using the intertwining identity (\ref{eqn_Ricinter}), we see that
$h(t)$ satisfies (\ref{eqn_defh}). This then clearly shows that
$(\ip_{g(t)})_{t \in [0,T)}$ is a solution of (\ref{eqn_ricciflowm})
with $\ip_{g(0)}=\ip_{g_0}$. Now $\ip_{g_0}$ is an $\Ad(H)$-invariant
scalar product on $\mg$, since $h_0 \in \GHm$. By Remark \ref{rem:generalizedRF}
it follows, that $(g(t))_{t \in [0,T)}$ is a $\GG$-invariant
Ricci flow solution with $g(0)=g_0$. This shows the claim.
\end{proof}

Let $(G_1/ H_1, g_1)$, $(G_2/ H_2, g_2)$ be two globally homogeneous spaces, with $G_i$ simply-connected and $H_i \subset$ closed and connected, $i=1,2$, and suppose that $\varphi : G_1/H_1 \to G_2/H_2$ is an equivariant isometry induced by the Lie group isomorphism $\hat \varphi : G_1 \to G_2$. Then, $\hat \psi :=d \hat \varphi|_e : \ggo_1 \to \ggo_2$ is a Lie algebra isomorphism respecting the canonical decompositions and inducing an orthogonal map $\psi : \big(\mg_1, \ip_{g_1} \big) \to \big(\mg_2, \ip_{g_2} \big)$.

 
The following functorial property is a kind of converse, which
holds also in the locally homogeneous case.
It follows immediately from the definition of $\Ricci^g_\mu$.

\begin{lemma}\label{lem_Ricciintertw}
Let $(M^n, p, \ggo)$ be a $\ggo$-homogeneous space, and consider the $G$-invariant metrics $g_1, g_2$ on $M^n$ and brackets $\mu_1, \mu_2 \in \OlH$. Suppose that $u\in \GHm \subset \Gg$ is such that $u: (\ggo,\mu_1) \to (\ggo, \mu_2)$ is a Lie algebra isomorphism, and $u: \big( \mg_1, \ip_{g_1}\big) \to \big( \mg_2, \ip_{g_2}\big)$ is an isometry. 
Then,
\begin{eqnarray*}
    \Ricci_{\mu_1}^{g_1} &=& u^{-1} \cdot \Ricci_{\mu_2}^{g_2} \cdot \,\,u\,.
\end{eqnarray*}
\end{lemma}


Notice that an analogous equation holds for any of the three summands of $\Ricci_\mu$ in formula \eqref{eqn_FormulaRicci}. Since for any bracket $\mu \in \OlH$ the maps in $\Ad(H)$ are automorphisms of $(\ggo,\mu)$ acting orthogonally on $\big(\mg, \ip\big)$,  an immediate consequence of the above lemma is the following

\begin{corollary}\label{cor_AdHequiv}
Let $(M^n, g,p,\ggo)$ be a $\ggo$-homogeneous space with bracket $\mu \in \OlH$. Then, the endomorphisms $\Ricci_\mu$, $\mm_\mumg$, $\kf_\mu$, $\ad_\mumg \! \mcv_\mu$ are $\Ad(H)$-equivariant, and the mean curvature vector $\mcv_\mu$ is $\Ad(H)$-invariant. 
\end{corollary}


\section{The Uhlenbeck trick}\label{sec:uhlenbeck}

When going from the homogeneous Ricci flow to the bracket flow on
the orbit $\OlH = \GHm \cdot \hml \subset \Vg$ in the previous
section, we used in \eqref{eqn_defh} what is known as Uhlenbeck's trick of moving frames.
In this section we show that this can be refined to obtain an equivalent flow of brackets which is better adapted to our purposes. 
The key idea is to use the compact gauge group $\OHm$ defined in \eqref{eqn_defOHm}, thus exploiting the natural $\OHm$-equivariance present in the bracket definition and in the bracket flow equation.
We would like to mention that the following results are by no means special 
to the bracket flow. They hold more generally for any reasonably well-behaved flow, which is equivariant with respect to a compact gauge group.

Regarding equivariance, notice that for all $h\in \Gg$ and all $A \in \glg(\ggo)$ we have
\begin{eqnarray}
 h \cdot \left( \pi(A) \mu\right) 
   = \pi(h \, A \, h^{-1}) \left( h \cdot \mu\right)\,.\label{eqn_glnequi}
\end{eqnarray}
Moreover, by Lemma \ref{lem_Ricciintertw} we have for all
$k \in \OHm$ and all $\mu \in \OlH$ that
\begin{eqnarray}
     k \cdot \Ricci_\mu \cdot \, k^{-1} = \Ricci_{k\cdot\mu}.
      \label{eqn_onequi}
\end{eqnarray}      
Recall also that by Theorem \ref{thm:Riccibracket}, a bracket flow solution $\mu(t)$ may be expressed as $\mu(t) = h(t) \cdot \hml$ for some $(h(t)) \subset \GHm$, and the corresponding Ricci flow solution is given by $\ip_{g(t)} = \la h(t) \cdot, h(t) \cdot\ra$. Thus, if $(k(t)) \subset \OHm$ then the curve of brackets $k(t) \cdot \mu(t)$ yields exactly the same Ricci flow solution $\ip_{g(t)}$.


Let us denote by $\sogHm$ the Lie algebra of $\OHm$.

\begin{proposition}\label{prop_gaugedBF}
Let $X: \OlH \to \sogHm$ be a smooth map 
and let $\mu(t)$, $\bar \mu(t)$ denote the solutions to the bracket flow \eqref{eqn_BracketFlow} and to the \emph{modified bracket flow}
\begin{equation}\label{eqn_modifBracketFlow}
  \frac{ \dd \bar \mu}{\dd t} = - \pi\left( \Ricci_{ \bar \mu} - X_{\bar \mu} \right) 
         \bar \mu \,, \qquad \bar \mu(0) = \mu_0,
\end{equation}
respectively, with the same initial condition $\mu_0\in \OlH$. Then, there exists a smooth one-parameter family $(k(t)) \subset \OHm$ such that $\bar \mu(t) = k(t) \cdot \mu(t)$ for all $t$.
\end{proposition}

\begin{proof}
Let $k(t)$ solve the ODE
\[
    k'(t) = X_{k(t) \cdot \mu(t)} \,\cdot \, k(t)\,, \qquad k(0) = \Id_\mg\,.
\]
It is easy to see that  $k(t) \in \OHm$ is a smooth family, defined for all $t$ for which the bracket flow solution $\mu(t)$ exists. 
Setting now $\tilde \mu=k\cdot \mu$, we use the formula
\[
	\left( k \cdot \mu \right)' =  k \cdot \mu' 
	     +\pi\left(k' k^{-1}\right) \left( k\cdot \mu\right) 
\] 
and \eqref{eqn_glnequi}, \eqref{eqn_onequi} to deduce that
\[
    \tilde \mu' 
 = - \pi \big(k \, \Ricci_{\mu} \, k^{-1} -X_{k\cdot \mu} \big) (k \cdot \mu)
 =   -\pi( \Ricci_{\tilde\mu} - X_{\tilde \mu}) \, \tilde \mu\,.
\]
By uniqueness of solutions it follows that $\tilde \mu = \bar \mu$.
\end{proof}

The \emph{unimodular} part of the Ricci curvature, the \emph{modified Ricci endomorphism}
\begin{eqnarray}
    \Riccim_\mu  &:=& \Ricci_\mu + S(\ad_\mumg \mcv_\mu) \,\, = \,\, \mm_\mu - \unm \kf_\mu\,\label{eqn_Ricciuni},
\end{eqnarray}    
was introduced by Heber in \cite{Heber1998}. It turns out to be of great importance in the study of homogeneous Ricci solitons.

\begin{definition}[Unimodular bracket flow]
The evolution equation
\begin{equation}\label{eqn_unimodBracketFlow}
    \frac{ \dd \mu}{\dd t} = - \pi \big(\Riccim_\mu\big) \mu\,, \qquad \mu(0) = \mu_0 \in \OlH\,,
\end{equation}
is called the \emph{unimodular bracket flow}.
\end{definition}

As a first application of Proposition \ref{prop_gaugedBF}, let us show that 
the bracket flow and the unimodular bracket flow are equivalent in the sense that they both correspond to the same flow of Riemannian metrics.

Recall that if $A\in \glg(\ggo)$ is a derivation of $\mu \in \Vg$, 
that is $A \in \Der(\mu)$, then 
\begin{eqnarray}\label{eqn_dermu}
 (\pi(A) \mu)(Y,Z)=A \, \mu(Y,Z)-\mu(A \,Y,Z)-\mu(Y,A\, Z)=0
\end{eqnarray}
for all $Y,Z \in \ggo$. By the Jacobi identity, $\ad_\mu X \in \Der(\mu)$ for any $X \in \ggo$, .

\begin{corollary}\label{cor_uhl1}
Let $\mu(t)$, $\mu^*(t)$ be solutions to the bracket flow 
and the unimodular bracket flow, respectively, both with the same initial condition.
Then, there exists a smooth one-parameter family $(k(t))\subset \OHm$ such that ${\mu}^*(t)  = k(t) \cdot \mu(t)$ for all $t$. 
\end{corollary}

\begin{proof}
For $A \in \glg(\mg)$ 
let $\proy_{\sog(\mg)}(A) := \unm (A - A^t) \in \sog(\mg)$.
Since $\ad_\mumg \mcv_\mu$ is $\Ad(H)$-equivariant by Corollary \ref{cor_AdHequiv}, the map
$X_\mu := \proj_{\sog(\mg)}(\ad_\mumg \mcv_\mu)$ actually takes values in $\sogHm$.
That result also says that $(\ad_\mu \mcv_\mu)|_{\hg} = 0$, thus under the identification \eqref{eqn_glgminclusion} the map $\ad_\mumg \mcv_\mu \in \glg(\mg)$ corresponds to $\ad_\mu \mcv_\mu \in \glg(\ggo)$. Hence, $\pi(\ad_\mumg \mcv_\mu) \mu = 0$.
The corollary now follows from Proposition \ref{prop_gaugedBF}, by using that
\[
	\Ricci_\mu - \proj_{\sog(\mg)}(\ad_\mumg H_\mu) = \Riccim_\mu - \ad_\mumg \mcv_\mu.
\]
\end{proof}

As to the second application, let $Q\subset \Gl^H(\mg)$ be a Lie subgroup with Lie algebra $\qg \subset \gHm$, with the property that 
\begin{equation}\label{eqn_GequalsOQ}
	\GHm = \OHm \, Q.
\end{equation}
As in (\ref{def_muaction}) (see also \eqref{eqn_Gminclusion}) the Lie group $Q$ acts linearly on the vector space $\Vg$. Clearly, an orbit $Q\cdot \mu_0$ is a smooth immersed submanifold with
tangent space 
\begin{eqnarray}
T_{\mu_0}(Q \cdot \mu_0) =  \{ \pi(A) \mu_0 : A \in \qg \}\,.\label{def_tanorbit}
\end{eqnarray}
By \eqref{eqn_GequalsOQ}, at the Lie algebra level we have that $\gHm = \sogHm + \qg$. Choose a subspace $\kg \subset \sogHm$ such that $\kg$ is complementary to $\qg$ in $\gHm$, that is, 
\begin{equation}\label{eqn_gkq}
	\gHm = \kg \oplus \qg.
\end{equation}
A canonical way of choosing $\kg$ is to take the orthogonal complement of 
\mbox{$\sogHm \cap \qg$} inside $\sogHm$, with respect to the usual scalar product  $\la A, B\ra = \tr AB^t$ on $\glg(\ggo)$. The reader should be warned that \eqref{eqn_gkq} is not an orthogonal decomposition in general. Nevertheless, it gives rise to a linear projection 
\begin{equation}\label{eqn_Xqg}
    X_\qg : \gHm \to  \sogHm,  \qquad A = A_\kg + A_\qg \mapsto   A_\kg. 
\end{equation}

The following result shows that one can assume
that a bracket flow solution in $\OlH$ lies in fact within an orbit of the smaller subgroup
$Q$ of $\GHm$. This should be considered as a conservation law 
for the bracket flow, and it is a consequence of 
the  $\OHm$-equivariance of the bracket flow equation, and
the fact that the group $Q$ is large enough so that \eqref{eqn_GequalsOQ} holds.

\begin{corollary}\label{cor_gauge}
Let $Q$ be a Lie subgroup of $\Gl(\mg)$ satisfying $\GHm = \OHm \, Q$, and consider
  a solution $\mu^*(t)$ to the unimodular bracket flow \eqref{eqn_unimodBracketFlow}. Then, there exists a smooth curve $k(t) \in \OHm$ such that $\bar \mu(t) :=k(t) \cdot \mu^*(t) \in Q \cdot \mu_0$ for all $t$. Moreover, $\bar \mu(t)$ is a solution to the \emph{$Q$-gauged bracket flow} equation 
\begin{equation}\label{eqn_QgaugedBF}
	\frac{\rm d \bar \mu}{ {\rm d} t} = - \pi \big( \Riccim_{\bar\mu} 
	-  X_\qg(\Riccim_{\bar \mu})  \big) {\bar \mu}, \qquad {\bar \mu}(0) = \mu_0 \in \OlH\,.
\end{equation}
\end{corollary}

\begin{proof}
By the very definition of $X_\qg$ in \eqref{eqn_Xqg}, we have that $A-X_\qg (A) \in \qg$ for all $A \in \gHm$. 
Thus by \eqref{def_tanorbit} we have $\bar \mu(t) \in Q\cdot \bar \mu(0)$
for any solution $\bar \mu(t)$ to \eqref{eqn_QgaugedBF}. 

Next,  replacing
$\Ricci_\mu$ by $\Riccim_\mu$, it follows precisely as in \mbox{Proposition \ref{prop_gaugedBF}}, that $\bar \mu(t) =  k(t) \cdot \mu^*(t)$ for a smooth curve $k(t) \in \OHm$
and a solution  $\mu^*(t)$ to the unimodular bracket flow \eqref{eqn_unimodBracketFlow}. This shows the claim.
\end{proof}

Notice that if  with respect to some basis $Q$ contains the
lower triangular matrices in $\Gl(\mg)$, then \eqref{eqn_GequalsOQ} holds. These are the so called \emph{parabolic subgroups} of $\Gl(\mg)$.

Finally, we show that the modified scalar curvature
$\scalm(\mu):=\tr \Riccim_\mu$ obeys the same evolution equation as the scalar curvature does.

\begin{lemma}\label{lem_scalmev}
Along a solution $(\mu(t))_{ t \in [0,T)}$ to the
unimodular bracket flow one has
\[
  \ddt \scalm(\mu(t)) = 2 \, \Vert {\Riccim_{\mu(t)} } \Vert^2.
\]
In particular, if $\scalm(\mu(t_0)) > 0$ for some $t_0$, then $T<\infty$.
\end{lemma}

\begin{proof}
The evolution equation follows directly from the proof of Proposition 3.8 in \cite{homRF}. The last claim follows by a standard ODE comparison argument, as in the case of the unmodified scalar curvature (cf.~ \cite[Prop.~4.1]{scalar}).
\end{proof}

Let us mention here that in general $\scal(\mu)< 0$ does not imply
\mbox{$\scalm(\mu)<0$:} if $\ggl$ is non-unimodular then one has the strict inequality $\scalm(\mu) > \scal(\mu)$, since they differ by $\tr_\mg (\ad_\mumg(\mcv_\mu)) = \Vert \mcv_\mu \Vert^2 > 0$. Thus for any Ricci flow solution of left-invariant metrics starting with $\scal(\mu_0) < 0$ but with finite extinction time, $\scal(\mu)$ becomes positive after some time by \cite{scalar}, hence so does $\scalm(\mu)$. It is clear though that the latter becomes positive at an earlier time, thus for some $t_0$ we have $\scal(\mu(t_0))<0$ and $\scalm(\mu(t_0)) > 0$. 
Recall however that $\scal(\mu)< 0$ does imply
$\scalm(\mu)<0$ for solvmanifolds: see \cite[Remark 3.2]{Heber1998}. 

We conclude this section by showing that the vanishing of the modified Ricci curvature is enough to conclude flatness.

\begin{lemma}\label{lem_Ricciflat}
Any $\ggo$-homogeneous space with $\Riccim_\mu = 0$ is flat. 
\end{lemma}

\begin{proof}
Extend the background scalar product $\ip$ on $\mg$ to an $\Ad(H)$-invariant scalar product on $\ggo$ such that $\hg \perp \mg$. Then $\langle \mcv_\mu,X\rangle = \tr_\ggo \ad_\mu (X)$
 holds for all $X\in \ggo$ (and not just in $\mg$).
 Using the Jacobi identity we conclude that $\mcv_\mu\perp \mu(\ggo,\ggo)$.
On the other hand, by Corollary \ref{cor_AdHequiv} we know that $\mu(\hg, \mcv_\mu) = 0$. 
From  \eqref{eqn_defKFend}, \eqref{eqn_formulamm} we deduce now
\begin{align*}
	\la \Riccim_\mu \mcv_\mu, \mcv_\mu \ra &= 
	-\unm \sum_{i,j} \la \mu_\mg (\mcv_\mu, E_i), E_j\ra^2 
	-\unm \tr_\ggo (\ad_\mu \mcv_\mu)^2 \\
	&= 
	-\unm \tr_\mg \big((\ad_\mumg \mcv_\mu)(\ad_\mumg \mcv_\mu)^t\big)
	-\unm \tr_\ggo \big( (\ad_\mumg \mcv_\mu)(\ad_\mumg \mcv_\mu)\big) \\
	& = - \tr_\mg \big( S(\ad_\mumg \mcv_\mu)^2 \big)\,.
\end{align*}
(cf.~\cite[Lemma 2.10]{AL16}). By assumption, this implies that $S(\ad_\mumg \mcv_\mu) = 0$, hence $\Ricci_\mu = \Riccim_\mu = 0$ and the lemma follows from \cite{AlkKml}, \cite{Spiro}.
\end{proof}


\section{Stratification of the space of brackets}\label{sec_stratif}

In this section we will associate to a  $\ggl$-homogeneous space
$(M^n,\ggl)$ with nonabelian Lie algebra $\ggo$ a 
self-adjoint endomorphism 
\[
  \Beta :(\ggo,\ip) \to (\ggo,\ip)\,.
\]
This will be done by applying geometric invariant theory for 
linear actions of real reductive Lie groups on Euclidean spaces. For a proof of the results stated in this section we refer the reader to \cite{GIT}.

Let $(M^n,\ggl)$ be a $\ggl$-homogeneous space 
with Lie bracket $\hml \neq 0$ and canonical 
decomposition $\ggl=\hg \oplus \mg$. After fixing a $G$-invariant background metric $\bar g$ on $M$, we obtain a fixed scalar product $\ip$ on $\mg$.

\begin{definition}\label{def_scalg}
We extend $\ip$ to an $\Ad(H)$-invariant scalar product on $\ggo$, also denoted by $\ip$, by declaring that $\la \hg,\mg\ra = 0$, and by letting $\ip|_{\hg\times\hg}$ be given by
$\langle Z_1,Z_2\rangle_\hg:=-\tr \big( (\nabla^{\bar g} Z_1)_p \cdot   (\nabla^{\bar g} Z_1)_p\big)$, for $Z_1, Z_2\in \hg$; see Section \ref{sec_algcon}.
\end{definition}

Notice that the $\Ad(H)$-invariant scalar product on $\hg$ does not depend on $\bar g$, since for $Z\in \hg$ we have that $(\nabla^{\bar g} Z)_p = (\ad_\hml Z)|_\mg$ after identifying $T_p M \simeq \mg$.

Let $\Og \subset \Gg$ denote the orthogonal group of $(\ggo,\ip)$.
 The scalar product $\ip$ on $\ggo$ 
induces on $\Vg$ an $\Og$-invariant scalar product given by 
\begin{eqnarray}\label{def_scal}
  \langle \mu,\eta \rangle := \sum_{i,j=1}^N \big\la \mu(\tilde E_i,\tilde E_j),
  \eta(\tilde E_i,\tilde E_j) \big\ra\,,
\end{eqnarray}
where $\{\tilde E_i\}_{i=1}^N$ is any orthonormal basis of $(\ggo, \ip)$, and on $\glgg$ the scalar product 
\[
 \la \Alpha_1, \Alpha_2\ra = \tr (\Alpha_1 \Alpha_2^t)
\] 
for $A_1, A_2 \in \ggo$, where again the transpose is taken with respect to $\ip$.
Recall the natural linear $\Gg$-action on $\Vg$ defined in \eqref{def_muaction}, with corresponding $\glgg$-representation 
$\pi:\glgg \to \End(\Vg)$, $A \mapsto \pi(A)$, described in
(\ref{eqn_defpib}).


\begin{definition}
The \emph{moment map} is the function implicitly  defined by
\begin{equation}\label{eqn_defmmb}
   \mmm : \Vg \backslash \{0 \} \to \Symg; \qquad  \la \mmm(\mu), \Alpha \ra 
    = \tfrac1{\Vert \mu\Vert^2} \cdot \la \pi(\Alpha) \mu, \mu\ra \,,
\end{equation} 
for all $\Alpha\in \Symg$, $\mu \in \Vg \backslash \{ 0\}$.
The \emph{energy} of $\mmm$ is denoted by
\[
	\normmm : \Vg\backslash \{ 0\} \to \RR\,\,; \qquad  \mu \mapsto \Vert {\mmm(\mu)} \Vert^2\,.
\] 
\end{definition}
Notice that by \eqref{eqn_glnequi} and the $\Og$-invariance of the scalar products, the moment map is $\Og$-equivariant: 
\begin{equation}\label{eqn_Kequivmm}
	\mmm(k \cdot \mu) = k \, \mmm(\mu) \, k^{-1},
\end{equation}
for all $k\in \Og$, $\mu \in \Vg \backslash \{ 0\}$. The energy $\normmm$ is therefore $\Og$-invariant. 
The following stratification result is essentially contained in \cite{HSS} and \cite{standard}:

\begin{theorem}\label{thm_stratifb}
There exists a finite subset $\bca \subset \Symg$ 
and a collection of smooth, $\Gg$-invariant submanifolds 
$\{ \sca_\Beta \}_{\Beta \in \bca}$ of $\Vg$, with the following properties:
\begin{itemize}
    \item[(i)]  
     We have $\Vg\backslash \{ 0\} = \bigcup_{\Beta\in \bca} \sca_\Beta$
     and $\sca_\Beta \cap \sca_{\Beta'}=\emptyset$ for $\Beta \neq \Beta'$.
    \item[(ii)] We have 
     $\overline{\sca_\Beta} \backslash \sca_\Beta \subset \bigcup_{\Beta'\in \bca, \Vert \Beta'\Vert > \Vert \Beta \Vert} \sca_{\Beta'}$.
    \item[(iii)] A bracket $\mu$ is contained in
       $\sca_\Beta$ if and only if the negative gradient 
             flow of $\normmm$ starting at $\mu$ 
             converges to a critical point $\mu_C$ of $\normmm$
             with $\mmm(\mu_C) \in \Og \cdot \Beta$. 
\end{itemize} 
\end{theorem}

The \emph{strata} $\sca_\Beta$ are $\Gg$-invariant submanifolds. Notice  that in (ii)
the closure is taken in $\Vg \backslash \{0\}$. Observe also that by the $\Og$-equivariance of the moment map, each $\Beta \in \bca$ may be replaced by any $ \beta'\in \Og \cdot \Beta := \{ k\, \Beta \, k^{-1} : k\in \Og\}$. 

For a non-abelian Lie algebra $\ggo$ we have  $\hml \neq 0$ and hence
 $\hml \in \sca_\Beta$ for some $\Beta \in \Symg$.
This is how we associate to a $\ggo$-homogeneous space $(M^n,\ggo)$ the endomorphism $\Beta$.
Notice that  $\OlH \subset \sca_\Beta$ by $\Gg$-invariance of $\sca_\Beta$.

\begin{remark}\label{rem_git1}
Since $\pi(\Id_\ggo) = -\Id_{\Vg}$, it follows from \eqref{eqn_defmmb} that $\tr_\ggo \mmm(\mu) = -1$
for all $\mu\in \Vg \backslash \{ 0\}$.
Each $\Beta\in \bca$ in the Stratification Theorem \ref{thm_stratifb} is the image under $\mmm$
of a critical point of $\normmm$, so in particular it also satisfies
$\tr \Beta = -1$.
\end{remark}

In order to define the new Lyapunov function for the bracket flow we need to
describe the strata $\sca_\Beta$ in greater detail. 
To that end, using that $\Gg$ is a real reductive Lie group we fix the Cartan decompositions 
\[
	\glgg = \sogg \oplus \pg,  \qquad \Gg = \Og \exp(\pg),
\] 
where $\pg := \Symg$ and $\exp : \glgg \to \Gg$ is the Lie exponential map. 
The group $\Og$ acts isometrically on $\Vg$, thus $\pi(\sogg):=\{\pi(A):A \in \sogg\}$ acts skew-symmetrically
on $\Vg$. A short computation shows that $\pi(A^t) = \pi(A)^t$, thus $\pi(\pg)$ consists of symmetric endomorphisms.

The following notation will also be convenient: 
\begin{notation}
 For $\Beta\in \pg$ we set $\Beta^+ := \Beta + \Vert \Beta \Vert^2 \Id_\ggo\,$.
\end{notation}
Denote by $\Vr \subset \Vg$ the eigenspace of $\pi(\Beta^+) = \pi(\Beta) - \Vert \Beta \Vert^2 \, \Id_{\Vg}$
corresponding to the eigenvalue $r\in \RR$, and consider the sum of the nonnegative eigenspaces 
\begin{equation}\label{eqn_defVnn}
   \Vnn := \bigoplus_{r \geq 0} \Vr\,.
\end{equation}
We now define subgroups of $\Gg$ adapted to these subspaces. To that end, observe that $\ad(\Beta) : \glgg \to \glgg$, $A \mapsto [\beta, A]$, is also a self-adjoint operator, so one can consider the decomposition $\glg(\ggo) = \bigoplus_{r\in \RR} \glg(\ggo)_r$ into $\ad(\Beta)$-eigenspaces, and set accordingly
\[
    \ggo_\Beta := \glg(\ggo)_0 = \ker (\ad(\Beta) ), \qquad \ug_\Beta := \bigoplus_{r> 0} \glg(\ggo)_r, \qquad \qg_\Beta := \ggo_\Beta \oplus \ug_\Beta.
\]


\begin{definition}\label{def_groups1b}
We denote by 
\[
    G_\Beta := \{ g \in \Gg : g \Beta g^{-1} = \Beta \}, \qquad U_\Beta := \exp(\ug_\Beta), \qquad Q_\Beta := G_\Beta U_\Beta
\] 
the centralizer of $\Beta$ in $\Gg$, the unipotent subgroup associated with $\Beta$, and the parabolic subgroup associated with $\Beta$, respectively. For the subalgebra
\[
    \hg_\Beta = \{ \Alpha \in \ggo_\Beta : \langle \Alpha,  \Beta\rangle = 0\}
\]
of $\ggo_\Beta$, we consider the codimension-one reductive subgroup $H_\Beta$ of $G_\Beta$, defined by
\[
	H_\Beta = K_\Beta \, \exp(\pg \cap \hg_\Beta),
\]
where here $K_\Beta = G_\Beta \cap \Og$.
\end{definition}

Notice that $G_\Beta$, $U_\Beta$ and $Q_\Beta$ are closed subgroups of $\Gg$, and $U_\Beta$ is normal in $Q_\Beta$. The subgroup $G_\Beta$ is real reductive, with Cartan decomposition $G_\Beta = K_\Beta \exp(\pg_\Beta)$, $\pg_\Beta = \pg \cap \ggo_\Beta$, induced from that of $\Gg$. The same holds for $H_\Beta$, and in fact $G_\Beta$ is the direct product of $\exp(\RR \Beta)$ and $H_\Beta$.

\begin{remark}\label{rem_git2}
If $\mu_C$ is a critical point of $\normmm$, then 
$\Og\cdot \mu_C$ consists of critical points too. By the $\Og$-equivariance of the moment map
\eqref{eqn_Kequivmm} we may  assume that $\Beta=\mmm(\mu_C)$ is diagonal with respect to some orthonormal basis of $(\ggo,\ip)$.
If $\Beta$ has $r$ different eigenvalues with multiplicities $m_1, \ldots, m_r$, then 
\begin{eqnarray}
  G_\Beta \simeq \Gl(m_1,\RR) \times \cdots \times \Gl(m_r,\RR)\,.\label{eqn_Gbetadecomp}
\end{eqnarray}
Moreover, the elements of $U_\Beta$ have diagonal entries $1$, and non-zero entries 
possibly only at those entries below the diagonal where all elements of $G_\Beta$ have zeros. Clearly, $Q_\Beta$ contains the lower-triangular matrices in $\Gg$, so we have 
\[
	\Gg= \Og \, Q_\Beta.
\]
 Finally,  by Remark \ref{rem_git1} we have that $\langle \Beta^+,\Beta\rangle=0$ and thus $\Beta^+ \in \hg_\Beta$.
\end{remark}

\begin{definition}\label{def_Vss}
Let $\mu_C$ be a critical point of $\normmm$ and set $\Beta := \mmm(\mu_C)$.
Then we call
\[
 \Vzeross := \big\{ \mu \in \Vzero : 0\notin \overline{H_\Beta\cdot \mu} \big\}\,.
\]
the open subset of \emph{semi-stable} brackets in $\Vzero$ with respect to the action of $H_\Beta$.
We also define accordingly 
\[
 \Vnnss := p_\Beta^{-1} \big(\Vzeross\big),
\] 
where $p_\Beta : \Vnn \to \Vzero$ denotes the orthogonal projection.
\end{definition}

It follows from the proof of Theorem \ref{thm_stratifb} that the stratum $\sca_\Beta$ is given by 
\begin{equation}
	\sca_\Beta = \Gg \cdot \Vnnss = \Og \cdot \Vnnss. \label{eqn_stratumbeta}
\end{equation}
In particular, for any $\mu \in \sca_\Beta$ one can always find $k\in \Og$ such that $k \cdot \mu \in \Vnn$. 

In the rest of this section we collect several useful properties of the groups and subsets associated with $\Beta\in \Sym(\ggo)$  (cf.~ \eqref{eqn_defVnn} and Definitions \ref{def_groups1b}, \ref{def_Vss}).

\begin{proposition}\label{prop_propertiesgroups}
The following properties hold:
\begin{itemize}
	\item[(i)] The subsets $\Vnnss$ and $\Vnn$ of $\Vg$ are $Q_\Beta$-invariant;
	\item[(ii)] The projection $p_\Beta : \Vnn \to \Vzero$ satisfies
	$p_\Beta(\mu) = \lim_{t\to +\infty} \exp(-t \Beta^+) \cdot \mu$.
	\item[(iii)] $p_\Beta$ is $G_\Beta$ invariant, and for $\mu\in \Vzero$ the fiber $p_\Beta^{-1}(\mu)$ is $U_\Beta$-invariant.
\end{itemize}
\end{proposition}

If $H$ is a closed subgroup of a Lie group $G$, and $H$ acts on a space $U$, we denote by $G \times_H U$ the $U$-fiber bundle associated with the $H$-principal bundle $G \to G/H$. That is, $G \times_H U$ is the quotient of $G \times U$ by the proper and free action of $H$ given by $h \cdot (g,u) = (gh^{-1}, h\cdot u)$. In our particular setting, we can construct the smooth manifold $\Gg \times_{Q_\Beta} \Vnnss$, and there is a natural map 
\[
	\Psi : \Gg\times_{Q_\Beta} \Vnnss \to \sca_\Beta, \qquad [(h, \mu)] \mapsto h\cdot \mu.
\] 
There is also an analogous map $\Psi_K: \Og \times_{K_\Beta} \Vnnss \to \sca_\Beta$.

\begin{proposition}\label{prop_bundlediffeosbeta}
The maps $\Psi$, $\Psi_K$ are diffeomorphisms.
\end{proposition}

Finally, the following important property of the automorphisms groups can be deduced from the previous proposition.

\begin{corollary}\label{cor_autmu}
If $\mu \in \sca_\Beta \cap \Vnn = \Vnnss$ then 
\[
	\Aut(\mu) := \{ h\in \Gg : h\cdot \mu = \mu\} \subset H_\Beta U_\Beta,
\]
and in particular we have that
$
	\Aut(\mu) \cap \Og \subset H_\Beta.
$
\end{corollary}




\section{The stratum of a homogeneous space}\label{sec_strathom}

In the last section we assigned to a $\ggo$-homogeneous space
$(M^n,\ggo)$ with $\hml \neq 0$ 
an endomorphism $\Beta \in \Symg$ and a stratum $\sca_\Beta$
containing the orbit $\OlH$. 
Moreover, we may also assume that $\hml$ is \emph{gauged correctly} with respect to our fixed $\Beta$, that is, 
\begin{equation}\label{eqn_orbitbeta}
	\hml \in \Vnn, \qquad \OlH := \GHm \cdot \hml \subset \sca_\Beta
\end{equation}
is satisfied.
In this section we will describe further properties of $\Beta$.

Denote by $\ngo$ the nilradical of $(\ggl,\hml)$, that is, the maximal nilpotent ideal. By \cite[$\S$4 Prop.~6, b)]{Bou} we know that $\kf_{\hml}(\ngo, \cdot)=0$, thus $\ngo \subset \mg$ by the
very definition of the canonical complement $\mg$. From Corollary \ref{cor_Imbetaplus} 
we deduce the following

\begin{lemma}\label{lem:betaplusimage}
We have ${\rm Im}(\Beta^+) = \ngo \subset \mg$ and $\Beta^+\vert_\ngo >0$.
\end{lemma}

Next, \eqref{eqn_orbitbeta} and Corollary \ref{cor_autmu} yield $\Ad(H) \subset \Aut(\hml)\subset Q_\Beta=G_\Beta U_\Beta$. 
Since  $\Ad(H) \subset \Og$, we conclude that 
$\Ad(H) \subset G_\Beta$. By the very definition of $G_\beta$ (see \eqref{eqn_Gbetadecomp}) and the fact that  $\Im(\Beta^+) \subset \mg$, the following holds:

\begin{lemma}\label{lem:betacommuteAdH}
We have that $\big[ \! \Ad(H), \Beta \big] = \big[ \! \Ad(H)|_\mg, \Beta|_\mg \big] = 0$. 
\end{lemma}

The following technical result shows that any $G$-invariant metric has an associated bracket which is also gauged correctly.

\begin{lemma}\label{lem_gaugeohm}
Let $(M^n, \ggl)$ be a $\ggl$-homogeneous space with $\hml \in \sca_\Beta \cap \Vnn$. Then, for any $\mu \in \OlH$ there exists $k\in \OHm$ such that 
\[
  k\cdot \mu \, \in \,  \big( Q_\Beta \cap \GHm \big)\cdot \hml \,  \subset  \, \sca_\Beta \cap \Vnn\,.
\]
\end{lemma}
\begin{proof}
It is enough to show that
\begin{equation}\label{eqn_GHmdecomp}
 \GHm = {\OHm(Q_\Beta \cap \GHm)}\,.
\end{equation}
 Recall that $\Vnn$ is $Q_\Beta$-invariant by Proposition \ref{prop_propertiesgroups}. To prove \eqref{eqn_GHmdecomp}, consider the decomposition $\mg = \pg_1\oplus \cdots \oplus \pg_k$ into $K$-isotypical summands, $K:= \overline{\Ad(H)|_\mg} \subset \Or(\mg)$, see the paragraph after Notation \ref{not_ipg}. By Lemma \ref{lem:betacommuteAdH}, for each fixed $i$ we may write $\pg_i$ as an orthogonal sum $\pg_i = \pg_i^1 \oplus \cdots \oplus \pg_i^{r_i}$ of eigenspaces of $\Beta|_\mg$, which are furthermore $K$-invariant. Set $\Beta_i := \Beta|_{\pg_i}$. After fixing a splitting of $\pg_i$ into $n_i$ $K$-irreducible summands so that each $\pg_i^j$ is a sum of some of these summands, we obtain an identification $\Gl^H(\pg_i) \simeq \Gl(n_i, \KK_i)$, where $\KK_i = \RR, \CC$ or $\HH$. Under said identification we also have $\Or^H(\pg_i) \simeq \Or(n_i, \KK_i)$,  with $\Or(n_i, \KK_i)$ isomorphic to $\Or(n_i)$, $\U(n_i)$ or $\Spe(n_i)$ according to whether $\KK_i = \RR$, $\CC$ or $\HH$, respectively.
 
Next, we claim 
 that the lower-triangular matrices $Q(n_i, \KK_i)\subset \Gl(n_i, \KK_i)$  are contained  in $Q_{\Beta_i} \cap \Gl^H(\pg_i)$, see Remark \ref{rem_git2}. Here, $Q_{\Beta_i}$ is the (parabolic) subgroup of $\Gl(\pg_i)$ defined as in Definition \ref{def_groups1b}. That is $Q_{\Beta_i}$  contains as in 
 \eqref{eqn_Gbetadecomp} diagonal block matrices corresponding to
 the $K$-invariant eigenspaces of $\Beta_i$ and the corresponding strictly 
 lower triangular matrices.
 Since the embedding of $Q(n_i, \KK_i)$ into $\Gl(\pg_i)$
 respects this block decomposition, see \cite{Bhm04} p.104,
 the above claim follows.
 
 Now using $\Gl(n_i, \KK_i) = \Or(n_i, \KK_i) Q(n_i, \KK_i)$ one gets
 \[
 	\Gl^H(\pg_i) \subset \Or^H(\pg_i) (Q_{\Beta_i} \cap \Gl^H(\pg_i)).
 \] 
From that and \eqref{eqn_decompGHm} it is clear that $\Gl^H(\mg) \subset \Or^H(\mg)(Q_\Beta \cap \Gl^H(\mg))$. The reverse inclusion is immediate.
\end{proof}


If $\ggo$ is abelian then all $\GG$-invariant metrics on $(M^n, \ggo)$ are flat. On such a space
one understands the Ricci flow of $\GG$-invariant metrics completely.

\begin{definition}
A $\ggo$-homogeneous space  $(M^n, \ggo)$ is called {\it non-flat}, if
it admits at least one non-flat $\GG$-invariant metric.
\end{definition}

Notice that a non-flat $\ggo$-homogeneous space may admit flat $\GG$-invariant metrics.

\begin{lemma}\label{lem_trbetam}
Let $(M^n, \ggo)$ be a non-flat $\ggo$-homogeneous space with $\hml \in \sca_\Beta \cap \Vnn$. Then, 
 $\tr_\mg (\Beta|_\mg) < 0$.
\end{lemma}
\begin{proof}
Let $\ngo$ be the nilradical of $(\ggo,\hml)$, and denote by $\nu : \ngo \wedge \ngo \to \ngo$ the restriction of $\hml$ to $\ngo \wedge \ngo$. It follows immediately from Theorem \ref{thm_refinedbeta} that $\tr_\mg (\Beta|_\mg) \leq 0$, since either $\nu = 0$, or $0\neq \nu \in \sca_{\Beta_\ngo}$ and $\tr \Beta_\ngo = -1$. Moreover, the inequality is strict, unless $\mg = \ngo$ and $\nu = 0$. If this is the case, then $\hml(\mg,\mg) = 0$, $(\ggo,\hml)$ is unimodular, and $\kf_\hml\vert_{\ngo}=0$.  
 Thus it follows from \eqref{eqn_Riccimg} that any $\GG$-invariant metric is Ricci-flat, hence flat by \cite{AlkKml} and \cite{Spiro}.
\end{proof}

\begin{notation}\label{not_betam}
For a non-flat $\ggo$-homogeneous space we define $\Beta_\mg \in \Symm$ by
\begin{equation}\label{eqn_defbetam}
	\Beta_\mg := b \cdot \Beta|_\mg,
\end{equation} 
with $b := (-\tr_\mg \Beta|_\mg)^{-1} > 0$ (Lemma \ref{lem_trbetam}) chosen so that 
$\tr \Beta_\mg = -1$. We also set
\begin{equation}\label{eqn_defbetamplus}
(\Beta_\mg)^+ := \Beta_\mg + \Vert \Beta_\mg \Vert^2\, \Id_\mg.
\end{equation} 
\end{notation}

\begin{lemma}\label{lem_betamgplusbetaplus}
Let $(M^n, \ggo)$ be a non-flat $\ggo$-homogeneous space with $\hml \in \sca_\Beta \cap \Vnn$.
Then $(\Beta_\mg)^+ = b \cdot (\Beta^+|_\mg)$.
\end{lemma}

\begin{proof}
We have 
$b^{-1} = -\tr_\mg \Beta|_\mg = 1 - \Vert \Beta \Vert^2 \, (\dim\hg)$,
since $\Beta|_\hg = -\Vert \Beta \Vert^2 \, \Id_\hg$  by Theorem \ref{thm_refinedbeta} and
$\tr \Beta = -1$ by Remark \ref{rem_git1}. On the other hand,
\[
	\Vert \Beta|_\mg \Vert^2 = \Vert \Beta \Vert^2 - \Vert \Beta \Vert^4 \, (\dim \hg) = \Vert \Beta \Vert^2 \, b^{-1}\,,
\]
from which the claim follows.
\end{proof}

\begin{remark}\label{rmk_normbetam}
From the proof of Lemma \ref{lem_betamgplusbetaplus} it follows that $\Vert \Betam \Vert^2 = f(\Vert \Beta \Vert^2)$, where $f : \RR_{>0} \to \RR_{>0}$, $f(x) = x \cdot (1- (\dim \hg) x)^{-1}$, is an increasing function.
\end{remark}


\section{New curvature estimates}\label{sec_newcurv}



The aim of this section is to obtain new Ricci curvature estimates which are sharp precisely on expanding homogeneous Ricci solitons. To that end, we first relate the moment map $\mmm(\mu) \in \Symg$ for the action \eqref{def_muaction} of $\Gg$ on $\Vg$ with the first summand $\mm_\mumg \in \Symm$ in the formula for the Ricci curvature \eqref{eqn_FormulaRicci}: see also \eqref{eqn_formulamm}. Recall that the latter depends only on the projection $\mumg \in \Vm$ (see \eqref{eqn_muhgmg}), and that in \cite{LafuenteLauret2014b} it was shown to coincide -up to scaling- with the moment map for the natural action of $\Gm$ on $\Vm$. However, since this action is on a different space, no a priori relation between them follows from the general theory.  

We set for simplicity $\mm_\mu := \tfrac{4}{\Vert \mu \Vert ^2} \mmm(\mu) \in \Symg$. The following formula for $\mm_\mu$ was given in \cite[Prop.~3.5]{Lau06} (the definition of $\mmm(\mu)$ in that paper differs from ours by a factor of $2$):
\begin{equation}\label{eqn_formulammgg}
    \mm_\mu =
    -\frac12 \sum_{i=1}^N \big(\ad_\mu \tilde E_i \big)^t  \big( \ad_\mu \tilde E_i \big) + 
    	\frac14 \sum_{i=1}^N \big( \ad_\mu \tilde E_i \big)  \big(\ad_\mu \tilde E_i \big)^t.
\end{equation}
Here, $\{ \tilde E_i\}_{i=1}^N$ is any orthonormal basis for $(\ggo, \ip)$. Let $\proy_\mg : \ggo \to \mg$ denote the orthogonal projection, and set $\mm_\mu^\mg := \proy_\mg \circ \mm_\mu \big|_\mg \in \Symm$. 

\begin{lemma}\label{lem_formulabigmm}
Let $(M^n,\ggo)$ be a $\ggo$-homogeneous space. Then for $\mu \in \OlH$ we have 
\[
	\mm_\mu^\mg = \mm_\mumg - {\operatorname P}_{\mu_\hg},
\]
where ${\operatorname P}_{\mu_\hg} \in \Symm$ is given by
\begin{equation}\label{eqn_defPmuhg}
	\la {\operatorname P}_{\mu_\hg} X, X\ra = \frac12 \sum_{i,j} \big\la \mu(X, E_i), Z_j \big\ra^2,
\end{equation}
for $X \in \mg$. Here $\{Z_j\}_{j=n+1}^N$, $\{ E_i\}_{i=1}^n$ are orthonormal basis for $\hg$, $\mg$, respectively.
\end{lemma}

\begin{proof}
Let $Z\in \hg$, $X\in \mg$. Using that $\mu(\hg,\hg) \subset \hg$, $\mu(\hg,\mg) \subset \mg$ (which holds by Definition \ref{def_bracket}, since $\hml$ satisfies that) we may write
\[
	\ad_\mu Z = \minimatrix{\ad_\hg Z}{0}{0}{\ad_\eta Z}, \qquad \ad_\mu X = \minimatrix{0}{\ad_{\mu_\hg} X}{\ad_\eta X}{\ad_\mumg X},
\]
according to the decomposition $\ggo = \hg \oplus \mg$. Here $\eta : (\hg \times \mg) \oplus (\mg \times \hg) \to \mg$ is the corresponding projection of $\mu$. Notice that $\ad_\eta X : \hg \to \mg$, $\ad_\eta Z : \mg \to \mg$. We now write the formula \eqref{eqn_formulammgg} using an orthonormal basis as in the statement. A straightforward computation using that $\ad_\eta Z_j \in \sog(\mg)$ yields
\begin{align*}
	\mm_\mu^\mg  =&  -\tfrac14 \sum_j (\ad_\eta Z_j)^t(\ad_\eta Z_j) \\
			& -\unm \sum_i (\ad_{\mu_\hg} E_i)^t (\ad_{\mu_\hg} E_i)  -\unm \sum_i  (\ad_\mumg E_i)^t(\ad_\mumg E_i) \\
			& + \tfrac14 \sum_i (\ad_\eta E_i)(\ad_\eta E_i)^t+ \unc \sum_i (\ad_\mumg E_i)(\ad_\mumg E_i)^t.
\end{align*}
First we notice that the first and fourth terms cancel out. Indeed, the $\ad_\eta Z_j$ are normal operators, and the following holds for any $X\in \mg$:
\begin{align*}
	\sum_i \left\la (\ad_\eta E_i)(\ad_\eta E_i)^t X, X\right\ra =& \sum_i \left\Vert (\ad_\eta E_i)^t X \right\Vert^2 = \sum_{i,j} \left\la (\ad_\eta E_i)^t X, Z_j \right\ra^2 \\
	=& \sum_{i,j} \left\la X, \eta(E_i, Z_j)  \right\ra^2 = \sum_{i,j} \left\la (\ad_\eta Z_j)^t X, E_i \right\ra^2 \\
	=& \sum_j \left\Vert (\ad_\eta Z_j)^t X\right\Vert^2 = \sum_j \left\la (\ad_\eta Z_j)(\ad_\eta Z_j)^t X, X \right\ra.
\end{align*} 
From \eqref{eqn_formulamm} it is clear that the third and fifth terms add up to $\mm_\mumg$. Finally, the lemma follows after noticing that for the second term we have 
\[
	\sum_i \left\la (\ad_{\mu_\hg} E_i)^t (\ad_{\mu_\hg} E_i) X, X \right\ra = \sum_i \left\Vert \mu_\hg(E_i, X)  \right\Vert^2 = \sum_{i,j} \la \mu(E_i,X), Z_j\ra^2.
\]
\vskip-10pt
\end{proof}

On the orbit $\OlH$ we have defined the Ricci endomorphism 
$\Ricci_\mu$ in (\ref{eqn_FormulaRicci}),
the modified Ricci curvature $\Riccim_\mu$ in (\ref{eqn_Ricciuni}),
and the modified scalar curvature 
$\scalm(\mu) = \tr \Riccim_\mu$. 
Recall that when $(\ggl,\hml)$ is unimodular we have $\Ricci_\mu = \Riccim_\mu$.

\begin{lemma}\label{lem_estimateRic}
Let $(M^n, \ggl)$ be a $\ggl$-homogeneous space with $\hml \in \sca_\Beta \cap \Vnn$. 
Then, for any $\mu \in \OlH \cap \Vnn$ we have
\[
  \big\la  \Riccim_\mu,\Beta^+  \big\ra \geq 0,
\]
with equality if and only if $\Beta^+ \in \Der(\ggo,\mu)$. Here, $\Beta^+$ is considered as an endomorphism of $\mg$ thanks to Lemma \ref{lem:betaplusimage}.
\end{lemma}

\begin{proof}
By \eqref{eqn_Ricciuni} and Lemma \ref{lem_formulabigmm} we have that
\[
	\Riccim_\mu = \mm_\mu^\mg + {\operatorname P}_{\mu_\hg} - \unm \kf_\mu.
\]
Write $\mu \in \Vnn$ as a sum of eigenvectors of $\pi(\Beta^+)$: $\mu = \sum_{r\geq 0} \mu_r$, $\mu_r \in V_{\Beta^+}^r$. Then, 
\[
	\la \mmm(\mu), \Beta^+\ra = \tfrac1{\Vert \mu\Vert^2} \cdot \la \pi(\Beta^+) \mu, \mu\ra = \tfrac1{\Vert \mu\Vert^2} \cdot  \sum_{r\geq 0} r \cdot \Vert \mu_r\Vert^2 \geq 0,
\]
with equality if and only if $\mu \in \Vzero$. The latter is in turn equivalent to 
\mbox{$\pi(\Beta^+)\mu = 0$}, that is, $\Beta^+ \in \Der(\ggo,\mu)$. From Corollary \ref{cor_Imbetaplus} we know that $\Im (\Beta^+) = \ngo_\mu$ is the nilradical of $(\ggo,\mu)$. Thus, we have that $\kf_\mu \perp \Beta^+$,
since $\ngo_\mu \subset \ker(\kf_\mu)\subset \mg$ by \cite[$\S$4 Prop.~6, b)]{Bou}. Since $\ngo_\mu$ is an ideal of ($\ggo,\mu$), and $\ngo_\mu \perp \hg$, a quick look at \eqref{eqn_defPmuhg} shows that ${\operatorname P}_{\mu_\hg} (\ngo_\mu) = 0$. Hence, ${\operatorname P}_{\mu_\hg} \perp \Beta^+$. Putting all this together,  we get
\[
	\la \Riccim_\mu, \Beta^+ \ra = \la \mm_\mu^\mg, \Beta^+ \ra = \tfrac{4}{\Vert \mu \Vert^2} \, \la \mmm(\mu), \Beta^+  \ra \geq 0,
\]
and the lemma follows.
\end{proof}


The following estimate, together with the equality rigidity from Theorem \ref{thm_lochomsolitons}, will be crucial in the construction of the Lyapunov function for immortal homogeneous Ricci flows.

\begin{corollary}\label{cor_mainestimate}
Let $(M^n,\ggl)$ be a non-flat $\ggl$-homogeneous space with $\hml \in \sca_\Beta \cap \Vnn$. Let $\mu \in \OlH \cap \Vnn$ with $\scalm(\mu) < 0$. Then,
\begin{equation}\label{eqn_mainestimate}
	\Vert \Riccim_\mu\Vert \geq \vert {\scalm (\mu)} \vert \cdot \Vert \Beta_\mg \Vert,
\end{equation}
with equality if and only if $\Riccim_{\mu} = -\scalm(\mu) \cdot \Betam$.
\end{corollary}
\begin{proof}
By Lemma \ref{lem_estimateRic} and Lemma \ref{lem_betamgplusbetaplus}
we obtain that $\la \Riccim_\mu, (\Beta_\mg)^+\ra \geq 0$, thus
\[
	\Vert \Riccim_\mu \Vert \, \Vert \Beta_\mg\Vert \geq \la \Riccim_\mu, \Beta_\mg\ra \geq - \scalm(\mu) \cdot \Vert \Beta_\mg \Vert^2\, 
\]
by \eqref{eqn_defbetamplus}, Cauchy-Schwarz inequality and the fact that $\tr \Beta_\mg = -1$. The equality condition is clear.
\end{proof}

We turn to the equality case in the above estimate.

\begin{theorem}\label{thm_lochomsolitons}
Let $(M^n,\ggl)$ be a non-flat $\ggl$-homogeneous space with $\hml \in \sca_\Beta \cap \Vnn$, let $\mu \in \OlH \cap \Vnn$ be a bracket corresponding to a $\GG$-invariant metric $g$ on $M^n$ and
assume that $\Riccim_\mu = \Beta_\mg$. Then, $(M^n,g)$ is locally isometric to an expanding 
$G$-invariant Ricci soliton, which is globally homogeneous. Moreover, the modified Ricci curvature is given by
\[
    \Riccim_\mu = c \cdot \Id_\mg + D_\mg, \qquad c = -  \Vert \Beta_\mg \Vert^2<0, \qquad D_\mg = \proy_\mg \circ D \, |_\mg, 
\]
where 
$D = b \cdot \Beta^+ \in \Der(\ggo,\mu)$ for some $b>0$ (see Notation \ref{not_betam}).
\end{theorem}

\begin{proof}
We first assume that $(M^n, g, p, \ggl)$ is globally homogeneous and simply connected, and show that it is a Ricci soliton.
First notice that by Lemma \ref{lem_betamgplusbetaplus}
\[
	\Riccim_\mu  = \Beta_\mg = -\Vert \Beta_\mg\Vert^2 \, \Id_\mg + (\Beta_\mg)^+  = c\cdot \Id_\mg  + b \cdot (\Beta^+ |_\mg),
\]
where $c = -\Vert \Beta_\mg\Vert^2<0$ by Lemma \ref{lem_trbetam}.
Thus, $\Ricci_\mu = \Riccim_\mu - S(\ad_\mumg \mcv_\mu)$ is of the form
$\Ricci_\mu = c \cdot \Id + S(D_\mg)$,
where $D_\mg = \proy_\mg \circ D |_\mg$ and $D = b \cdot\Beta^+ - \ad_\mu (\mcv_\mu)$.  
Since $\mu \in \Vnn$, the equality in Lemma \ref{lem_estimateRic} tells us that 
$\Beta^+ \in \Der(\ggo,\mu)$. Hence, $D \in \Der(\ggo,\mu)$.
By Proposition 3.3 in \cite{LafuenteLauret2014a} (see also the remark after equation (20) in that paper) we conclude that $(M^n, g, p, \ggl)$ is a homogeneous Ricci soliton. Since $c<0$, it is of expanding type.

To conclude the proof we now show that $(M^n,g, p, \ggl)$
is locally isometric to a simply-connected, globally homogeneous space. 
By \cite{Tri92} it suffices to show that $H$ is closed in $G$, where $G$ is the simply-connected Lie group with Lie algebra $(\ggl, \hml)$, and $H$ the analytic subgroup corresponding to the isotropy subalgebra $\hg$. Suppose this is not the case, and let $\bar{H}$ be its closure in $\GG$, with Lie algebra $\bar{\hg}$. As in Definition \ref{def_scalg}, 
we may extend the scalar product $\ip_g$ on $\mg$  corresponding to $g_p$, to an $\Ad(H)$-invariant scalar product on $\ggo$, also denoted by $\ip_g$, and such that $\hg \perp \mg$. By invariance we have $\Ad(H)\subset \Or(\ggo,\ip_g)$, thus $\Ad(\bar H) \subset \overline{\Ad(H)} \subset \Or(\ggo,\ip_g)$ and hence $\ip_g$ is also $\Ad(\bar H)$-invariant. 
Denoting by $\hg^\perp := \bar \hg \cap \mg$,
we have that $\ad_{\ml}(X):\mg \to \mg$ is $\ip_g$-skew-symmetric
for $X \in \hg^\perp$, thus $\ric_g(X,X) \geq 0$: 
see \cite[Thm.~5.2]{Bhm2014}.

Recall also that we have chosen a background metric $\bar g$ on $M^n$
in order to write $g(\, \cdot \,,\cdot \,)=\bar g( h \,\cdot \,, h \,\cdot )$ for some $h \in \GHm \subset \Gg$,  see \eqref{eqn_Gminclusion}, and then $\mu = h \cdot \hml$. The metric $\bar g$ also induces a background on $\ggo$ (Definition \ref{def_scalg}).
Let $\ngo_\mu = \Im(\Beta^+)$ be the nilradical of $(\ggo,\mu)$ (Corollary \ref{cor_Imbetaplus}), and denote by $\ug_\mu$ its $\ip$-orthogonal complement in $\ggo$. 
 Since $\Beta^+$ is a symmetric derivation of $(\ggo,\mu)$, $\ug_\mu = \ker \Beta^+$ is a Lie subalgebra. 
It follows that $\ggo = \ug \oplus \ngo$, where $\ug = h^{-1} ( \ug_\mu)$ is a Lie subalgebra of $(\ggo,\hml)$, $\ngo = h^{-1} (\ngo_\mu)$ its nilradical, 
and $\la \ug,\ngo\ra_g = 0$. 

Now let $U$, $N$ be the connected Lie subgroups of $G$ with Lie algebras $\ug$, $\ngo$, respectively. Since $G$ is simply connected, it follows from \cite[Lemma 3.18.4]{Varad84} that $G$ is diffemorphic to the product manifold $U \times N$. In particular, $U$ is a closed subgroup of $G$. This implies that $\bar H \subset U$, thus $\hg^\perp$ is $\ip_g$-orthogonal to $\ngo$. Then 
$\la h\hg^\perp, \ngo_\mu\ra = 0$ and from $\ngo_\mu = \Im(\Beta^+)$ 
it follows that $\Beta^+ |_{h \hg^\perp} = 0$. 
Recall also that $(\Beta_\mg)^+ = b \cdot (\Beta^+ |_\mg)$ for some $b>0$, by Lemma \ref{lem_betamgplusbetaplus}. 
Hence for any $X\in \hg^\perp$ we get
\begin{equation}\label{eqn_riceq0}
	0 = \big\la (\Beta_\mg)^+ (hX), hX   \big\ra = \la \Beta_\mg (hX), hX \ra + \Vert \Beta_\mg \Vert^2  \Vert hX \Vert^2.
\end{equation}
On the other hand, using that $\ad_\mumg(h Z) \in \sog(\mg,\ip)$ for all $Z\in \hg^\perp$, we have 
\[
	\big\la  S(\ad_{\mumg} \mcv_\mu) h Z, hZ \big\ra = \big\la  (\ad_{\mumg} \! \mcv_\mu) h Z, hZ \big\ra = \big\la \mcv_\mu, (\ad_\mumg (hZ) )^t hZ\big\ra = 0.
\] 
Since $\Beta_\mg = \Riccim_\mu = \Ricci_\mu + S(\ad_\mumg \mcv_\mu) = h \Ricci_\hml^g h^{-1} + S(\ad_\mumg \mcv_\mu) $  we deduce that 
\begin{equation}\label{eqn_ricgeq0}
	\la \Beta_\mg (hX), hX\ra =  \la h \Ricci_\hml^g h^{-1} h X, hX \ra = g(\Ricci_\hml^g X, X) =  \ric_g(X,X) \geq 0.
\end{equation}
From \eqref{eqn_riceq0} and \eqref{eqn_ricgeq0} one obtains $\Vert \Beta_\mg \Vert^2  \Vert hX \Vert^2 = 0$ for all $X\in \hg^\perp$, which is a contradiction because $\Beta_\mg\neq 0$. Hence $H$ is closed in $G$ and the theorem follows.
\end{proof}

\begin{remark}\label{rmk_solitonbracket}
It follows from Theorem \ref{thm_lochomsolitons} that the condition of being locally isometric to a soliton may be expressed just in terms of the bracket $\mu \in \Vg$ and the fixed background scalar product $\ip$ on $\ggo$. Indeed, with that information one can compute $\Riccim_\mu$, the stratification of $\Vg$, and then check whether the two conditions $\mu \in \sca_\Beta$ and $\Riccim_\mu = \Beta_\mg$ are satisfied.  
\end{remark}

\begin{remark}\label{rem_bracketmin}
The previous results show that solitons with $\mu\in\sca_\Beta$ and $\scalm = -1$ minimize the norm of the modified Ricci curvature among all homogeneous metrics with brackets in $\sca_\Beta$ and $\scalm = -1$. These minima may come in families:
in \cite{finding}, a curve of pairwise non-isometric nilsolitons is given; the corresponding brackets belong to $\sca_\Beta$ where $\Beta = \diag(1,2,3,4,5,6,7) \in \Sym(\RR^7)$. The first such example is an 84-dimensional family of pairwise non-isometric nilsolitons in \cite{Heber1998}, containing the nilsoliton corresponding to the hyperbolic Cayley  plane $M^{16} = \CC a H^2$. They belong to $\sca_\Beta$, with $\Beta = \diag(1,\ldots,1,2,\ldots,2)$ (eight $1$'s and seven $2$'s).
\end{remark}



\section{A Lyapunov function for immortal homogeneous Ricci flows}\label{sec:lyap}

Let $(M^n,\ggl)$ be a $\ggl$-homogeneous space and let
$(g(t))_{t\in [0,\infty)}$ be a non-flat immortal Ricci flow solution
of $\GG$-invariant metrics. After fixing a $\GG$-homo\-geneous 
background metric $\bar g$ on $M^n$,  we considered in Section \ref{sec_BF}
a corresponding bracket flow solution
 $(\mu(t))_{t\in [0,\infty)}$ within an orbit $\OlH$ in the space of brackets $\Vg$. 
 
 Let $\Beta \in \Symg$ such that $\hml \in \sca_\Beta \cap \Vnn$ and
 consider the closed subgroup of $\GHm$ given by
\[
	Q_\Betam^H := Q_\Beta \cap \GHm. 
\]
Recall that $\GHm = \OHm \, Q_\Betam^H$ by (\ref{eqn_GHmdecomp}). Thus, by 
\mbox{Corollary \ref{cor_gauge} }
and Lemma \ref{lem_gaugeohm}
we can associate to the above Ricci flow solution a \emph{gauged} unimodular bracket flow solution $(\mub(t))_{t\in[0,\infty)}$ satisfying \eqref{eqn_QgaugedBF}, and lying on the smaller orbit
\[
	(\mub(t))_{t\in[0,\infty)} \, \subset  \, Q_\Betam^H \cdot \hml  \,  \subset  \, \sca_\Beta \cap \Vnn.
\]

The main result of this section is the following

\begin{theorem}\label{thm_lyapunov}
Let $(g(t))_{t\in [0,\infty)}$ be a non-flat Ricci flow solution
of $\GG$-invariant metrics on a $\ggl$-homogeneous space $(M^n,\ggl)$, and let $\Beta \in \Symg$ such that $\hml \in \sca_\Beta \cap \Vnn$. Then, the scale-invariant function
\[
 F_\Betam : Q_\Betam^H \cdot \hml \to \RR \, ; \qquad
 	\mu \mapsto v_\Betam(\mu)^2 \cdot \scalm( \mu ),
\] 
evolves along the gauged unimodular bracket flow $(\mub(t))_{t\in[0,\infty)}$  by
\begin{equation}\label{eqn_Fbetanondec}
    \ddtbig \,  F_\Betam(\bar \mu)\, = \, 
     2 \cdot v_\Betam(\bar \mu)^2 \cdot \left( \left\Vert {\Riccim_{\bar \mu} } \right\Vert ^2  + \scalm({\bar \mu}) \cdot \, \left\la \Riccim_{\bar \mu}, \Betam \, \right\ra \right) \geq 0\,.
\end{equation}
Equality holds for some $t>0$ if and only if $(M^n,g(0))$ is 
locally isometric to a non-flat, expanding, $\ggl$-homo\-geneous Ricci soliton.
\end{theorem}

For the definition of the \emph{$\Betam$-volume functional} $v_\Betam$ we refer
to Lemma \ref{lem_vbetamu}. 
By \emph{scale-invariance} we mean of course with respect to the \emph{geometric scaling} on brackets, induced from the scaling of the associated Riemannian metrics: see \mbox{Definition \ref{def_scaling}} below. Notice that $ v_\Betam(c \cdot \mu)=c^{-1} \cdot v_\Betam(\mu)$, for
any $c>0$.

\begin{remark}\label{rem_Qbm}
Let  $\Betam$ have pairwise different eigenvalues $\beta_1, \ldots, \beta_s \in \RR$
with multiplicities $m_1, \ldots, m_s\geq 1$ and corresponding eigenspaces $\mg_1, \ldots, \mg_s$. 
Then, of course
 $\sum_{i=1}^s m_i=n$ and $\tr \Betam=\sum_{i=1}^s m_i \Beta_i=-1$
 by (\ref{eqn_defbetam}).
Similar to the decomposition of $\GG_\Beta$ given in (\ref{eqn_Gbetadecomp}),
 we have
\[
   Q_\Betam^H =\GG_\Betam^H \cdot U_\Betam^H \simeq \big(\GH(\mg_1) \times \cdots \times \GH(\mg_s)\big) \cdot U_\Betam^H\,,
\]
where $\GG_\Betam^H=\GG_\Beta \cap \GHm$, $U_\Betam^H=U_\Beta \cap \GHm$, since
$\Beta_\mg$ commutes with $\Ad(H)\vert_\mg$ by Lemma \ref{lem:betacommuteAdH}: see the proof of
Lemma \ref{lem_gaugeohm}.
 We now set
\[
    \Slbm
     := \Big\{ h  u  \in Q_\Betam^H \, : \,  \prod_{i=1}^s \det(h_i)^{-\Beta_i}=1 \textrm{ and }
       h =(h_1,\ldots,h_s) \in \GG_\Betam^H,\,\,\,u\in U_\Betam^H \Big\}
\]
We clearly have $Q_\Betam^H = \exp(\RR \Betam)\ltimes\Slbm$. Now, the $\Betam$-volume functional
$v_\Betam$ has the property that it is constant along  $\Slbm$-orbits in
$Q_\Betam^H \cdot \hml$ and that it scales correctly along the $\exp(\RR \Betam)$-orbit.
For instance,
if the homogeneous space $M^n$ in question is a semisimple Lie group $\GG$, then
we have $\Beta = \Beta_\mg=-\tfrac{1}{n}\cdot \Id$, $\Slbm=\Sl(n)$, and 
$F_\Beta$ is nothing but the volume-normalized scalar curvature.
\end{remark}

\begin{remark}\label{rem_Ricscalnot}
From Remark \ref{rem_bracketmin} it would seem natural to consider $\frac{\Vert \Ricci \Vert}{\vert \scal \vert}$ as a candidate for a monotone decreasing quantity. However,
on the $3$-dimensional unimodular solvable Lie group $E(2)$, there
exist immortal homogeneous  Ricci flow solutions
where $\frac{\Vert \Ricci \Vert}{\vert \scal \vert}$ is unbounded \cite{IJ92}. Its Lie algebra
 $\mathfrak{e}(2) = \ag \oplus \ngo$ has a $2$-dimensional abelian nilradical $\ngo$, and  $\ag=\Span_\RR\{X\}$
with $A:=\ad(X)\vert_{\ngo}$ being skew-symmetric with respect to some scalar product. By replacing $A$ with
$A_\epsilon=A+\diag(\epsilon,-\epsilon)$ one obtains an non-flat unimodular solvable Lie group, 
such that  $\frac{\Vert \Ricci \Vert}{\vert \scal \vert}$ is still
not monotonously decreasing. There exist also higher-dimensional examples.
\end{remark}

\begin{definition}\label{def_scaling}
For $\mu \in \OlH$ and $c>0$ we define the \emph{scaled bracket} $c\cdot \mu \in \OlH$ by
\[
	c \cdot \mu := (c^{-1} \Id_\mg) \cdot \mu,
\]
where $c^{-1} \Id_\mg$ is considered as in $\Gg$ by trivial extension (see \eqref{eqn_Gminclusion}).
\end{definition}
If $\mu$ is a bracket associated to a metric $g$ (Def.~ \ref{def_bracket}), then $c\cdot \mu$ is associated to the rescaled metric $c^{-2} g$. In the Lie groups case this scaling is nothing but the scalar multiplication in the vector space $\Vg$. In general, we have

\begin{lemma}\label{lem_scaling}
Let $(M^n, g,p,\ggo)$ be a $\ggo$-homogeneous space with associated bracket $\mu$. Then, for the bracket $c\cdot \mu$ associated to $c^{-2} g$, $c>0$, we have that 
\[
(c\cdot \mu)|_{\hg \wedge \hg} = \mu|_{\hg \wedge \hg}, \qquad (c\cdot \mu)|_{\hg \wedge \mg} = \mu|_{\hg \wedge \mg}, \qquad (c\cdot \mu)|_{\mg\wedge \mg} = c^2\, \mu_\hg + c \, \mu_\mg,
\]
where the map on the right-hand-side is scalar multiplication in $\Vg$, and $(\cdot)_\hg$, $(\cdot)_\mg$ denote the projections with respect to $\ggo = \hg \oplus \mg$ (see \eqref{eqn_muhgmg}). 
\end{lemma}

In the following we assume that $(M^n,\ggo)$ is a non-flat $\ggo$-homogeneous space.
Recall that $Q_\beta = G_\beta U_\beta$, where $G_\beta$ is the centralizer of $\beta$ in $\Gg$, and that the group $G_\beta$ is the direct product of its normal subgroups $\exp(\RR \beta)$ and $H_\beta$. Consider now the subgroups $\Slb := H_\beta U_\beta \subset Q_\beta$ and
\[
	\Slbm := \Slb \cap \GHm \subset Q_\Betam^H\,,
\]
where $\Slbm$ is indeed just the group described in Remark \ref{rem_Qbm}: see also Lemma \ref{lem_betamgplusbetaplus}. Notice that  $Q_\Betam^H = \exp(\RR \Betam) \ltimes \Slbm$.

\begin{lemma}\label{lem_vbetamu}
For each $\mu \in Q_\Betam^H \cdot \hml$ there
 exists a unique $v_\Betam(\mu) > 0$ 
 such that $v_\Betam(\mu) \cdot \mu \in \Slbm \cdot \hml$. We call $v_\Betam(\mu)$ the \emph{$\Betam$-volume} of $\mu$. The function
\[
  v_\Betam : Q_\Betam^H \cdot \hml \to \RR_{>0}
\] 
has the following properties:
\begin{enumerate}
 \item[(a)] We have $v_\Betam(c \cdot \mu) = c^{-1} v_\Betam(\mu)$, for each $c>0$.
 \item[(b)] We have $v_\Betam( (\exp(a\Betam)\bar h) \cdot \hml  ) = e^{- a \Vert \Betam \Vert^2}$, for
 $a \in \RR$ and $\bar h \in\Slbm$.
 \item[(c)] The function $v_\Betam$ is smooth, $\Slbm$-invariant, and for any $A \in \qg_\Betam^H$
 we have
	\[
		({\rm d} \, v_\Betam )_\mu  \, \left( \pi(A) \mu \right)\,  = \, -\la A, \Betam\ra \, v_\Betam(\mu)\,.
	\]
\end{enumerate} 
\end{lemma}

\begin{proof}
By the above we have $\mu = (\exp(a \Betam) \bar h) \cdot \hml$ for
$a \in \RR$ and $\bar h\in \Slbm$ and  $\Betam=-\Vert \Betam\Vert^2\Id_\mg+(\Betam)^+$
by \eqref{eqn_defbetamplus}. Thus, 
\mbox{$\mu  = e^{a \Vert \Betam \Vert^2 } \cdot ( ( \exp(a (\Beta_\mg)^+) \bar h)\cdot \hml )$}
by Definition  \ref{def_scaling}.
Since $\Beta^+ \in \hg_\Beta$ by Remark \ref{rem_git2},
we deduce from Lemma \ref{lem:betacommuteAdH} that
$\exp((\Beta_\mg)^+) \in H_\Beta \cap \GHm \subset \Slbm $.
Thus,  $c\cdot \mu \in \Slbm \cdot \hml$ for $c = e^{-a \Vert \Betam\Vert^2} $. 
 This shows the existence of $v_\Betam(\mu)>0$. To prove uniqueness
suppose that  we have $ (c_1 \Id_\mg) \cdot \mu = \bar h_1 \cdot \hml$ and $(c_2 \Id_\mg) \cdot \mu = \bar h_2 \cdot \hml$ for $c_1, c_2 > 0$  and $\bar h_1, \bar h_2 \in \Slbm$. Then
$(c_2 c_1^{-1} \bar h_2^{-1} \bar h_1) \in \Aut(\hml) \subset \Slbm$ by Corollary \ref{cor_autmu}. Thus, for $c := c_2 c_1^{-1}$ we have $c \, \Id_\mg \in \Slbm$, hence $c=1$: see Remark \ref{rem_Qbm}.

To show (a), notice that $
 \big(c^{-1} v_\Betam(\mu)\big) \cdot (c \cdot \mu) = v_\Betam(\mu) \cdot \mu \in \Slbm \cdot \hml$, 
and by uniqueness $v_\Betam(c \cdot \mu) = c^{-1} v_\Betam(\mu)$. 
Item (b) was shown already.
We now prove (c). Smoothness and $\Slbm$-invariance follows from (b) and  
$Q_\Betam^H = \exp(\RR \Betam) \ltimes \Slbm$. Finally, the formula for the differential is clear if $A\in \slbm$, i.e. $\la A, \Betam \ra = 0$, by $\Slbm$-invariance. 
And for $A = \Betam$ it follows immediately from (b).
\end{proof}

Next, we show that on the gauged orbit $ Q_\Betam^H \cdot \hml$ the $\Betam$-volume is controlled by
the norm of the bracket.

\begin{lemma}\label{lem_lowbdvbeta}
There exists $C_\hml >0$ such that for all $\mu \in Q_\Betam^H \cdot \hml$ we have
\[
	v_\Betam(\mu) \cdot \left(\Vert \mumg\Vert + \Vert \mu_\hg\Vert^{1/2} \right) \geq C_\hml > 0.
\]
\end{lemma} 

\begin{proof}
Using Lemma \ref{lem_scaling} and Lemma \ref{lem_vbetamu}, we see that under the geometric scaling 
$\mu \mapsto c \cdot \mu$ the left-hand-side is scale invariant. To show the claim, we argue by contradiction, and assume that there is a sequence $(\mu_k)_{k\in \NN} \subset Q_\Betam^H \cdot \hml$ such that 
\[
	\Vert (\mu_k)_\mg \Vert + \Vert (\mu_k)_\hg\Vert^{1/2} \equiv 1, \qquad v_\Betam(\mu_k) \underset{k\to\infty}\to 0\,.
\]
Let $\mub_k := v_\Betam(\mu_k) \cdot \mu_k \in \Slbm \cdot \hml$. Notice that again by Lemma \ref{lem_scaling} the projections of $\mub_k$ to the subspaces $\Lambda^2 (\hg^*) \otimes \hg$ and $(\hg^* \otimes \mg^*) \otimes \mg$ of $\Vg$ are the same as those of $\mu_k$, while the other two projections $(\cdot)_\mg$ and $(\cdot)_\hg$ go to zero. Thus, $\mub_k \to \mub_\infty$ for some $\mub_\infty \in \Vg$ satisfying
\begin{equation}\label{eqn_muinftyalmostzero}
	(\mub_\infty)_\mg = (\mub_\infty)_\hg = 0.
\end{equation}
By Proposition \ref{prop_propertiesgroups}, for the
 orthogonal projection $p_\Beta : \Vnn \to \Vzero$ we have the identity
$p_\Beta(\mu)=\lim_{t \to \infty}\exp(-t\Beta^+)\cdot \mu$. 
Since by Lemma \ref{lem:betaplusimage} the image of $\Beta^+$ is a subspace of $\mg$,
for each $\mu \in \Lambda^2 (\hg^*) \otimes \hg$ we have $p_\Beta(\mu)=\mu$.
Regarding $\mu \in (\hg^* \otimes \mg^*) \otimes \mg$, for $Z \in \hg$ and $X,Y \in \mg$ we have
\[
 \big\la (\exp(t\Beta^+)\cdot \mu)(Z,X),Y \big\ra=
  \big\langle \big(\exp(t\Beta^+)\ad_\mu(Z)\exp(-t\Beta^+)\big) (X),Y\big\rangle \,.
\]
Now, if $\mu = h\cdot \hml$ with $h\in \GHm$, then $\ad_{\mu}(Z)=h \ad_{\hml}(Z )h^{-1} = \ad_\hml (Z)$
commutes with $\Beta$ by Lemma \ref{lem:betacommuteAdH}.
Thus in the above equation the $\exp(t\Beta^+)$-terms cancel out.
As a consequence we have  $p_\Beta(\mub_\infty) = \mub_\infty$. Moreover,
denoting by  $\lambda_k := p_\Beta(\mub_k)$ we deduce 
$\lim_{k \to \infty}\lambda_k = \mub_\infty$. Since $\bar \mu_k \in H_\Beta U_\Beta \cdot \hml$ we have that $\lambda_k \in H_\Beta \cdot \lambda^\ggo$ by Proposition \ref{prop_propertiesgroups}, where $\lambda^\ggo := p_\Beta (\hml)$.


Finally, choose $E   = \Id_\hg +x \Id_\mg\in \glg(\ggo)$ for $x\in \RR$, such that $E\perp \Beta$.
This is possible by Theorem \ref{thm_refinedbeta}, since $\Beta \vert_\hg= - c \cdot \Id_\hg$ for some
$c>0$ and $\tr_\mg \Beta|_\mg < 0$ by Lemma \ref{lem_trbetam}. 
We clearly have $E\in \hg_\Beta$, and $\exp(t E) \cdot \mub_\infty \to 0$ as $t\to \infty$ by \eqref{eqn_muinftyalmostzero}. Therefore, $0\in \overline{H_\Beta \cdot \lambda^\ggo}$, and this contradicts the fact that $\hml \in \sca_\Beta$: see Section \ref{sec_stratif} (Definition \ref{def_Vss} and \eqref{eqn_stratumbeta}).
\end{proof}

\begin{proof}[Proof of Theorem \ref{thm_lyapunov}]
The scale invariance of $F_\Betam$ is a direct consequence of the identities 
$\scalm(c \cdot \mu) = c^2 \scalm(\mu)$ and $v_\Betam(c \cdot \mu) = c^{-1} v_\Betam(\mu)$
for any $c>0$: see \mbox{Lemma \ref{lem_vbetamu}.} 
To obtain the evolution equation, recall that the vector field giving the evolution equation for the $Q_\Betam^H$-gauged unimodular bracket flow \eqref{eqn_QgaugedBF} is of the form $-\pi(A) \mu$, where $A := (\Riccim_\mu)_{\qg_\Betam^H}$ is the modified Ricci curvature projected non-orthogonally onto $\qg_\Betam^H$. Thus along that flow we have by Lemma \ref{lem_vbetamu}, (c) that
\[
	\ddt v_\Betam(\mu) = 
	\big\la A, \Betam \big\ra v_\Betam(\mu) 
	= \left\la \Riccim_\mu, \Betam \right\ra v_\Betam(\mu)\,,
\]
since $\Riccim_\mu$ and $A$ differ by a  skew-symmetric endomorphism according to \eqref{eqn_gkq}.
 Since the modified scalar curvature is $\OHm$-invariant, by Corollary \ref{cor_gauge} its evolution equation along the gauged unimodular bracket flow is the same as along the ``un-gauged'' flow, 
 computed in Lemma \ref{lem_scalmev}.
The evolution equation for $F_\Betam$ now follows.

Regarding monotonicity, recall that  along a
non-flat immortal solution the modified scalar curvature is negative by Lemma \ref{lem_scalmev}.
Consequently we may apply 
Corollary \ref{cor_mainestimate}, which together with Cauchy-Schwarz inequality imply that
\begin{equation}\label{eqn_proofLyap}
	 \Vert {\Riccim_{\mu} } \Vert ^2  + \scalm({\mu}) \cdot \, \la \Riccim_{\mu}, \Betam \, \ra \geq  
	 \Vert \Riccim_\mu \Vert \cdot \left(  \Vert \Riccim_\mu \Vert - \vert { \scalm(\mu)} \vert \cdot \Vert \Betam \Vert \right) \geq 0\,.
\end{equation}
Finally, the equality rigidity follows from Theorem \ref{thm_lochomsolitons}.
\end{proof}



\section{Convergence to an expanding soliton}\label{sec:conv}

Let $(g(t))_{t\in[0,\infty)}$ be an immortal homogeneous Ricci flow solution
on a $\ggl$-homo\-ge\-neous space $(M^n,\ggl)$. To understand its long-time behaviour it is natural to consider for each $s>0$ the \emph{parabolic blow-downs}
\[
     g_s(t):=\tfrac{1}{s}\cdot g(s\cdot t)\,.
\]
The parabolic nature of the Ricci flow equation implies that these are again Ricci flow solutions. They are of course also immortal and homogeneous. A time interval $[a,b]$ for $g_s$ corresponds to $[sa, sb]$ for $g$, thus the long-time behaviour of $g(t)$ is reflected by the behavior of $g_s$ on a fixed interval of time as $s\to \infty$.

It was proved in \cite{Bhm2014} that such solutions are necessarily of \mbox{Type III},
that is, there exists a constant
$C_{g_0}>0$ such that $\Vert{ \Riem_{g(t)} }\Vert_{g(t)} \cdot t \leq C_{g_0}$
for all $t \geq 0$. An immediate application of this estimate is the fact that
 the parabolic blow-downs have uniformly bounded curvature tensor:
\begin{equation}\label{eqn_TypeIIIcurvbound}
	\sup_{t\in [1,\infty)} \Vert {\Riem_{g_s(t)} } \Vert_{g_s(t)} \leq  C_{g_0}, \qquad \mbox{ for all } s>0.
\end{equation}
Assuming that we have in addition a uniform lower bound on the injectivity radius,
for any sequence $(s_k)_{k\in\NN}$ converging to infinity 
the sequence $(M^n,(g_{k}(t)\vert_{[0,\infty)},p)_{k\in \NN}$ 
subconverges by Hamilton's compactness theorem in  pointed $C^\infty$-Cheeger-Gromov topology
to a limit Ricci flow solution $(M^n_\infty,(g_\infty(t))\vert_{(0,\infty)}, p_\infty)$,
where we have set $g_k(t):=g_{s_k}(t)$.
By definition, this means that  there exists an exhaustion
$\{U_k\}_{k \in \NN}$ of open sets of $M^n_\infty$ and diffeomorphisms
$\varphi_k:U_k \to V_k \subset M^n$ with
$\varphi_k(p_\infty)=p$ such that the pulled-back metrics
$\varphi^*_k (g_{k}(t)\vert_{V_k})$ on $U_k$ converge to $g_\infty(t)$ in
$C^\infty$-topology, uniformly on compact subsets of $(0,\infty)\times M^n_\infty$.

But even without a lower bound for the injectivity radius
one can obtain a limit Ricci flow, in the sense of Riemannian groupoids:
see \cite{Lot07}.
For homogeneous spaces this is much easier and can be explained as follows. Using the uniform
curvature estimates of $g_k(1)$, which we can assume to hold for $C_{g_0} = 1$, we
can pull back the complete metrics $g_k(t)$ by the Riemannian exponential map
$\exp_{p}(g_{k}(1))$ to the ball $B_\pi(0_p)$ of radius $\pi$ in the tangent space
$(T_{p}M^n,g_k(1))$. At time $t=1$, clearly the injectivity radius of this pull back at $0_{p}$ is $\pi$. Again, by a local version of Hamilton's compactness
theorem one obtains a limit Ricci flow solution.

We are now in a position to prove the main result of this paper

\begin{theorem}\label{thm_conv}
Let $(M^n,(g(t))_{t\in [0,\infty)})$ be an immortal homogeneous
Ricci flow solution of complete metrics.
Then, for any sequence $(s_k)_{k\in \NN}$ converging to infinity
the corresponding sequence of blow-downs $(M^n, (g_{s_k}(t))_{t \in [0,\infty)} )_{k \in \NN}$ subconverges, in the sense of Riemannian groupoids, to a locally homogeneous Ricci flow solution $(M^n_\infty,g_\infty(t))_{t \in (0,\infty)}$, which is locally 
isometric to a globally homogeneous expanding Ricci soliton.
\end{theorem}

\begin{proof}
The globally homogeneous space $(M^n,g(0))$ is a $\ggl$-homogeneous space for
$\ggl=\isog(M^n,g(0))$. Moreover, by Theorem \ref{thm:gRicflow} 
all the evolved metrics are $\GG$-invariant. If the initial metric is flat, then so is the solution and we are done of course. Hence we may assume that $(M^n,g(0))$ is not flat, which implies that
$(M^n,\ggo)$ is a non-flat $\ggo$-homogeneous space.

In a first step we assume that the immortal solution $(g(t))_{t \in [0, \infty)}$
is non-collapsed, meaning that there exists $\delta>0$ such that  $\inj(\tfrac{1}{t}g(t)) \geq \delta$, for 
any $t \geq 1$.
Let $(s_k)_{k \in \NN}$ be any sequence with $\lim_{k \to \infty} s_k=+\infty$.
By \eqref{eqn_TypeIIIcurvbound} we are in a position
to apply Hamilton's  compactness theorem \cite{Ham95b}, see also
Theorem 6.35 in \cite{ChowLuNi}:
this yields a subsequence, again denoted by 
$(s_k)_{k \in \NN}$, such that the sequence of pointed immortal Ricci flow solutions 
$\big(M^n, \allowbreak (g_{s_k}(t))_{t \in [0,\infty)},p\big)$ converges in Cheeger-Gromov sense 
to an immortal limit Ricci flow solution
$\big(M^n_\infty,(g_\infty(t))_{t \in (0,\infty)},p_\infty \big)$, 
which is still of \mbox{Type III}. By \cite{PTV} it is also homogeneous.
Notice also that this limit solution is still $\delta$-non-collapsed, that is 
$\inj(\tfrac{1}{t}g_\infty(t)) \geq \delta $ for all $t>0$, 
since $\tfrac{1}{t}g_\infty(t)=\lim_{k \to \infty}\tfrac{1}{s_k t}g(s_kt)$.
Here and in the following we are suppressing the intertwining diffeomorphisms
coming from the Cheeger-Gromov convergence.

For a fixed given sequence $(s_k)_{k\in \NN}$ we now pick 
a subsequence $(s_{k_l})_{l \in \NN}$, such that the dimension $N_\infty$ of the isometry group
$\ggl_\infty$ of the limit space is maximal among all possible limits.
We denote by $(M^n_\infty,(g_\infty(t))_{t \in (0,\infty)},p_\infty,\ggl_\infty)$ the corresponding
limit Ricci flow solution of $\GG_\infty$-homogeneous metrics. 

Since $(g_\infty(t))_{t \in (0,\infty)}$
is a non-collapsed immortal homogeneous Ricci flow solution of Type III,
for any sequence of times $(t_k)$ with $t_k \to \infty$ we can pick a further subsequence of 
blow-downs $\big( M^n_\infty,(g_\infty)_{t_k}(1) \big) = \big( M^n_\infty, \tfrac{1}{t_k} g_\infty(t_k) \big)$ converging to a second
limit space $\big( \tilde M^n_\infty,\tilde g_\infty\big)$ as above.
Notice that for each fixed $t_k$ we know that $\lim_{l \to \infty}\frac{1}{s_l}\cdot g(s_l\cdot t_k)=g_\infty(t_k)$. Hence we may choose $s_{l_k}$ so large, that 
$\frac{1}{s_{l_k}\cdot t_k}\cdot g(s_{l_k}\cdot t_k)$ and $\tfrac{1}{t_k}\cdot g_\infty(t_k)$ are $\tfrac{1}{t_k}$-close in $C^\infty$-topology. By
setting $\tilde s_k=s_{l_k}\cdot t_k$ we see that
$(M^n,g_{\tilde s_k}(1))_{k \in \NN}$ converges to $(\tilde M^n_\infty,\tilde g_\infty)$ for $k\to \infty$. In other words, this second limit space is also a limit of the initial sequence of blow-downs. It suffices then to show that $\big( \tilde M^n_\infty, \tilde g_\infty \big)$ is locally isometric to a homogeneous Ricci soliton.

Let $\ggl_\infty=\hg_\infty \oplus \mg_\infty$ be the canonical decomposition of 
$\ggl_\infty$.
 By Theorem \ref{thm:Riccibracket}, after chosing a $\GG_\infty$-homogeneous
background metric $\bar g_\infty$ on $M^n_\infty$, 
inducing a scalar product $\ip_\infty$ on $\mg_\infty$ and on $\ggo_\infty$, we obtain a corresponding bracket flow solution 
$(\mu_\infty(t))_{t \in (0,\infty)}$ on  $V(\ggo_\infty)$. 
We claim that
\begin{equation}\label{eqn_boundnormmu}
    \sup_{t\in [1,\infty)} \big\Vert (\mu_\infty(t))_{\mg_\infty} \big\Vert_\infty +  \big\Vert (\mu_\infty(t))_{\hg_\infty} \big\Vert_\infty^{1/2} \leq C \cdot \tfrac{1}{\sqrt{t}}, 
\end{equation}
for some constant $C>0$, where $\ip_\infty$ is the scalar product on $V(\ggo_\infty)$  
defined in (\ref{def_scal}), and $(\mu_\infty)_{\hg_\infty}$, $(\mu_\infty)_{\mg_\infty}$ denote the components of the brackets restricted to $\mg_\infty \wedge \mg_\infty$ (see \eqref{eqn_muhgmg}). 
If this is not the case, then
there must be sequence $t_k \to \infty$ with 
\mbox{$\big\Vert \sqrt{t_k} \cdot  \mu_\infty(t_k) \big\Vert_\infty \to \infty$} (cf. Def.~ \ref{def_scaling}). 
Arguing as in the previous paragraph,
we pick a convergent subsequence of 
blow-downs $(M^n_\infty, \tfrac{1}{t_k} g_\infty(t_k))$ converging to $ (\tilde M^n_\infty,\tilde g_\infty)$ for $k \to \infty$.
Notice that for each $k$ the bracket $\sqrt{t_k} \cdot \mu_\infty(t_k)$ corresponds to $\tfrac{1}{t_k} g_\infty(t_k)$, since we are using on brackets the scaling coming from the scaling on metrics.
Corollary \ref{cor:algcollapse} implies now that 
$\dim (\isog(\tilde M^n_\infty,\tilde g_\infty)) > N_\infty$. 
By the very definition of $N_\infty$ this is however impossible, thus \eqref{eqn_boundnormmu} holds.

As at the beginning, we may assume that the limit solution
$(g_\infty(t))_{t \in (0,\infty)}$ is non-flat.
 As explained in Sections \ref{sec_newcurv} and \ref{sec:lyap}, there exists $\Beta\in \Symg$ such that $\bar\mu_\infty(t) \in \sca_\Beta$
for all $t \geq 0$, where $\bar \mu_\infty(t)$ denotes the corresponding solution to the gauged unimodular bracket flow \eqref{eqn_QgaugedBF}.

Consider now the scale-invariant Lyapunov function $F_\Betam(\mu)=v_\Betam(\mu)^2 \cdot \scalm(\mu)$ from Theorem \ref{thm_lyapunov} 
and set $f(t):=F_\Betam(\bar \mu_\infty(t))$ -- we write $\mg$ instead of $\mg_\infty$ to simplify notation. 
The function $f$ is non-decreasing
by Theorem \ref{thm_lyapunov}  and bounded: see Lemma \ref{lem_scalmev}
and Lemma \ref{lem_vbetamu}. Next, 
let $g(s)=f(e^s)$, say for $s\geq 0$. Then $g$ is still bounded and non-decreasing.
Thus there exists a sequence of times $s_k \to \infty$ with $g'(s_k) \to 0$.
Setting $t_k:=e^{s_k}$ and $\mu(t):=\sqrt{t}\cdot \bar \mu_\infty(t)$,
and using that the right hand side of \eqref{eqn_Fbetanondec} scales like
$Q(c\cdot \bar\mu) = c^2 Q(\bar \mu)$,
it follows from \eqref{eqn_Fbetanondec} and \eqref{eqn_proofLyap}  that
\begin{equation}\label{eqn_derivLyapzero}
	v_\Betam(\mu_k)^2 \cdot \big\Vert {\Riccim_{\mu_k} } \big\Vert \cdot \left(  \big\Vert {\Riccim_{\mu_k} } \big\Vert - \vert { \scalm(\mu_k)} \vert \cdot \Vert \Betam \Vert \right) \underset{k\to \infty}\longrightarrow 0,
\end{equation}
where $\mu_k := \mu(t_k)$. 

By (\ref{eqn_boundnormmu}) we have $\Vert (\mu_k)_{\mg}\Vert_\infty + \Vert (\mu_k)_{\hg_\infty}\Vert_\infty^{1/2} \leq C$ for large $k$. The rest of the components of the brackets remain unchanged along a bracket flow solution, thus the full norm $\Vert \mu_k \Vert_\infty$ is also bounded. Hence by compactness we may assume that $\mu_k \to \tilde\mu_\infty \in \Vg$ as $k\to \infty$.

Observe that \eqref{eqn_boundnormmu} and Lemma \ref{lem_lowbdvbeta} imply that $(v_\Betam(\mu_k))_k$ is uniformly bounded below by a positive constant.
 Assume that $\Riccim_{\tilde\mu_\infty} \neq 0$. From \eqref{eqn_derivLyapzero} we get that 
\begin{equation}\label{eqn_Riccilimit}
	\big\Vert {\Riccim_{\tilde \mu_\infty} }\big\Vert = \big\vert {\scalm( \tilde \mu_\infty)} \big\vert \cdot \Vert \Betam \Vert,
\end{equation}
 and $\scalm(\tilde \mu_\infty)<0$ (in particular, $\tilde \mu_\infty \neq 0$). We now claim that $\tilde \mu_\infty \in \sca_\Beta$. Suppose on the contrary that $\tilde \mu_\infty \notin \sca_\Beta$. By the Stratification Theorem \ref{thm_stratifb} we must have that $\tilde \mu_\infty \in \sca_{\tilde \beta}$ with $\Vert \tilde \beta \Vert > \Vert \Beta \Vert$, and this implies that   $ \Vert \tilde\Beta_\mg \Vert > \Vert \Beta_\mg \Vert$ by Remark \ref{rmk_normbetam}.
 However, Corollary \ref{cor_mainestimate} applied to $\tilde \mu_\infty \in \sca_{\tilde \beta}$ yields an estimate which contradicts \eqref{eqn_Riccilimit}. The claims thus follows. Using that, equality in \mbox{Corollary \ref{cor_mainestimate}} and \mbox{Theorem \ref{thm_lochomsolitons}} imply that $\tilde \mu_\infty$ is a \emph{soliton bracket}, in the sense that 
 \begin{equation}\label{eqn_solitonbracket}
 	\tilde \mu_\infty \in \sca_\Beta \quad \hbox{and} \quad \Riccim_{\tilde \mu_\infty} = \Beta_{\mg}.
 \end{equation}
We now consider the sequence $(M^n_\infty, g_k, p_\infty, \ggo_\infty)$, $g_k := \tfrac{1}{\tau_k} g_\infty(\tau_k)$. The corresponding brackets $\mu_k$ converge to the soliton bracket $\tilde \mu_\infty$. As explained above,
we may assume that $(g_k)_{k \in \NN}$ subconverges to a 
second limit metric $\tilde g_\infty$ on $\tilde M^n_\infty$.
Passing to a further subsequence, by Proposition \ref{prop_limitKf} we may assume that
the sequence $(M^n_\infty, g_k, p_\infty, \ggl_\infty)_{k \in \NN}$ 
converges to $ (\tilde M^n_\infty,  \tilde g_{\infty}, \tilde p_\infty, \tilde \ggl_{\infty})$  in equivariant Cheeger-Gro\-mov topology (see Section \ref{sec_algcon} for the definition). Moreover, by our choice of $N_\infty$ and Lemma \ref{lem_collapsedseq}, the sequence is algebraically non-collapsed. Thus, in the case of $\Riccim_{\tilde \mu_\infty} \neq 0$, it follows from \eqref{eqn_solitonbracket}, Theorem \ref{thm:Liebracketconvergence} and Lemma \ref{lem_solitonbracket} that $(\tilde M^n_\infty, g_\infty)$ is locally isometric to an expanding homogeneous Ricci soliton. 
In the case $\Riccim_{\tilde \mu_\infty} = 0$, an analogous argument using Lemma \ref{lem_Ricciflat} shows that 
$(\tilde M^n_\infty, \tilde g_\infty)$ is flat.

We turn to the case where the given immortal solution $(g(t))_{t \in [0,\infty)}$
is not necessarily non-collapsed. The proof is precisely as above, except that we cannot 
apply Hamilton's compactness theorem directly, since we don't have a uniform lower
bound for the injectivity radius of the metric $g_{s_k}(t)$.

First notice that after rescaling the initial metric $g(0)$ we may assume that we have
$\sup_{t\in [0,\infty)} \Vert {\Riem_{g_{s(t)}} }\Vert_{g_s(t)} \leq 1$, 
for each $s>0$. Then, as explained above,
one can  pull back  the metrics $g_{s_k}(t)$ by the  exponential  map  $\exp_p(g_{s_k}(1))$ to the ball $B_\pi$ or radius $\pi$ in tangent space $T_p M^n$ for each $k \in \NN$.
As a consequence, for each sequence $\{ s_k\}_{k \in \NN}$ with \mbox{$s_k \to \infty$} 
one obtains a pointed $\GG_{k}$-homogeneous Ricci flow solution on
the locally homogeneous space $(B_\pi,0_p,\ggl_k)$, 
again denoted by $(g_{s_k}(t))_{t \in (0,\infty)}$, 
with $\inj_{0_p} (g_s(1)) \geq 3$. 
Then, by the local version of Hamilton's compactness theorem \cite{Ham95b}, see
Theorem 6.36 in \cite{ChowLuNi},
there exists a  subsequence
$\big( (g_{s_{k_l}}(t))_{t\in [1,\infty)} \big)_{l \in \NN}$ of Ricci flow solutions, 
where we again set $g_{s_k}:=g_{s_{k_l}}$,
which converges to a limit Ricci flow solution
$(g_\infty(t))_{t \in [1, \infty)}$ on $B_\pi$. The curvature estimates needed to
apply Theorem 6.36, come from the Shi estimates, applied to a previous time, say $t=0.5$.

The convergence is up to pull-back by $s_k$-dependent diffeomorphisms, uniformly on compact subsets of $B_\pi \times [1,\infty)$. Notice that by \cite{BLS} the limit $(B_\pi,g_\infty(1))$ is a simply-connected, locally homogeneous space, hence by \cite{Nomizu1960} and Theorem 1 of \cite{Nomizu1960}
an $\ggl_\infty$-homogeneous space for a transitive Lie algebra $\ggl_\infty$ of Killing fields. We may assume that $\ggl_\infty := \isog(B_\pi,g_\infty(1))$.
The first remark is that by Theorem \ref{thm:gRicflow} this Ricci flow solution is
uniquely determined by the $\GG_\infty$-invariant initial metric $g_\infty(1)$,
and therefore all the metrics $g_\infty(t)$ are $\GG_\infty$-invariant.
Secondly, it is also of \mbox{Type III} with the same constant $1$, 
since it is the limit of Type III solutions.

As above, for a fixed sequence $(s_k)_{k \in \NN}$ 
we pick now a subsequence $(s_{k_l})_{l\in \NN}$ and a limit solution 
$(g_\infty(t))_{t \in [1,\infty)}$ such that $N_\infty:=\dim (B_\pi,\isog(g_\infty(1)))$ is maximal. 
In order to proceed as above we have 
to take a second limit $\tilde g_\infty=\lim_{k\to \infty }\frac{1}{t_k}g_\infty(t_k)$
for an appropriate sequence $(t_k)_{k \in \NN}$ converging to $\infty$.
As above we use now, that $\frac{1}{t_k}g_\infty(t_k)=\lim_{l \to \infty}\tilde g_l$,
where $\tilde g_l= \frac{1}{s_lt_k} g(s_lt_k)$. Since for the approximating metrics $\tilde g_l$
we still have $\Vert\Riem_{\tilde g_l}\Vert_{\tilde g_l}\leq 1$ 
we obtain as above by pull back of the exponential map $\exp_p(\tilde g_l)$ 
a locally homogeneous space, again denoted by $(B_\pi,\tilde g_l,0_p)$ ,
which is locally isometric to $\tilde g_l$ but has $\inj_{0_p} (\tilde g_l) \geq 3$.
Clearly, the new sequence $(B_\pi,\tilde g_l)_{l \in \NN}$ will subconverge to 
a locally homogeneous limit metric on $B_\pi$, which is locally isometric to 
$\frac{1}{t_k}g_\infty(t_k)$, but has $3$ as a lower bound for the injectivity radius at
the center point of $B_\pi$.
It follows that we may assume that the locally homogeneous metrics
$\frac{1}{t_k}g_\infty(t_k)$ have injectivity radius greater are equal than $3$
for all $k \in \NN$. 
This now, together with the uniform bounds on all covariant derivatives of the
curvature tensor allows us to take a sublimit as above. 
The proof now follows as in the non-collapsed case.
\end{proof}


The following theorem shows that, up to covering, all known immortal
homogeneous Ricci flow solutions are indeed non-collapsed.

\begin{theorem}\label{thm_noncollapsing}
Let $(\GG/H,g)$ be a homogeneous space which is diffeomorphic to $\RR^n$, and such that $\vert K(g)\vert \leq 1$.
Then ${\rm inj}(\GG/H,g)) \geq \tfrac{\pi}{2}$.
\end{theorem}

\begin{proof}
First, notice that we may assume that $G$ acts effectively on $G/H$, and that $\GG$ and
$H$ are connected. Indeed, the connected component $\GG_0$ of $\GG$ containing the
identity still acts transitively on the connected space $\RR^n$.
Then it follows from the long homotopy sequence 
of the fibration $\GG_0\cap H \to \GG_0 \to \RR^n$ that
$\GG_0 \cap H$ must be connected. Secondly, recall that $G/H \simeq \RR^n$ if and only if $H$ is a maximal compact subgroup of $G$ (see for instance \cite{Hch}, \cite{Iws} or \cite[Sections 13.1-13.3]{HlgNeb}). 

Suppose now that the injectivity radius of $(G/H,g)$ 
is strictly smaller than $\tfrac{\pi}{2}$. Homogeneity allows us to assume that this happens at $eH \in G/H$.
By  Klingenberg's Lemma (see \cite[Lemma 8.4]{Petersen06}) there exists a geodesic
loop $\gamma:[0,b] \to \RR^n$ of length $L(\gamma)=b< \pi$ with $\gamma(0)=\gamma(b)=eH$. Let $X_1,\ldots,X_{n-1}$ be Killing fields which span $\gamma'(0)^\perp$. Then, since the scalar product between
a Killing field and $\gamma'$ is a constant function along the geodesic, 
the loop $\gamma$ is in fact a simply closed, hence periodic geodesic, 
with period $b< \pi$.

Let $\ggl=\hg \oplus \mg$ be the canonical decomposition of $\ggl$ and
let $V:(-\epsilon,\epsilon)\to \mg$  be a smooth curve with 
$V(0)=0$ and $V'(0)=N$, $\Vert N \Vert_g = 1$, such that 
 $\gamma(s) = \exp(V(s)) H$ for all $s\in (-\epsilon,\epsilon)$. 
Since $\gamma_s(t):= \exp(-V(s)) \cdot \gamma(s+t)$ is a geodesic variation
of $\gamma(t)$,  $J(t)=\tfrac{d}{ds}\big\vert_{s=0} \gamma_s(t)$
is a Jacobi field along $\gamma(t)$. By the very choice of $V(t)$
we have $J(0)=0$, and by the periodicity of $\gamma$ also $J(b)=0$.
However $b< \pi$ and $\vert K(g)\vert \leq 1$, thus $J \equiv 0$, since by Rauch's comparison theorem there are no conjugate points at distance less than $\pi$.

If $X_N$ denotes the Killing field induced by $N\in \ggo$ on $G/H$, then the vanishing of $J$ is equivalent to
\[
 	0=\tfrac{d}{ds}\big\vert_{s=0}  \exp\big(-V(s)\big) \cdot \gamma(s+t) = - X_N(\gamma(t)) + \gamma'(t).
\]
Hence, $\gamma(t)$ is a periodic 
integral curve of $X_{N}$, and consequently $\gamma(t)=\exp(t\, N) H$ for all $t \in \RR$.

Now recall that $H$ is the isotropy subgroup at $eH$, and
let $Z \in \hg$. As above, the Jacobi field corresponding to the geodesic variation $\tilde\gamma_s(t) := \exp(s Z) \cdot \gamma(t)$ also vanishes. We conclude that the isotropy group $H$ fixes the geodesic $\gamma$. Moreover, since $\gamma(t)=\exp(t\, N) H$ we also deduce that $\Ad(H)$ fixes $N$. Thus,
\[
 \hat H := \{ \exp(t\cdot N):t \in \RR\} \times H
\]
is a compact subgroup of $\GG$ with $\dim \hat H=\dim H+1$.
This is a contradiction.
\end{proof}

We propose the following 

\begin{problem}
Is it possible for a homogeneous space $(\RR^n,g)$ to have a closed geodesic?
\end{problem}

To the best of our knowledge, this is open even for left-invariant metrics on nilpotent Lie groups.

It is well-known that a Cheeger-Gromov limit of homogeneous manifolds is again homogeneous. However, the topology may change in the limit, even the fundamental group, as the example of the Berger spheres on $S^3$ converging to $S^1 \times \RR^2$ shows, see e.g. \cite[Ex.~6.17]{Lauret2012}. The next result, whose proof we postpone to Appendix \ref{app_chgr}, says this does not happen on $\RR^n$:

\begin{theorem}\label{thm_chgr}
Let $(\RR^n, g_k)_{k\in \NN}$ be a sequence of homogeneous manifolds
converging in Chee\-ger-Gromov topology to $(\bar M^n ,\bar g)$. Then, $\bar M^n$ is diffeomorphic to $\RR^n$.
\end{theorem}




\section{Equivariant convergence of homogeneous spaces}\label{sec_algcon}

In this section we study sequences $(M^n_k, g_k, p_k,\ggl_k)_{k\in \NN}$
of $\ggl_k$-homogeneous spaces converging in Cheeger-Gromov topology  to
a Riemannian manifold $(M^n_\infty, g_\infty, p_\infty)$, which by
\cite{Singer1960} and \cite{NicolodiTricerri1990} is again locally
homogeneous (see also \cite{PTV}). Following \cite{Heber1998}, we will
explain how to obtain a limit Lie algebra $\ggl_\infty$ of
$g_\infty$-Killing fields on $M_\infty^n$, thus arriving at the notion
of \emph{equivariant Cheeger-Gromov convergence} of homogeneous spaces:
see Definition \ref{def_algconv}. We then distinguish between two cases
according to whether $\ggl_\infty$ is \emph{transitive} on $M^n_\infty$,
that is the Killing fields in $\ggl_\infty$ span the whole tangent space
at $p_\infty$, or not. The sequence is called \emph{algebraically
non-collapsed} and \emph{collapsed}, respectively.

\begin{example}\label{ex_algcollapse}
It was shown by Lott in \cite{Lot07} that for any left-invariant
initial metric on the universal cover of $\Sl(2,\RR)$,
any sequence of Ricci flow blow-downs will converge
to a unique limit solution, which is the Riemannian product of
the hyperbolic plane $H^2=\Sl(2,\RR)/\SO(2,\RR)$ and a flat factor $\RR$.
For this limit space $(M^3_\infty,g_\infty,p_\infty)$,
the Lie algebra of the full isometry group
equals $\slg(2,\RR) \oplus \RR$, whereas the approximating Lie
algebras of
Killing fields are isomorphic to $\slg(2,\RR)$ for generic
left-invariant initial metrics. Clearly, algebraic collapse
occurs in this situation: the Lie algebra of Killing fields
$\slg(2,\RR)$ is transitive on all the spaces of the sequence, but it
only spans a subspace of dimension two of $T_{p_\infty} M_\infty^3$ (cp.~\cite[$\S$ 4.3 (iv)]{homRF}).
\end{example}

The main theorem in \cite{Lauret2012} shows that convergence of the Lie
brackets together with a Lie-theoretical non-collapsedness hypothesis,
imply subconvergence in Cheeger-Gromov topology. The main result of this
section is the following converse assertion.

\begin{theorem}\label{thm:Liebracketconvergence}
Suppose that the sequence $(M^n_k, g_k, p_k,\ggl_k)_{k\in \NN}$
converges to the limit space
$(M^n_\infty, g_\infty, p_\infty)$ in equivariant Cheeger-Gromov
topology, with $M^n_\infty$ diffeomorphic to $D^n$ in the incomplete
case, and suppose furthermore that this sequence is algebraically
non-collapsed. Then, $(M^n_\infty, g_\infty, p_\infty)$ is a
$\ggl_\infty$-homogeneous space, and
the corresponding sequence of abstract brackets $([\mu_k])_{k\in \NN}$
subconverges to the abstract bracket  $[\mu_\infty]$ of the limit space.
\end{theorem}

We explain now, how we associate to a $\ggo$-homogeneous space
an \emph{abstract bracket}. After having chosen a $\GG$-invariant
background metric $\bar g$ on $M^n$, we let $\mu \in \Vg$ be a bracket
associated to  $(M^n, g, p, \ggo)$ (Def.~ \ref{def_bracket}).
  In \mbox{Section \ref{sec_stratif}} we denoted by $\ip$ an
$\Ad(H)$-in\-variant scalar product on $\ggo$ with  $\la \hg, \mg\ra =
0$, which extended the scalar product on
  $\mg$ induced by $\bar g$.
We then choose a $\ip$-orthonormal basis on $\ggo$ and use it to
identify $(\ggo,\ip) \simeq (\RR^N,\ip_{\rm can})$ as Euclidean vector
spaces.
  All tensor spaces associated to $\ggo$ and $\RR^N$ are thus identified
in a natural way; in particular we have the isometric identification
\begin{equation}\label{eqn_identifVgVN}
     \Vg \simeq V_N := \Lambda^2(\RR^N)^* \otimes \RR^N.
\end{equation}

\begin{definition}\label{def_absbracket}
Let $(M^n, g, p, \ggo)$ be a $\ggo$-homogeneous and  $\mu\in \Vg$ an
associated bracket. Then
the  \emph{abstract bracket} $[\mu] \in V_N/\Or(N)$ associated to $(M^n,
g, p, \ggo)$
is the $\Or(N)$-orbit  in $ V_N$ of the bracket corresponding to $\mu$
under the above identification.
\end{definition}

The abstract bracket $[\mu]$ is well-defined, since the bracket $\mu$ is
well-defined
up to an element in $\OHm \subset \Om \subset \Og$.
Moreover, it is invariant under under pull-back by equivariant local
diffeomorphisms: see Lemma \ref{lem_absbracketinv}.

\begin{remark}
As already mentioned in 
the introduction, knowing the
bracket of a locally homogeneous space is equivalent to knowing its
Levi-Civita connection.
Using local orthonormal frames one obtains an abstract connection
in a universal space of connections (up to the action of the orthogonal
group).
Pulling back the original metric by a diffeomorphism does
not change this abstract connection.
\end{remark}

\begin{remark}\label{rem:bracketindback}
The abstract bracket $ [\mu]$ does not depend on the choice of a
background metric
$\bar g$.
\end{remark}

Recall that $\isog(M^n,g)$ denotes the Lie algebra of all
Killing fields on $(M^n,g)$. In the globally homogeneous case, this is
nothing but the Lie algebra of the full isometry group $\Iso(M^n,g)$. A
first immediate
consequence of Theorem \ref{thm:Liebracketconvergence} is the following

\begin{corollary}\label{cor:algcollapse}
Suppose that the sequence $(M^n_k\!, g_k, p_k,\ggl_k)_{k\in \NN}$
converges to
the limit space
$(M^n_\infty, g_\infty, p_\infty)$ in equivariant Cheeger-Gromov
topology, with $M^n_\infty$ diffeomorphic to $D^n$ in the incomplete
case, and that $\dim \ggl_k = \dim \left(\isog(M^n_k,g_k) \right) \equiv N$.
If moreover we have that $\lim_{k \to \infty }\Vert \mu_k \Vert_{\bar
g_k}=\infty$, then the sequence is algebraically collapsed, and in
particular $\dim \left(\isog(M^n_\infty,g_\infty) \right) > N$.
\end{corollary}

Geometrically, this results reflects the fact that there is algebraic
collapse if and only if some Killing fields ``run into the isotropy'': that is, there is a sequence of Killing fields $X_k$ with
$X_k(p_k)\neq 0$,
whose 1-jets have norm one (see (\ref{Killingnormone})),
which converge to a limit Killing field $X_\infty$ with
$X_\infty(p_\infty)=0$.
This happens already for left-invariant metrics on $\Sl(2,\RR)$ as
explained in Example \ref{ex_algcollapse}.

The following result is most important for our applications.

\begin{lemma}\label{lem_solitonbracket}
A $\ggo$-homogeneous space $(M^n,g,p,\ggo)$ is locally isometric to an
expanding homogeneous Ricci soliton if and only if its abstract bracket
$[\mu] \in V_N/\Or(N)$ is
a \emph{soliton bracket}, in the sense that for a representative $\mu
\in V_N$  we have
\[
     \mu \in \sca_{\Beta} \quad \hbox{and} \quad \Riccim_{\mu} =
{\Beta}_\mg.
\]
\end{lemma}

See Section \ref{sec_stratif} and (\ref{eqn_defbetam}) for the definition
of $\Beta_\mg$. Notice that for this recognition result it is not enough to require only the geometric condition
$\Riccim_{\mu} = {\Beta}_\mg$, since for non-Einstein
solitons there exist counter examples. For instance, from \cite{Mln} it follows that the Ricci eigenvalues of a certain left-invariant metric on $\Sl_2(\RR)$ are $0, 0$ and $-1$, and these coincide with those of the solvsoliton metric on $E(1,1)$.

We now work towards proving Theorem \ref{thm:Liebracketconvergence}. Let
$(M^n_k,g_k, p_k, \ggl_k)_{k\in \NN}$ be a sequence of pointed
$\ggl_k$-homogeneous spaces, and recall that by definition $\ggl_k$ is a
transitive Lie algebra of $g_k$-Killing fields
on $M^n_k$.
We say that $(M^n_k,g_k, p_k, \ggl_k)_{k\in \NN}$ converges
to a Riemannian manifold $(M^n_\infty,g_\infty, p_\infty)$ in pointed
$C^\infty$-Cheeger-Gromov topology if there exists an exhaustion
$\{U_k\}_{k \in \NN}$ of open sets of $M^n_\infty$ and diffeomorphisms
$\varphi_k:U_k \to V_k \subset M^n_k$ with
$\varphi_k(p_\infty)=p_k$ such that the pulled-back metrics
$\varphi^*_k (g_k\vert_{V_k})$ on $U_k$ converge to $g_\infty$ in
$C^\infty$-topology, uniformly on compact subsets of $M^n_\infty$.
In the following, for the sake of notation we will in general omit these
diffeomorphisms and
rather work on the limit space directly.

Recall that a Killing field  $X_k \in \ggl_k$ is uniquely determined by
$X_k(p_k) \in T_{p_k} M^n_k$ and $(\nabla^{g_k} X_k)_{p_k} \in
\End(T_{p_k} M^n_k)$.
We say that a sequence $(X_k)_{k \in \NN}$ of $g_k$-Killing fields
on $M^n_k$ converges to a $g_\infty$-Killing field $X_\infty$ on
$M^n_\infty$ in $C^1$-topology,
if (after pulling back by the diffeomorphisms from the Cheeger-Gromov
convergence),
we have that $X_k \to X_\infty$ and
$\nabla^{g_k} X_k \to \nabla^{g_\infty} X_\infty$ as $k \to \infty$,
uniformly on compact subsets of $M^n_\infty$. Here $\nabla^g$ denotes
the Levi-Civita connection of the metric $g$.

\begin{proposition}[\cite{Heber1998}]\label{prop_limitKf}
Suppose that $(M^n_k, g_k, p_k,\ggl_k)_{k\in \NN}$ converges to a
locally homogeneous space
$(M^n_\infty, g_\infty, p_\infty)$ in Cheeger-Gromov topology. Then,
$(\ggl_k)_{k \in \NN}$ subconverges  to a Lie algebra $\ggl_\infty$ of
$g_\infty$-Killing fields  on $M^n_\infty$ in $C^1$-topology.
\end{proposition}

\begin{proof}
Since for all $k \in \NN$ $\ggl_k$ is a transitive Lie algebra of
$g_k$-Killing fields on $M^n_k$,
we have $n \leq \dim \ggl_k \leq \tfrac{1}{2} \, n(n+1)$. Hence we may
assume that
$\dim \ggl_k \equiv N$.

Consider a sequence $(X_k)_{k\in \NN}$, with $X_k\in \ggo_k$, and assume
that
\begin{eqnarray}
   \la X_k ,X_k \ra_{g_k}^* :=  \big\Vert X_k(p_k) \big\Vert_{g_k}^2
      - \tr \big((\nabla^{g_k} X_k)_{p_k} \big)^2 = 1\,.
\label{Killingnormone}
\end{eqnarray}
Recall that for a Killing field $X$, the endomorphism $\nabla X$ of $T_p
M$ is skew-symmetric. By \cite[p.~ 330]{Heber1998} along a subsequence
we have convergence
to a $g_\infty$-Killing field $X_\infty$ on $M^n_\infty$ of norm one in
$C^1$-topology.
Let $\{X^1_k, \ldots, X^N_k\}$ be a $\ippk$-ortho\-nor\-mal
basis  of $\ggl_k$. After passing to a further subsequence,
there exist an orthonormal basis $\{X^1_\infty, \ldots, X^N_\infty\}$ of
$g_\infty$-Killing fields with respect to  $\ippi$, such that
for all $i=1, \ldots ,N$ and $k \to \infty$ we have
$X^i_k \rightarrow X^i_\infty$ and
$\nabla^{g_k} X^i_k \rightarrow \nabla^{g_\infty} X^i_\infty$
uniformly on compact subsets of $M^n_\infty$.
We set $\ggl_\infty = \Span \{X^1_\infty, \ldots, X^N_\infty \}$.
Since we have $[X^i_k, X^j_k]=(\nabla^{g_k} X^j_k) (X^i_k
)-(\nabla^{g_k} X^i_k)(X^j_k)$ and $\nabla^{g_k} \to \nabla^{g_\infty}$,
it follows that also $[X^i_k, X^j_k]$ converges to $ [X^i_\infty,
X^j_\infty]$
for $k \to \infty$.
This shows that $\ggl_\infty$ is an $N$-di\-men\-sional Lie algebra of
Killing fields
on $(M^n_\infty,g_\infty)$.
\end{proof}

Notice that in the proof of \mbox{Proposition \ref{prop_limitKf}}, the
construction
of the limit Lie algebra $\ggl_\infty$ is independent of the
bases chosen for each $\ggl_k$.

\begin{definition}\label{def_algconv}
In the situation of Proposition \ref{prop_limitKf}
we say that the sequence of pointed $\ggl_k$-homo\-ge\-neous spaces
$(M^n_k,g_k,p_k,\ggl_k)_{k \in \NN}$ converges to
$(M^n_\infty,g_\infty,p_\infty)$ in \emph{equivariant Cheeger-Gromov}
topology.
Furthermore, we call the sequence
\emph{algebraically non-collapsed} if  $\ggl_\infty$ is transitive,
and \emph{collapsed} otherwise.
\end{definition}

Note that in the algebarically non-collapsed case, $(M^n_\infty,
g_\infty, p_\infty,\ggl_\infty)$
is a $\ggl_\infty$-ho\-mo\-geneous space.
In the collapsed case we have
the following obvious observation.

\begin{lemma}\label{lem_collapsedseq}
Suppose that $(M^n_k\!,g_k,p_k,\ggl_k)_{k \in \NN}$ converges to
$(M^n_\infty,g_\infty,p_\infty)$ in equivariant Cheeger-Gromov topology,
that $\dim \ggl_k = \dim \left(\isog(M^n_k,g_k) \right) \equiv N$
and that the sequence is algebraically collapsed.
Then  $\dim \left(\isog(M^n_\infty,g_\infty) \right) > N$.
\end{lemma}

\begin{proof}
By assumption the limit algebra $ \ggl_\infty= \Span\{
X_1^\infty, \ldots ,X_N^\infty\}$ is not transitive, that is  the vectors
$\{X_1^\infty(p_\infty), \ldots ,X_N^\infty(p_\infty)\}$
do not span $T_{p_\infty}M^n_\infty$. Since on the other hand the limit
space $(M^n,g_\infty,p_\infty)$ is locally homogeneous,
$\isog(M^n_\infty,g_\infty)$ is transitive and hence it must contain
$\ggo_\infty$ properly.
\end{proof}

Let us emphasize that for a $\ggo$-homogeneous space $(M^n, g, p, \ggo)$
we have two different norms on Killing fields:
On one hand, there is the scalar
product $\ipp$ on  $\ggl \simeq \ggl_p \subset T_p M \oplus \End(T_p M)$
of $1$-jets at $p$, defined in \eqref{Killingnormone} (see also \eqref{eqn_scalonejet}).   On the
other hand,
the metric $g$ itself gives a scalar product $g_p$ on $T_p M \simeq \mg$,
which measures only the length of $X(p)$ of a vector field $X$. Here $\ggo
= \hg \oplus \mg$ is the canonical decomposition (Def.~ \ref{def_can}). This scalar product
can be extended to a scalar product on $\ggo$ making $\hg \perp \mg$,
which we again denote by $\ip_g$ and which is given by
\begin{equation}\label{eqn_defipg}
     \ip_g := \ip_\hg \, \oplus  \, g_p.
\end{equation}
Here $\ip_\hg$ is simply the restriction of $\ipp$ to $\hg$.

  The scalar product $\ipp$ induces another reductive decomposition
$\ggo = \hg \oplus \mg_{g}$, by letting $\mg_{g} = \hg^\perp$ with
respect to $\ipp$: see Lemma \ref{lem:reductive}. We call it the
\emph{geometric reductive decomposition}.

\begin{remark}\label{rem:geocomplement}
In general, the geometric reductive decomposition depends on the
metric $g$, and hence does not coincide with the canonical reductive
decomposition. An explicit example is the homogeneous space
$M^5=\SO(4)/\SO(2)$, where $\SO(2)$ is embedded canonically as a lower
block: in this case $\mg= \RR \oplus \mg_1 \oplus \mg_2$, where $\mg_1$
and $\mg_2$ are equivalent real representations of complex type.
The desired metrics have then a complex off-diagonal entry.
\end{remark}


\begin{proof}[Proof of Theorem \ref{thm:Liebracketconvergence}]
Let 
\[ \hg_k = \Span \{Z^{n+1}_k\!\!,\ldots, Z^N_k \}\,,\quad 
\mg_{g_k} = \Span\{X^1_k, \ldots, X^n_k\}
\] 
and
assume that these Killing fields form an $\ippk$-orthonormal basis
adapted to the geometric reductive decomposition $\ggl_k=\hg_k \oplus
\mg_{g_k}$.
By Proposition \ref{prop_limitKf}, we may assume that
these Killing fields converge to $N$ linearly independent limit Killing
fields
$\{Z^{n+1}_\infty\!\!,\ldots, Z^N_\infty \}$ and $\{ X^1_\infty, \ldots,
X^n_\infty \}$
of the limit space $(M^n_\infty,g_\infty,p_\infty)$, possibly along a
subsequence.
Notice that these limit Killing fields span the Lie algebra $\ggl_\infty$:
see remark after Proposition \ref{prop_limitKf}.
Moreover, we obtain a geometric reductive decomposition
$\ggl_\infty=\hg_\infty \oplus \mg_{g_\infty}$ as the limit of the
decompositions $\ggl_k =\hg_k \oplus \mg_{g_k}$.
Finally and most importantly,
since by assumption there is no algebraic collapse, the limit Killing
fields
  $\{X^1_\infty, \ldots, X^n_\infty \}$
span the tangent space $T_{p_{\infty}} M^n_\infty$ of the limit space at
the point
$p_\infty$.

We denote by $\mu^{\ggl_k}:\ggl_k\wedge \ggl_k \to \ggl_k$
and by $\mu^{\ggl_\infty}:\ggl_\infty\wedge \ggl_\infty \to \ggl_\infty$
the corresponding Lie brackets of Killing fields.
Then, as explained in the proof of Proposition \ref{prop_limitKf}, $
\mu^{\ggl_k} \to  \mu^{\ggl_\infty}$,
in the sense that we have $ \mu^{\ggl_k}(\tilde E^i_k,\tilde E^j_k) \to
\mu^{\ggl_\infty}(\tilde E^i_\infty,\tilde E^j_\infty)$
for $k\to \infty$, and so on, where now
$\{ \tilde E_k^1, \ldots ,\tilde E_k^N \}$ denotes the above $\ippk$-orthonormal
basis of $\ggl_k$.
It follows that the corresponding adjoint maps
and the corresponding Killing forms converge, that is  $\kf_{
\mu^{\ggl_k}} \to \kf_{ \mu^{\ggl_\infty}}$ as  $k\to \infty$.

  Consider now the canonical decompositions $\ggo_k = \hg_k \oplus
\mg_k$, which induce a projection of the brackets $(\mu_k)_{\mg_k} :
\mg_k \wedge \mg_k \to \mg_k$ (cf.~ \eqref{eqn_muhgmg}). We claim that,
after passing to a subsequence there exists $C>0$ such that $\Vert
(\mu_k)_{\mg_k}\Vert_{\bar g_k} \leq C$ for all $k\in \NN$.
Recall that by (\ref{def_scal}) we have
\[
   \big\Vert (\mu_k)_{\mg_k} \big\Vert_{\bar g_k}^2
   =\sum_{i,j=1}^n
    g_k \big( \, [E_k^i,E_k^j] \, , \, [E_k^i,E_k^j] \, \big)^2_{p_k}\,,
\]
where $\{E_k^1, \ldots ,E_k^n\}$ denotes a $\ip_{g_k}$-orthonormal basis of
Killing vector fields in the canonical complement $\mg_k$,
that is $\{E_k^1(p_k), \ldots ,E_k^n(p_k)\}$ is a $g_k$-orthonormal
basis of
$T_{p_k}M^n_k$.
We write each of these Killing fields with respect to the
geometric reductive decomposition $\ggl_k =\hg_k \oplus \mg_{g_k}$ as
$E_k^i=\tilde Z_k^i + \tilde X_k^i$. Notice that $\tilde Z_k^i(p_k)=0$
and that
$E_k^i(p_k)=\tilde X_k^i(p_k)$.

We write each $\tilde X^i_k \in \mg_{g_k}$ now as
\[
     \tilde X_k^i = c_k^{1,i} X_k^1 + \cdots + c_k^{n,i} X_k^n,
\]
with $c_k^{j,i} \in \RR$ and evaluate this identity of Killing fields at
the point
$p_k \in M^n_k$: we have $\big\Vert \tilde X^i_k (p_k) \big\Vert_{g_k} =
1$ and
moreover the sequence $\{ X_k^1, \ldots ,X_k^n \}$ converges to linearly independent Killing vector fields $\{ X_\infty^1, \ldots ,X_\infty^n \}$.
Since by the non-collapsing assumption
also $\{ X_\infty^1(p_\infty), \ldots ,X_\infty^n(p_\infty) \}$ are linearly
independent
vectors, the sequence $(c_k^{j,i})_{k\in\NN}$ is uniformly bounded for each
$i, j =1,\ldots,n$. It follows that the Killing fields $(\tilde
X_k^i)_{k\in\NN}$ are uniformly bounded also in the geometric norm, that
is $\Vert \tilde X_k^i\Vert_{g_k}^* \leq D$ for all $k \in \NN$ and all
$i=1, \ldots ,n$.

Next, let us show that also the Killing fields $\tilde Z_k^i$ are
uniformly bounded, that is $\Vert \tilde Z_k^i\Vert^*_{g_k} \leq D$ for
all $k \in \NN$ and all $i=1, \ldots ,n$.
If this would not be the case then one of the sequences, say $(\tilde
Z_k^{1})_{k\in\NN}$, must be unbounded. After rescaling
the corresponding sequence $(E^1_k)_{k\in \NN}$ we would obtain a
sequence $(E_k)_{k\in \NN}$ with $E_k \in \mg_{k}$ and $\Vert
E_k\Vert^*_k = 1$, such that $\Vert E_k(p_k)\Vert_{g_k} \to 0$ as $k\to
\infty$.
By passing to a subsequence we would get
  $E_k \to Z_\infty \in \hg_\infty$ for $k \to \infty$,
  and \mbox{Lemma  \ref{lem_Kill}} would imply that
$\kf_{\mu^{\ggl_\infty}} (Z_\infty,Z_\infty ) \leq -1$. On the other
hand, since $E_k \in \mg_k$,
by the very definition of the canonical decomposition
we have $\kf_{ \mu^{\ggl_k}}(E_k,Z_k)=0$ for all $Z_k \in \hg_k$.
Since by assumption the sequence of homogeneous spaces considered is
algebraically non-collapsed, we can find a sequence $Z_k \in \hg_k$
converging to $Z_\infty$
for $k \to \infty$, and since $\kf_{ \mu^{\ggl_k}}\to \kf_{
\mu^{\ggl_\infty}}$
we deduce that $\kf_{ \mu^{\ggl_\infty}}(Z_\infty,Z_\infty)=0$. This is
a contradiction.

So far, we have proved that for each $i= 1,\ldots, n$ the above sequence
$\big( E_k^i \big)_{k\in\NN}$ is $\ip_{g_k}^*$-bounded. By
\mbox{Proposition \ref{prop_limitKf}}
we may then assume that for each $i$ the Killing fields $E_k^i$,
subconverge in $C^1$-topology to limit Killing fields on
$(M^n_\infty,g_\infty)$.
Since $(M^n_k,g_k,p_k)_{k\in \NN}$ converges to
$(M^n_\infty,g_\infty,p_\infty)$
for $k\to \infty$, we have that $[E_k^i, E_k^j]$ converges to a
$g_\infty$-Killing field on $M^n_\infty$. This shows that
$\Vert (\mu_k)_{\mg_k}\Vert_{\bar g_k} \leq C$ for all $k\in \NN$.

In the next step, we replace the above $\ippk$-orthonormal basis of
$\ggl_k$,
adapted to the geometric decomposition $\ggl_k=\hg_k \oplus \mg_{g_k}$,
by the same basis $\{Z_k^{n+1}, \ldots ,Z_k^N\}$ for $\hg_k$ but by the
new $g_k$-orthonormal basis $\{E_k^1, \ldots ,E_k^n\}$ of the canonical
complement $\mg_k$.
For each $k\in \NN$ this gives a $\ip_{g_k}$-orthonormal basis  of $\ggl_k$,
again denoted by $\{ \tilde E_k^1, \ldots ,\tilde E_k^N \}$: see (\ref{eqn_defipg}).

It remains to show that the abstract brackets $[\mu_k]$
associated to $(M^n_k,g_k,p_k,\ggl_k)$ converge
to the abstract bracket $[\mu_\infty]$ associated to
  $(M^n_\infty,g_\infty,p_\infty,\ggl_\infty)$. To this end, we write
  again $E_k^i=\tilde Z_k^i+\tilde X_k^i$ with $\tilde Z_k^i \in \hg_k$ and
$\tilde X_k^i \in \mg_{g_k}$, $i=1, \ldots ,n$. By the above we know
that $\Vert \tilde X_k^i\Vert_{g_k}^*,\,\,
\Vert \tilde Z_k^i\Vert_{\hg_k} \leq D$ for all $k \in \NN$ and all
$i=1,\ldots,n$.
As above we may assume that all these Killing fields subconverge to
limit Killing fields on $(M^n_\infty,g_\infty)$.  From Definitions
\ref{def_bracket} and \ref{def_absbracket}
it follows that in order to compute the coefficients of the abstract
bracket
$[\mu_k]$, we have to evaluate the bracket $\mu_k$ associated to
$(M^n_k,g_k,p_k,\ggl_k)$ as in Definition \ref{def_bracket} in an
orthonormal
basis with respect to the scalar product $\ip_k$ on $\ggl_k$ induced
by the background metric $\bar g_k$: see \mbox{Section \ref{sec_stratif}}.
As we have seen above,
this is equivalent to evaluating the standard bracket $\mu^{\ggl_k}$ of
smooth
Killing vector fields in the
$\ip_{g_k}$-orthonormal basis $\{ \tilde E_k^1, \ldots ,\tilde E_k^N \}$ of $\ggl_k$,
since both these scalar product agree on $\hg_k$. From this the
convergence of the abstract brackets follows immediately.
\end{proof}

\begin{proof}[Proof of Corollary \ref{cor:algcollapse}]
If the sequence is algebraically non-collapsed, then Theorem
\ref{thm:Liebracketconvergence} implies that the sequence of Lie
brackets converges, contradicting the fact that their norms diverge. The
last assertion follows by Lemma \ref{lem_collapsedseq}.
\end{proof}

\begin{proof}[Proof of Lemma \ref{lem_solitonbracket}]
The above mentioned isometric identification $\ggo\simeq \RR^N$ induces
an identification $\Gg \simeq \Gl(N,\RR)$, and under
\eqref{eqn_identifVgVN} their respective natural actions on $\Vg$, $V_N$
coincide. Thus, the strata given by the Stratification Theorem
\ref{thm_stratifb} also coincide, and so does any data defined only in
terms of the bracket and the scalar product, such as $\Riccim_\mu$.
Hence, the lemma follows from Remark \ref{rmk_solitonbracket}.
\end{proof}

\begin{lemma}\label{lem_absbracketinv}
Let $f: (M_1^n,g_1,p_1,\ggo_1) \to (M_2^n, g_2, p_2, \ggo_2)$ be a local
isometry with $f(p_1) = p_2$, inducing a Lie algebra isomorphism $\hat
f : \ggo_1 \to \ggo_2$. Then,
for the corresponding brackets $\mu_1$, $\mu_2$ we have
$[\mu_1]=[\mu_2]$ in $V_N/\Or(N)$.
\end{lemma}

\begin{proof}
Since $f$ is a local isometry, the corresponding map on Killing fields
$\hat f$ takes $\hg_1$ onto $\hg_2$. The map $\hat f$ is a Lie algebra
isomorphism, since the bracket of vector fields can be expressed in
terms of the Riemannian connection, preserved by $f$. Thus, we also
have  $\hat f(\mg_1) = \mg_2$ for the canonical complements, as they are
defined only in terms of Lie theoretical data. Using $\bar g_1 = g_1$,
$\bar g_2 = g_2$ as background metrics, and extending the corresponding
scalar products on $\mg_i$ to orthonormal basis $\ip_i$ for $\ggo_i$ as
in \eqref{eqn_defipg} above, $i=1,2$, it is clear that $\hat f :
(\ggo_1, \ip_1) \to (\ggo_2, \ip_2)$ is an isometry. This implies that
the abstract brackets constructed with orthonormal basis $\{\tilde
E_i\}_{i=1}^N$ and $\{\hat f (\tilde E_i)\}_{i=1}^N$ coincide.
\end{proof}

Using the notation from the proof of Theorem \ref{thm_noncollapsing}, we prove the following characterization of algebraic collapse:

\begin{corollary}
A sequence $(M^n_k, g_k, p_k,\ggl_k)_{k\in \NN}$ converging in equivariant Cheeger-Gromov topology to $(M^n_\infty, g_\infty, p_\infty)$ (with $M^n_\infty$ diffeomorphic to $D^n$ in the incomplete case) is algebraically non-collapsed if and only if the sequence of brackets restricted to the corresponding canonical complements is uniformly bounded, that is,	$\Vert (\mu_k)_{\mg_k} \Vert_{\bar g_k} \leq C$ 
for all $k\in \NN$, for some $C>0$.
\end{corollary}

\begin{proof}
One direction follows immediately from Theorem \ref{thm_noncollapsing}. On the other hand, if the sequence is algebraically collapsed, then there exists a sequence of Killing fields $X_k \in \ggo_k$ with $\Vert X_k(p_k)\Vert_{g_k} \equiv 1$ and $\big\Vert \big(\nabla^{g_k} X_k \big)_{p_k}\big\Vert_{g_k} \to \infty$ as $k\to\infty$. Proposition 7.28 in \cite{Bss} then implies that the restrictions to the canonical complements (and in fact to \emph{any} reductive complement) $(\mu_k)_{\mg_k} : {\mg_k \wedge \mg_k} \to \mg_k$  are unbounded.
\end{proof}



\begin{appendix}


\section{Locally homogeneous spaces}\label{app_lochomog}

Let $(M^n,g)$ be a locally homogeneous space and fix a point $p \in M^n$. 
It is well-known (see for instance \cite{Nomizu1960}) that there exists a Lie algebra $\ggl$ of Killing vector fields defined on an open neighbourhood of $p$ which span $T_p M$.
Since Killing fields are Jacobi fields, every Killing field
$X \in \ggl$  is uniquely determined by the data 
\[
	X(p) \in T_p M, \qquad (\nabla X)_p\in \sog(T_p M, g_p),
\]
were $\nabla=\nabla^g$ is the Levi-Civita-connection of  $(M^n,g)$.
As a consequence, there is a linear isomorphism identifying 
\begin{equation}\label{eqn_identifKilling}
  \ggl \simeq \ggl_p := \left\{ (X(p), - (\nabla X)_p) : X\in \ggl \right\} \subset
     T_p M \oplus \sog(T_p M, g_p)\, =: E_p.
\end{equation}
The Lie bracket $\hml(X,Y)=[X,Y]$ of two Killing fields $X,Y$ is again a Killing field.
Under \eqref{eqn_identifKilling}
it corresponds to the following
Lie bracket \mbox{on $\ggl_p$ (cf.~\cite{Nomizu1960}):}
\begin{equation}\label{eqn_bracketKf}
    \left[ (v, A), (w, B) \right] = \left(Aw - Bv, [A,B] - {\rm Rm}_{g_p}(v,w) \right)\,,
\end{equation} 
where $(v,A), (w,B) \in \ggl_p$, and $[A,B] = A \cdot B - B \cdot A$.
The metric $g_p$ on $T_p M$ induces on $E_p$ a natural scalar product, which for $(v,A), (w,B) \in E_p$ is given by
\begin{eqnarray}
  \big\la (v,A),(w,B)\big\ra^*_g = g_p(v,w) - \tr (A\cdot B)\, .\label{eqn_scalonejet}
\end{eqnarray}

Let  $\hg = \{X\in \ggl : X(p) = 0 \}$
be the isotropy subalgebra at $p$, and denote by $\mg_g := \hg^\perp$ the $\ipp$-orthogonal complement of $\hg$
in $\ggl$. 



\begin{lemma}\label{lem:reductive}
The decomposition $\ggl = \hg \oplus \mg_g$ is reductive, that is, $[\hg,\mg_g] \subset \mg_g$.
\end{lemma}

\begin{proof}
We identify $\ggl$ with $\ggl_p$ using \eqref{eqn_identifKilling}. Let $(0,A),(0,\tilde A)\in \hg$, $(w,B) \in \mg_g$. Then, 
\[
	 \big\la\,[(0,A),(w,B)],(0,\tilde A)\,\big\ra^*_g
	 =0- \tr \big([A,B]\cdot \tilde A\big)
	 =- \tr \big([\tilde A,A]\cdot B\big)=0\,,
\]
by (\ref{eqn_bracketKf}) and (\ref{eqn_scalonejet}), since $\hg$ is a subalgebra. This shows the claim.
\end{proof}

Notice that, the reductive complement $\mg_g$ may depend on the metric $g$: see Remark \ref{rem:geocomplement}.
The following technical result will be needed in the sequel.

\begin{lemma}\label{lem_Kill}
For $Z\in \hg$, $\Vert Z \Vert^*_g = 1$, we have that $\kf_\hml(Z,Z) \leq -1$.
\end{lemma}
\begin{proof}
The endomorphism 
\[
 \ad_{\hml}(Z) : \ggl \to \ggl\,\,;\,\,\,X \mapsto [Z,X]=\nabla_Z X-\nabla_X Z
\] 
preserves the metric reductive decomposition $\ggl = \hg \oplus \mg_g$. Let us denote by $(\ad_\hml(Z))|_\hg$
and $(\ad_\hml(Z))|_{\mg_g}$ the restrictions to each of those subspaces. Using the identification \eqref{eqn_identifKilling}, 
the map $T:\mg_g \to T_pM^n\,;\,\,X\mapsto X(p)$ is a linear isomorphism,
since $\ggl$ is a transitive Lie algebra of Killing fields. 
Notice that from the above identity we get 
$T \circ (\ad_\hml (Z))|_{\mg_g} \circ T^{-1}=-(\nabla Z)_p$. As a consequence,
\begin{eqnarray*}
	\lefteqn{\tr_\ggl\big({(\ad_\hml (Z))\cdot (\ad_\hml (Z))} \big) }&&\\
	&=& \tr_\hg \big( (\ad_\hml(Z))|_\hg \cdot (\ad_\hml(Z))|_\hg \big) 
	+  \tr_{\mg_g} \big( (\ad_\hml(Z))|_{\mg_g} \cdot (\ad_\hml(Z))|_{\mg_g} \big)  \\
	&\leq&  \tr_{\mg_g} \big( (\ad_\hml(Z))|_{\mg_g} \cdot (\ad_\hml(Z))|_{\mg_g} \big) 
	  = \tr_{T_p M} \big( (\nabla_\cdot Z)_p \cdot (\nabla_\cdot Z)_p \big)  = - 1,
\end{eqnarray*}
where the inequality follows from $(\ad_\hml (Z))|_\hg$ being skew-symmetric (recall that the metric $\ip_g^*$ restricted to $\hg$ is nothing but a negative multiple of the Killing form of $\sog(T_p M, g_p)$), and the last equality is just by definition of $\ipp$.
\end{proof}

The previous result justifies the definition of the \emph{canonical decomposition}: consider $\mg \subset \ggl$ the orthogonal complement of $\hg$ with respect to the Killing form $\kf_\hml$.

\begin{lemma}\label{lem_can}
The canonical decomposition
$\ggl=\hg \oplus \mg$ is reductive.
\end{lemma}

\begin{proof}
Let $\GG$ be the simply-connected Lie group with Lie algebra $\ggl$ and let $H$ denote the connected subgroup corresponding to $\hg$. The Killing form $\kf_\hml$ of $\ggl$ is $\Ad(\GG)$-invariant.
Thus, since $\hg$ is $\Ad(H)$-invariant, the $\kf_\hml$-orthogonal complement $\mg$ of $\hg$ in $\ggl$
is also $\Ad(H)$-invariant. This implies that $[\hg,\mg]\subset \mg$.
\end{proof}

For a homogeneous space we now have two reductive decompositions: the first one $\ggl = \hg \oplus \mg_g$ is more of a metric nature, and encodes geometric information, whereas the canonical decomposition $\ggl = \hg \oplus \mg$ is of a Lie-theoretical nature, and does not depend on a particular metric. It turns out that, for technical reasons, the canonical decomposition is more useful in most situations. The exception is when dealing with convergence of Killing fields as in Section \ref{sec_algcon}.

As mentioned in Section \ref{sec_homricflow}, a $\ggl$-homogeneous space
admits a local action by a connected, simply connected Lie group $\GG$
with Lie algebra $\ggl$. Since the following beautiful proof of this fact, which we
owe to Robert Bryant, doesn't seem to be known, 
we present it here for the convenience of the reader.

\begin{theorem}[Second fundamental theorem of Lie]\label{thm_Lie}
Let $M^n$ be a smooth manifold and let $\ggl$ be a finite-dimensional Lie algebra 
of smooth vector fields on $M^n$. Then there exists an effective, smooth,  local action
of a simply-connected Lie group $\GG$, for which the corresponding action fields span
$\ggl$.
\end{theorem}

\begin{proof}
The existence of the Lie group $\GG$ is well-known.
The Lie bracket $[\,\cdot \,, \cdot \,]$ on $T_e\GG$ is defined as follows: 
if $R_v,R_w$ are right-invariant vector fields on $\GG$, 
uniquely determined by $v,w \in T_e \GG$, then the Lie bracket $[R_v,R_w]$
is a right-invariant vector field as well and we set
$[v,w]:=[R_v,R_w](e)$. 
 By assumption we obtain a Lie algebra isomorphism
$\psi:T_e \GG\to \ggl$, that is $[\psi(v),\psi(w)] =  \psi([v,w])$. 
Now, on $\GG \times M^n \times M^n$ we define for each $v\in\ggl$
the smooth vector field $Z_v$ by
\[
  Z_v(g,p,q) = R_v(g) + 0_p + \psi(v)(q) \in T_g\GG\oplus T_pM^n\oplus T_qM^n\,.
\]
Notice that for $v \neq 0$ the vector field $Z_v$ does not have any zeros.
Moreover, the vector fields
$Z_v$ span an involutive $N$-dimensional distribution, where $N = \dim(\GG)$. It follows
by the Theorem of Frobenius,  that $\GG \times M^n \times M^n$ 
is foliated by maximal integral manifolds of dimension $N$. 
We denote the corresponding foliation by $F$.

Consider the $n$-dimensional submanifold  $\{e\} \times \Delta(M^n)$ of
 $\GG \times M^n \times  M^n$ and notice that $F$-leaves intersect  $\{e\} \times \Delta(M^n)$ transversally. Thus, the union of the leaves, which
intersect $\{e\} \times \Delta(M^n)$, is a smooth 
submanifold $P^{N+n}$ of $\GG \times M^n \times M^n$ of dimension $N+n$ near the diagonal
$\{e\} \times \Delta(M^n)$.
It can be written as the graph of a smooth mapping 
$F: U\to M^n$, where $U$ is an open neighborhood of $\{e\}\times M^n$ in $\GG \times M^n$.
That is, locally $P^{N+n}$ consists precisely of the triples $(g,m,F(g,m))$. Of course
$F$  satisfies $F(e,m) = m$ for all $m\in M^n$. 

Recall, that the vector fields $Z_v$ are tangent to the submanifold $P^{N+n}$.
For initial values  $m_0 \in M^n$ and $g_0\in \GG$, sufficiently close to $e$, so that
$(g_0,m_0)\in U$, 
let $\gamma(t)=(g(t),m(t),F(g(t),m(t))$ be an integral curve of $Z_v$.
Then clearly $m(t)\equiv m_0$
and $g(t)=\exp(tv)\cdot g_0$. More importantly  we have
\[
 	F( \exp(tv)\cdot g_0, m_0) = \Phi_t^{\psi(v)}(F(g_0,m_0))
\]
for small $t$,
where $\Phi_t^{\psi(v)}$ denotes the local flow of $\psi(v)$.
Writing $h(t)=\exp(tv)$ it follows from this that
$F(h(t),F(g,m)) = F(h(t)g,m)$ holds, whenever this makes sense.
This shows the claim.
\end{proof}


\section{Uniqueness of locally homogeneous Ricci flows}\label{app_unique}

Let us fix in this section a pointed, simply-connected smooth manifold $(M^n, p)$. A Ricci flow solution $g(t)_{t\in I}$ on $(M^n, p)$ is called \emph{locally homogeneous} if  for each $t\in I$, $g(t)$ is a possibly incomplete locally homogeneous metric. Here $I\subset \RR$ is an open interval.
 In this appendix we will prove

\begin{theorem}\label{thm_uniqueRf}
A locally homogeneous Ricci flow solution $(g(t))_{t \in I}$ on a simply connected locally homogeneous space $(M^n, p)$ is uniquely determined by an initial metric $g(t_0)$, for any $t_0 \in I$.
Moreover, the local isometry group of these metrics does not change.
\end{theorem}

 As a consequence, such a solution  is nothing but the $\GG$-invariant Ricci flow solution
defined in (\ref{eqn_ricciflowm}). Recall that uniqueness of Ricci flows holds for complete metrics with bounded curvature, by a result of Kotschwar \cite{Kot}. However, since completeness is used in an essential way, his methods do not seem to apply here.

In order to prove Theorem \ref{thm_uniqueRf},
we need to study locally homogeneous metrics $g$ on $(M^n, p)$ which do not necessarily have the same Lie algebra of Killing fields. By a classical result of Singer \cite{Singer1960} (see \cite{NicolodiTricerri1990} for the incomplete case), any such $g$ is determined up to local isometry by the data 
\[
	w_g := \Big( (\nabla_g^k \, \Riem_g)_p \Big)_{k=0}^{N} \, \in  \, \, \bigoplus_{k=0}^{N} \, \big( \otimes^{k+4} (T_pM)^*  \big) =: W.
\]
where $\nabla_g$ denotes the Levi-Civita connection, $\nabla_g^0 \, \Riem_g = \Riem_g$ and $N = n(n-1)/2$.

For each locally homogeneous metric $g$ there exists a maximal Lie algebra of Killing fields $\ggo_g$ defined on all of $M^n$, whose values at each point span the corresponding tangent spaces (see Section \ref{sec_homricflow} and Appendix \ref{app_lochomog}). 
As in Section \ref{sec_homricflow} we denote by $\hg_g = \{ X \in \ggo_g : X(p) = 0\}$ the isotropy subalgebra at $p$. The metric $g$ induces a scalar product $\ip_g$ on $T_p M^n$, which in turn induces scalar products on $(T_p M^n)^*$ and all the tensor spaces, such as $W$.

The group $\SO(T_p M^n, \ip_g)$ acts linearly on $W$ by the natural extension of its action on 
$T_p M^n$, and there is a Lie algebra representation $\pi : \sog(T_p M) \to \End(W)$, corresponding to
that action.
It follows from \cite{Singer1960,NicolodiTricerri1990} that $\hg_g$ is also 
the isotropy subalgebra of $\pi$ at the vector $w_g \in W$, that is,
\begin{equation}\label{eqn_hgg}
	\hg_g = \{ A \in \sog(T_p M^n, \ip_g) : \pi(A) w_g = 0\}\,.
\end{equation}

\begin{lemma}\label{lem_hgsmooth}
The map $g\mapsto \hg_g$ from locally homogeneous metrics on $(M^n, p)$ to subspaces of $\End(T_p M^n)$ is smooth, provided the dimension $\dim \hg_g$ does not jump. 
\end{lemma}

\begin{proof}
Let $(g(s))_{s\in (-\epsilon,\epsilon)}$ be a smooth family of locally homogeneous metrics. 

Assume first that $g(s)_p \equiv \ip$, a fixed scalar product on $T_p M^n$. Since the map $g\mapsto w_g \in W$ is smooth, it suffices to show that $\hg_g$ depends smoothly on $w_g$. By  \eqref{eqn_hgg} we have that for each $s$ the subalgebra $\hg_{g(s)}$ is the kernel of the linear map
\[
	T_s : V\to W, \qquad T_s A := \pi(A)\, w_{g(s)}\,,
\]
where  $V=\sog(T_p M^n, \ip)$.
The family of linear maps $T_s$ is smooth, and by assumption the maps $T_s$ have constant rank, thus their kernels $\hg_{g(s)}$ depend smoothly on $s$. This is seen as follows:
We let $T_s^t:W \to V$ denote the transpose map of $T_s:(V,\ip_g)\to (W,\ip_g)$.
Then $P_s:=T_s^t \cdot T_s:V\to V$ is positive semi-definite with
$\ker(P_s)=\ker (T_s)$. Since the image of $P_s$ depends smoothly on $s$, clearly also
the orthogonal complement does, which shows the above claim.

In the general case, let $h(s) \subset \End(T_p M^n)$ be a smooth family of endomorphisms such that $h(s) : (T_p M^n, g(0)_p) \to (T_p M^n, g(s)_p)$ is an isometry for each $s\in (-\epsilon, \epsilon)$. We then consider the pull-backs $h(s)^* w_{g(s)} \in W$, and by the previous case it follows that the isotropy subalgebras $\hg_s^*$ corresponding to those pull-backs form a smooth family. Finally, the lemma follows by noticing that $\hg_{g(s)} = h(s) \circ \hg_s^* \circ h(s)^{-1}$.
\end{proof}

Recall that the full Lie algebra of Killing fields of a locally homogeneous metric $g$ on $(M^n, p)$ can be identified with a subspace of $T_p M^n \oplus \End(T_p M^n)$ as in (\ref{eqn_identifKilling}). 
As explained in \cite{NicolodiTricerri1990} (see also \cite{Tri92}), the difference tensor between the canonical Ambrose-Singer connection $\nabla_0^{AS}$ and the Levi-Civita connection $\nabla^g$ (at the
given point $p$) is a linear map
\[
	S^g : T_p M^n \to \hg_g^\perp \subset \sog(T_p M^n, \ip_g)\,\,;\,\,\,  X\mapsto S^g_X\,,
\]
where $\hg_g^\perp$ is the orthogonal complement of $\hg_g$ with respect to the natural scalar product on $\sog(T_p M^n, \ip_g)$ induced by $\ip_g$. In more concrete terms, this means the following: there exists a linear map $S^g$ as above, such that the full Lie algebra of Killing fields at $p$ is given by 
\begin{equation}\label{eqn_ggoSg}
	\ggo_g \simeq  \hg_g \oplus \big\{ (X, S^g_X) : X\in T_p M^n  \big\} 
	\subset T_p M^n \oplus \sog(T_p M^n, \ip_g)\,,
\end{equation}
see the proof of Prop.~4.4 in \cite{Tri92}.

\begin{lemma}\label{lem_ggosmooth}
For locally homogeneous metrics $g$ on $(M^n, p)$ with a fixed isotropy subalgebra $\hg$, the full Lie algebra of Killing fields $\ggo_g$ depends smoothly on $g$.
\end{lemma}

\begin{proof}
By \eqref{eqn_ggoSg}, it suffices to show that the map $g\mapsto S^g$ is smooth.  Let $(g(s))_{s\in (-\epsilon,\epsilon)}$ be a smooth family of locally homogeneous metrics with constant isotropy subalgebra $\hg$. 
As in the proof of Lemma \ref{lem_hgsmooth}, assume first that $g(s)_p \equiv \ip$ is constant. Recall that by \cite{NicolodiTricerri1990}, $S^g$ is the unique linear map
\[
  S^g:T_pM^n \to \hg^\perp \subset \sog(T_pM^n,\ip)\,
\]
such that for all $X \in T_pM^n$ one has
\[
  	\pi( S^g_X )\,w_g = i_X(w_g) \,.
\]
Here $i_X \in \End(W)$ is defined by $i_X(T)(X_1, \ldots ,X_{k+4})=T(X,X_1, \ldots ,X_{k+4})$ for a $(k+5,0)$-tensor $T\in \otimes^{k+5} (T_p M^n)^*$. To write this in a more robust way, consider for each fixed $w\in W$ the linear map 
\begin{align*}
	A_w : \End(T_p M^n, \hg^\perp) \longrightarrow \End(T_p M^n, W),
\end{align*}
associating to $S\in \End(T_p M^n, \hg^\perp)$ the linear map $A_w(S) : T_p M^n \to W$ given by
\[
	 X\in T_p M^n \mapsto \pi(S_X)w \in W.
\]
Also, let $b_w \in \End(T_p M^n, W)$ be given by $b_w(X) = i_X(w)$.
Then, for each fixed $s\in (-\epsilon, \epsilon)$, $S^{g(s)}$ is the unique solution to the linear equation  
\[
	A_{w_s} \, S = b_{w_s},
\]
where $w_s := w_{g(s)}$.
Since this  linear equation has a unique solution, $A_{w_s}^t A_{w_s}$ is invertible (the transpose taken with respect to the natural scalar product induced by $\ip$ on the corresponding space of tensors), and the map 
$A_{w_s}^+ := (A_{w_s}^t A_{w_s})^{-1} A_{w_s}^t$
is a right inverse for $A_{w_s}$, so that $S^{g(s)} = A_{w_s}^+ \, b_{w_s}$, which clearly depends smoothly on $s$. 

For the general case we argue as in the proof of Lemma \ref{lem_hgsmooth}. We need to pull-back by maps $h(s) \in \End(T_p M^n)$ which commute with the isotropy representation (so that the pulled-back spaces also have fixed isotropy $\hg$). The existence of these maps is 
shown as follows: we  first consider an arbitrary family $h(s) \subset \End(T_p M^n)$ and then use a standard averaging argument with the compact group $\overline{\Ad(H)} \subset \SO(T_p M, \ip)$. 
\end{proof}

\begin{proof}[Proof of Theorem \ref{thm_uniqueRf}]
Since $M^n$ is simply connected, 
for all $t \in I$,  by \mbox{Theorem 1} in \cite{Nomizu1960} the space $(M^n,g(t),p)$ is $\ggl_t$-homogeneous, where $\ggl_t$ is the full  Lie algebra of Killing vector fields on $(M^n, g(t))$. 

Let $d(t):=\dim \ggl_t$ for $t \in I$. By Proposition \ref{prop_limitKf}, the function $d(t)$ is upper semi-continuous. Thus, if we let  
\[
  d_{\rm min}:= \min\{ d(t):t \in I\}, \qquad I_{\min} = \{ t\in I : d(t) = d_{\min} \},
\]
then $I_{\min} \subset I$ is an open subset. We will assume in what follows that $I_{\min}$ is connected, and it will turn out at the end of the proof that we are not losing generality by doing that. 

By Lemma \ref{lem_hgsmooth}, for $t\in I_{\min}$ the isotropy subalgebras $\hg_t$ vary smoothly in terms of $g(t)$. Now for each $t_0 \in I_{\min}$, as explained in the proof of
Theorem \ref{thm:gRicflow}, there is a $G_{t_0}$-invariant Ricci flow solution $(\tilde g(t))_{t\in \tilde I}$ with $\tilde g(t_0) = g(t_0)$. This nice solution has the property that the corresponding isotropy subalgebras are constant on $t$. Since $\tilde g'(t_0) = g'(t_0)$, and since the isotropy subalgebras depend smoothly on the metrics, we have that $\hg_t$ do not change in first order, that is $\ddt|_{t=t_0} \hg_t = 0$. This holds for any $t_0 \in I_{\min}$, therefore $\hg_t \equiv \hg$ is constant for $t\in I_{\min}$. 

Now by applying Lemma \ref{lem_ggosmooth} and repeating the above argument, we conclude that the Lie algebra of Killing fields $\ggo_t \equiv \ggo$ is constant for $t\in I_{\min}$. 
As a consequence, the restriction $(g(t))_{ t\in I_{d_{\rm min}}}$ agrees
with the Ricci flow solution defined in Section \ref{sec_homricflow}.
We may assume that $I_{d_{\rm min}}=(a,b)$ and that $b \in I$. Then it remains
to show that we cannot have $d(b)>d_{\rm min}$.

So suppose $d(b)>d_{\rm min}$. For simplicity we write $\ggl:= \ggl_{t_0}$.
Since $g(b)$ is still a $\ggl$-homogeneous metric,
the Lie algebra $\ggl_b$ of all Killing vector fields of $g(b)$
clearly contains $\ggl$, and the same holds for the isotropy subalgebras: $\hg \subset \hg_b$.
As a consequence the scalar product $\ip_b$ on $T_p M^n$
corresponding to $g(b)$ is not only $\Ad(H)$-invariant
but even $\Ad( H_b)$-invariant. However, by Theorem \ref{thm:gRicflow} and Lemma \ref{lem_correspGinvmetrics}, 
$S^2(T_p M^n)^{\Ad( H_b)}$ 
is a Ricci flow invariant submanifold of $S^2(T_p M^n)^{\Ad(H)}$.
In particular, if $\ip_b \in S^2(T_p M^n)^{\Ad( H_b)}$ then
for any solution to the Ricci flow equation (\ref{eqn_ricciflow}) 
defined in Section \ref{sec_homricflow},
we have $\ip_t \in  S^2(T_p M^n)^{\Ad( H_b)}$ for all $t$. 
This is a contradiction, and the theorem is proved. 
\end{proof}


\section{The stratum of a Lie algebra}\label{app_stratum}

Consider as in Section \ref{sec_stratif} the action of $\Gg$ on the space of brackets $(\Vg,\!\ip)$. 
Our aim in this section is to obtain information about the stratum corresponding to a Lie bracket $\mu \in \Vg$. Recall that the \emph{variety of Lie algebras}
\[
	\lca = \{ \mu \in \Vg : \mu \mbox{ satisfies the Jacobi identity}\}
\]
is an algebraic subset of $\Vg$ (that is, given by the zero locus of some polynomials). In particular, it is closed in  the usual vector space topology.

We mention here that in the complex case, for the analogous action of $\Gl_n(\CC)$ on $\Lambda^2(\CC^n)^* \otimes \CC^n$, the critical values for the moment map were computed in \cite{Lau2003}. In this section we prove an analogous result for the real case. Moreover, besides computing the possible critical values, we show that the stratum to which a Lie bracket belongs to is completely determined by the stratum of its nilradical.

For a Lie bracket $\mu \in \Vg$ we denote by $\ngo$ the nilradical of $(\ggo,\mu)$, and by 
$\ag := \ngo^\perp$ its orthogonal complement in $\ggo$. 
The main result of this section is the following

\begin{theorem}\label{thm_refinedbeta}
Let $\mu \in \lca \backslash \{0\}$ and denote by 
$\nu = \proy_{\ngo} \circ \, \mu \, |_{\ngo\wedge \ngo}$ the nilpotent Lie bracket given by the restriction of $\mu$ to $\ngo$. Then, if $\nu =0$ we have 
$\mu \in \sca_\Beta$ with
\[
	\Beta = a^{-1} \cdot \twomatrix{-\Id_{\ag}}{0}{0}{ 0  }
\]
and $a = \dim\ag$. If $\nu \in \sca_{\Beta_\ngo}$, then 
$\mu \in \sca_\Beta$, with
\[
	\Beta =  
			b \cdot \twomatrix{ - \Vert{\Beta_\ngo}\Vert^2 \cdot \Id_{\ag}}{0}{0}{ \Beta_\ngo } 
\]
and $b = (1 + a \cdot \Vert{\Beta_\ngo}\Vert^2)^{-1}$. In both cases
the blocks are according to $\ggo = \ag \oplus \ngo$.
\end{theorem}

Setting $\Beta^+ := \Beta + \Vert \Beta\Vert^2 \, \Id_\ggo$,
 we deduce  from Lemma 2.17 in \cite{standard}

\begin{corollary}\label{cor_refinedbetaplus}
Under the assumptions of Theorem \ref{thm_refinedbeta}, one has that
\[
	\Beta^+ = a^{-1} \cdot \twomatrix{0}{0}{0}{\Id_\ngo} \qquad \hbox{or} \qquad \Beta^+ = b  \cdot \twomatrix{0}{0}{0}{ {\Beta_{\ngo}^+} },
\] 
according to whether $\nu = 0$ or $0\neq \nu \in \sca_{\Beta_\ngo}$, respectively. 
Moreover, $\Beta^+\vert_\ngo>0$.
\end{corollary}

Recall that the stratum label $\Beta$ is not unique: for any $\Beta_1 := k \Beta k^{-1}$, $k\in \Og$, we have that $\sca_{\Beta_1} = \sca_\Beta$.  However, the condition that $\mu$ is gauged correctly w.r.t.~ $\Beta_1$ forces $\Beta_1$ to have a form similar to the $\Beta$ from Theorem \ref{thm_refinedbeta}:

\begin{corollary}\label{cor_Imbetaplus}
Let $\mu \in \lca \cap \sca_\Beta$, with $\Beta$ as in Theorem \ref{thm_refinedbeta},  and $\Beta_1 := k \Beta k^{-1}$ for  $k\in \Og$ . 
Then,  $\Im(\Beta_1^+) = \ngo$ and $\Beta_1^+\vert_{\ngo}>0$, provided
 that $\mu\in   \sca_{ \Beta_1} \cap V_{\Beta_1^+}^{\geq0}$.
\end{corollary}

The proof shows in fact that  $k$ preserves $\ngo$. This implies that, as a subspace of $\ggo$, the nilradical $\ngo$ of $(\ggo,\mu)$ is constant for all $\mu\in\lca \backslash \{ 0\}\cap \sca_{ \Beta_1} \cap V_{ \Beta_1^+}^{\geq0}$.

The main ingredient for proving Theorem \ref{thm_refinedbeta} is a technical lemma which we discuss now. Let $\Beta \in \Sym(\ggo)$ be a symmetric endomorphism (not necessarily a stratum label), and consider $\mu_0 \in \lca \cap \Vzero$. In other words, $\mu_0$ is a Lie bracket  for which $\Beta^+ \in \Der(\mu_0)$. Denote by $\ngo := \Im(\Beta^+)$, $\ag := \ngo^\perp = \ker \Beta^+$, and assume that $\ngo$ is the nilradical of $(\ggo,\mu_0)$. Since the kernel of a derivation is a Lie subalgebra, $(\ag, \eta_0)$ is a reductive Lie subalgebra of $(\ggo, \mu_0)$, where $\eta_0 := \mu_0 |_{\ag \wedge \ag}$. Notice that the orthogonal decomposition $\ggo = \ag \oplus \ngo$ induces an embedding $\Gl(\ag) \times \Gl(\ngo) \subset \Gl(\ggo)$.

\begin{lemma}\label{lem_alpha}
Under the assumptions of the previous paragraph, consider the map
\[
\alpha : (\Gl(\ag) \times \Gl(\ngo)) \cdot \mu_0 \to \RR\,\,;\,\,\,
	h \cdot \mu_0 \mapsto  \vert \det h_\ag\vert^{-1},
\] 
where $h = \minimatrix{h_\ag}{}{}{h_\ngo} \in \Gl(\ag) \times \Gl(\ngo)$. Then $\alpha$ is well-defined, continuous, and satisfies
 $\lim_{\mu \to 0}\alpha(\mu) \to 0$.
\end{lemma}

\begin{proof}
The map $\alpha$ is well-defined, if
$h\cdot \mu_0 = \bar h \cdot \mu_0$ 
implies $\vert \det h_\ag \vert = \vert \det \bar h_\ag \vert$. Thus, it is
sufficient to show that
 $\vert \det h_\ag  \vert = 1$ for $h = \minimatrix{h_\ag}{}{}{h_\ngo} \in \Aut(\mu_0)$. This holds true,  since by \cite[Lemma 2.6]{LafuenteLauret2014b}
 $\tr D_\ag = 0$ for any derivation of the form 
 $D = \minimatrix{D_\ag}{}{}{D_\ngo} \in \Der(\mu_0)$.
 Continuity of $\alpha$ is also clear since we are dealing with a smooth action.
 
  In order to show $\lim_{\mu \to 0}\alpha(\mu) \to 0$ we argue by contradiction: suppose that $h^{(k)} \cdot \mu_0 \to 0$ as $k\to \infty$ for a sequence $\big(h^{(k)} \big)_{k\in \NN} \subset \Gl(\ag)\times \Gl(\ngo)$, with $\vert \det h^{(k)}_\ag \vert$ uniformly bounded. Let us decompose orthogonally $\ag = \lgo \oplus \zg$, where $\zg$ is the center of $(\ag, \eta_0)$ (so that $\zg \oplus \ngo$ is the solvable radical of $(\ggo,\mu_0)$). Let $Q(\zg) \subset \Gl(\ag)$ denote the subgroup of maps preserving $\zg$. Since $\Gl(\ag) = \Or(\ag) Q(\zg)$, after acting with $\Or(\ag)$ we may assume without loss of generality (after possibly passing to a convergent subsequence, using that $\Or(\ag)$ is compact) that according to $\ag = \lgo \oplus \zg$ we have
\[
	h_\ag^{(k)} = \minimatrix{h^{(k)}_\lgo}{0}{\star}{h^{(k)}_\zg}, \qquad h_\zg^{(k)} \in \Sym(\zg).
\]
That is, $\zg$ is preserved by $h_\ag^{(k)}$ for all $k\in \NN$, so the radical of $h^{(k)}\cdot \mu$ is the fixed subspace $\zg\oplus \ngo$. The fact that one can furthermore assume that $h_\zg^{(k)}$ is symmetric follows from $\Gl(\zg) = \Or(\zg) \exp(\Sym(\zg))$ and a similar argument as above, using compactness of $\Or(\zg)$.

We now claim that $\big\vert{ \det h_\lgo^{(k)} } \big\vert$ is uniformly bounded.
If not then, since $\big\vert \det h_\ag^{(k)} \big\vert = \big\vert \det h_\lgo^{(k)} \big\vert \cdot \big\vert \det h_\zg^{(k)} \big\vert$ is uniformly bounded,  $\big\vert \det h_\zg^{(k)} \big\vert$ would be uniformly bounded. Extract a bounded sequence $(\lambda_k)_{k\in \NN}$ of real eigenvalues of $h_\zg^{(k)}$, say $\vert\lambda_k \vert \leq L$ for all $k\in \NN$, and let $\big(Z^{(k)} \big)_{k\in \NN} \in \zg$ be a corresponding sequence of unit norm eigenvectors. After passing to a subsequence we may assume that $Z^{(k)} \to \bar Z \in \zg$, with $\bar Z$ of unit norm. Using that $\ad_{h\cdot \mu} Z = h \ad_\mu (h^{-1} Z) h^{-1}$ and $h^{(k)} \cdot \mu \to 0$ we deduce that
\[
	 \lambda_k^{-1} \cdot h^{(k)}  \big(\ad_{\mu_0} Z^{(k)} \big) \big(h^{(k)}  \big)^{-1}  =  \ad_{h^{(k)}\cdot \mu_0} Z^{(k)}\underset{k\to\infty}\longrightarrow 0.
\]
But $\lambda_k$ is uniformly bounded, thus $ h^{(k)}  \big(\ad_{\mu_0} Z^{(k)} \big) \big(h^{(k)}  \big)^{-1} \to 0$, and hence the eigenvalues of $\ad_{\mu_0} Z^{(k)}$ converge to $0$. In other words, $\ad_{\mu_0} \bar Z$ is nilpotent. Since $\bar Z\in \zg \oplus \ngo$, the radical of $(\ggo,\mu_0)$, by \cite[Thm.~3.8.3]{Varad84} this implies that $\bar Z\in \ngo$, a contradiction. This shows the above claim.

Let us get a contradiction for our starting assumption by using the above claim. Notice that $\eta_k := h_\ag^{(k)} \cdot \eta_0 = \big(h^{(k)} \cdot \mu_0 \big)\big|_{\ag \wedge \ag}$ is a sequence of reductive Lie brackets on $\ag$, whose centers are the fixed subspace $\zg$. According to the decomposition $\ag = \lgo \oplus \zg$, their Killing form endomorphisms are given by
\[
	\kf_{\eta_k} = \minimatrix{\kf^\lgo_{\eta_k} }{0}{0}{0}, 
\]
with $\kf_{\eta_k}^\lgo$ invertible (since the quotient $\ag / \zg$ is a semisimple subalgebra for all $k\in \NN$). They converge to $0$ as $k\to \infty$. On the other hand, by \cite[Lemma 3.7]{homRF} we have  
\[
	\big( \big(h_\lgo^{(k)}\big)^{-1} \big)^t \cdot \kf_{\eta_0}^\lgo  \cdot \big(h_\lgo^{(k)}\big)^{-1} = \kf^\lgo_{\eta_k},
\]
for all $k\in \NN$, thus taking determinants we get that 
\[
	\big(\det h_\lgo^{(k)} \big)^{-2} \cdot \det \kf_{\eta_0}^\lgo = \det \kf_{\eta_k}^\lgo \underset{k\to\infty}\longrightarrow 0,
\]
which contradicts the claim, since $\det \kf_{\eta_0}^\lgo \neq 0$. This finishes the proof.
\end{proof}
  
We are now in a position to prove the main result of this section.

\begin{proof}[Proof of Theorem \ref{thm_refinedbeta}]
Let $\mu \in \lca \backslash \{0\}$ and assume that $\nu \in \sca_{\Beta_\ngo}$.
Let $\Beta$ and $b$ be defined as in the statement. Then
$ \Beta^+ = b  \cdot \minimatrix{0}{0}{0}{ {\Beta_{\ngo}^+} }$.
In order to show that $\mu \in \sca_\Beta$ we apply the definition of the strata (see \eqref{eqn_stratumbeta}), that is we need to prove that for some $\tilde \mu = k \cdot \mu$, $k\in \Og$, we have that $\tilde\mu \in \Vnn$ and that $0\notin \overline{H_{\Beta} \cdot \tilde \mu}$.

\vskip5pt
{\noindent \bf Step 1.} \emph{  There exists $k\in \Og$ such that $k\cdot \mu \in \Vnn$.}
\vskip3pt 

Let $k_\ngo \in \Or(\ngo)$ be such that $k_\ngo \cdot \nu \in {V_{\Beta_\ngo^+}^{\geq 0}} \subset V(\ngo)$. We extend it trivially to all of $\ggo$ (according to $\ggo = \ag \oplus \ngo$) to obtain $k\in \Og$, and in what follows we replace $\mu$ by $k\cdot \mu$. In this way, we may assume that $\nu \in V_{\Beta_\ngo^+}^{\geq 0}$. We now show that this implies that $\mu \in \Vnn$. By definition of $\Vnn$ (see \eqref{eqn_defVnn}), it is enough to show that the limit of $ \mu(t) := \exp(-t  \Beta^+) \cdot \mu$ as $t\to \infty$ exists, and this will in turn follow from the fact that the structure coefficients $ \mu(t)_{i,j}^k$ converge. We may of course assume that the orthonormal basis $\{ e_i\}$ for $\ggo$ is chosen so that it is the union of bases for $\ngo$ and $\ag = \ngo^\perp$. The nilradical of $\mu(t)$ is also $\ngo$ as a subspace, for all $t\geq 0$, since $\exp(-t \Beta^+)$ preserves $\ngo$. Let $\nu(t)$ be the bracket induced by $\mu(t)$ on $\ngo$. 

For $e_i,e_j, e_k \in \ngo$, the structure coefficients are precisely those of $\nu(t)$, and since $\nu(t) = \exp(-t \Beta_\ngo^+) \cdot \nu$ and $\nu \in V_{\Beta_\ngo^+}^{\geq 0}$, they converge (see Prop.~\ref{prop_propertiesgroups}). For $e_j, e_k \in \ngo$, $e_i\in \ag$, the structure coefficients are the entries of $\ad_{\mu(t)} (e_i)  |_{\ngo} \in \End(\ngo)$. But 
\[
    \ad_{\mu(t)} (e_i)  \big|_{\ngo} = \exp(-t \Beta_\ngo)\left( \ad_\mu(e_i)\big|_\ngo \right) \exp(t\Beta_\ngo),
\]
and moreover $\left( \ad_\mu(e_i)\big|_\ngo \right) \in \Der(\nu) \subset \qg_{\Beta_\ngo}$ by Corollary \ref{cor_autmu}. We claim that by definition of the parabolic subgroup $Q_{\Beta_\ngo}$ 
these maps converge as $t\to \infty$. Indeed, by Definition  \ref{def_groups1b} we may write $A:= \ad_\mu(e_i)\big|_\ngo \in \qg_{\Beta_\ngo}$ as a sum $A = \sum_{r\geq 0} A_r$ where $A_r\in \qg_{\Beta_\ngo}$ are eigenvectors of $\ad(\Beta_\ngo)$ with eigenvalues $r\geq 0$. We thus get
\[
	\ad_{\mu(t)} (e_i)  \big|_{\ngo}  = \Ad (\exp(-t \Beta_\ngo)) (A) = e^{- t \ad(\Beta_\ngo)} A  = \sum_{r\geq 0} e^{-tr} A_r,
\]  and the claim is now clear. 
 Finally, if $e_k \in \ngo$, $e_i, e_j \in\ag$ we have that
\begin{equation}\label{eqn_aantozero}
    \mu(t)_{i,j}^k = \big\la  \mu(e_i,e_j)\, , \,\exp(- t \Beta^+_\ngo) e_k \big\ra \underset{t\to\infty}\longrightarrow 0.
\end{equation}
Indeed, from Lemma 2.17 in \cite{standard} we know that $\Beta^+_\ngo > 0$. Thus, $\exp(- t \Beta^+_\ngo) \to 0$ as $t\to \infty$, hence $\mu(t)_{i,j}^k \to 0$ and the claim is proved. Notice that the three cases we have considered are enough, since $\ngo$ is an ideal of $\ggl$.

In what follows we replace $\mu$ by its projection $\lim_{t\to\infty} \mu(t) \in \Vzero$. After doing so, $\ngo$ is still the nilradical of $\mu$. By definition of the strata \eqref{eqn_stratumbeta} the theorem will follow once we show the following:

\vskip5pt
{\noindent \bf Step 2.} \emph{ We have that $0\notin \overline{H_{ \beta}\cdot \mu}$.}
\vskip3pt 

Notice that the Lie algebra $\hg_{ \Beta}$ of $H_{ \Beta}$ is given by 
\[
    \hg_{ \Beta} = \left\{ \twomatrix{A_1}{0}{0}{(\tr A_1) \cdot \Beta_\ngo + A_2} \, : \, A_1\in \glg(\ag), \, A_2\in \hg_{\Beta_\ngo}\right\}.
\]
This follows easily by using that $\hg_{ \Beta} = \{A\in \glg(\mg) : [A,\beta]=0, \, \tr A  \Beta = 0 \}$, and the fact that $\ag$ is the $0$ eigenspace of $ \Beta^+$ and $\ngo$ is the sum of the eigenspaces corresponding to positive eigenvalues. 

Assume on the contrary that there exists a sequence $\big(h^{(k)}\big)_{k\in \NN} \subset H_{\Beta}$ such that ${ h^{(k)} \cdot \mu \to 0}$. By acting with $\Og$ on each element of the sequence we can furthermore assume that $h^{(k)} = \exp(A^{(k)})$, with $A^{(k)} \in \hg_{\Beta} \cap \Sym(\ggo)$. Moreover, we have that $h^{(k)} = h_1^{(k)} \cdot h_2^{(k)}$, where 
\[
	h_1^{(k)} =  \minimatrix{\exp \left( A_1^{(k)} \right) }{0}{0}{ \Id_\ngo }, \qquad 
	h_2^{(k)} =  \minimatrix{\Id_\ag}{0}{0}{ \exp\left(a_k \cdot \Beta_\ngo + A_2^{(k)} \right) },
\]
with $A_1^{(k)} \in \Sym(\ag)$, $A_2^{(k)} \in \hg_{\Beta_\ngo}$, $a_k := \tr A^{(k)}_1$. The brackets induced on the nilradical $\ngo$ are 
\[
    \nu^{(k)} = \exp(a_k \Beta_\ngo) \cdot \exp(A^{(k)}_2) \cdot \nu.
\] 
Recall that $ {H_{\Beta_\ngo}} \cdot \nu \subset V^0_{\Beta_\ngo^+}$, and that the action of $\exp(a_k \Beta_\ngo)$ on $V^0_{\Beta_\ngo^+}$ is just scalar multiplication by $e^{a_k \Vert{\Beta_\ngo}\Vert^2}$ on $V(\ngo)$. 
On the other hand, since $\nu \in \sca_{\Beta_\ngo}$ by assumption, we have that $0 \notin\overline{H_{\Beta_\ngo} \cdot \nu}$, from which we deduce
$ \Vert \exp(A^{(k)}_2) \cdot \nu \Vert \geq c_\nu > 0$ for some constant $c_\nu > 0$. Thus,
\[
    \big\Vert {\nu^{(k)}} \big\Vert = e^{a_k \Vert{\Beta_\ngo}\Vert^2} \cdot 
    \big\Vert \exp(A^{(k)}_2) \cdot \nu \big\Vert \, \, \geq \, \,
    e^{a_k \Vert{\Beta_\ngo}\Vert^2} \cdot c_\nu > 0,
\]
and by $\nu^{(k)} \to 0$ we deduce $\lim_{k \to \infty}a_k = -\infty$.

To arrive at a contradiction we use the map $\alpha: H_{\Beta} \cdot \mu \to \RR$ described in Lemma \ref{lem_alpha}. 
Recall, that by assumption $h^{(k)} \cdot \mu \to 0$. Thus, by Lemma  \ref{lem_alpha}
we obtain $\alpha(h^{(k)} \cdot \mu) \to 0$. On the other hand, we have
\[
	\alpha(h^{(k)} \cdot \mu) = \big\vert \det h_1^{(k)} \big\vert^{-1}  \, \alpha(\mu),
\]
with $\alpha(\mu) > 0$. Since $\det h^{(k)}_1 = e^{a_k}$ with $a_k = \tr A_1^{(k)}$, we deduce  
$\lim_{k\to \infty}a_k =+ \infty$, which contradicts the previous paragraph.

For the case where $\nu = 0$ one proceeds analogously. The only difference in the argument is in the proof of the second step: $0 \notin \overline{H_{ \beta} \cdot \mu}$. Here one also argues by contradiction, but using that now we have $a_k \equiv 0$ (by definition of $\hg_{ \beta}$). This gives $\alpha(h^{(k)}\cdot \mu) \equiv \alpha(\mu) > 0$, contradicting the fact that $h^{(k)} \cdot \mu \to 0$.
\end{proof}

\begin{proof}[Proof of Corollary \ref{cor_Imbetaplus}]
Let $\Beta$ be as in Theorem \ref{thm_refinedbeta}, so that $\Im(\Beta^+) = \ngo$,
and suppose $\mu \in \sca_{ \Beta_1} \cap V_{\scriptscriptstyle \beta_1^+}^{ \scriptscriptstyle\geq 0}$. The proof of Theorem \ref{thm_refinedbeta}
 provides us on the one hand with $k\in \Og$ preserving $\ngo$ such that $\bar \mu := k\cdot \mu \in \sca_\Beta \cap \Vnn$. On the other hand, since $\mu \in \sca_{ \Beta_1} \cap V_{\scriptscriptstyle \beta_1^+}^{ \scriptscriptstyle\geq 0}$  the $\Og$-equivariance of the construction of the strata implies that $\bar \mu \in \sca_{\overline \Beta} \cap V_{ \scriptscriptstyle \overline \Beta^+}^{\scriptscriptstyle\geq0}$, where $\overline \Beta = k  \Beta_1 k^{-1}$.

Since $\bar \mu \in \sca_\Beta \cap \sca_{\overline \Beta}$, we may write $\overline \Beta = \bar k \Beta \bar k^{-1}$ for some $\bar k \in \Og$. Again by $\Og$-equivariance, this yields $\bar k^{-1} \cdot \bar \mu \in \sca_\Beta \cap \Vnn$. But now $\bar \mu, \bar k \cdot \bar\mu \in \Vnn$, so Proposition \ref{prop_bundlediffeosbeta} yields $\bar k\in K_\beta$ and hence $\overline \Beta = \Beta$. In this way,  $\Beta = k \Beta_1 k^{-1}$ where $k(\ngo)\subset \ngo$, and clearly $\Im( \Beta_1^+) = \Im(\Beta^+) = \ngo$. 
\end{proof}

\section{Cheeger-Gromov limits of homogeneous metrics in $\RR^n$}\label{app_chgr}

\begin{definition}\label{def_admissible}
 For  a connected,
 homogeneous space $(M^n, g)$, we call a subgroup $G \subset \Iso(M^n,g)$ {\it admissible} if it is closed, connected and acts transitively on $M^n$.
\end{definition}

Since an admissible $G \subset \Iso(M^n,g)$ acts effectively
on $M^n$,  the isotropy subgroup $H$ of a point in $M^n$ is compact. Moreover,
$M^n$ and $G/H$ are diffeomorphic. Let now $K$ be a maximal compact subgroup
of $G$ containing $H$.  Then $G$ is diffeomorphic  to $K \times \RR^m$, $m \in \NN_0$, see \cite{Mos49}, and consequently  $G/H$ is diffeomorphic to $K/H \times \RR^m$.
Thus, $G/H$  is contractible if and only if $H$ is a maximal compact subgroup,  and this in turn holds  if and only if $G/H$ is diffeomorphic to a Euclidean space (cf.~ \cite[Prop.~ 3.1]{LafuenteLauret2014b}).

\begin{theorem}\label{thm_chgrapp}
Let $(\RR^n, g_k)_{k\in \NN}$ be a sequence of homogeneous manifolds
converging in Chee\-ger-Gromov topology to  $(\bar M^n ,\bar g)$. Then, $\bar M^n$ is diffeomoporhic to $\RR^n$.
\end{theorem}

\begin{proof}
Pointed convergence of $(\RR^n,g_k,0)_{k \in \NN}$ to
$(\bar M,\bar g,\bar p)$ implies that for some $K>0$ we have 
uniform sectional curvature bounds $\vert K(g_k)\vert \leq K$ for all $k \in \NN$. The limit space is of course also homogeneous. 
In a first step, we show that $\bar M^n$ is simply connected.  To that end, suppose on the contrary that there is a non-contractible loop $\gamma$ of lenght $r>0$, based at $\bar p \in \bar M^n$. Clearly,
the image of $\gamma$ is contained in $B_{3r/4}^{\bar g}(\bar p)$. 
We now set $R:=R(K,r,n)$ as the scalar given by Lemma \ref{lem_contract}.
By definition of Cheeger-Gromov convergence, for all
 $k \in \NN$ there exists an exhaustion $(\bar U_k)_{k \in \NN}$ of
 $\bar M^n$, and diffeomorphisms $\varphi_k:\bar U_k \to 
V_k \subset \RR^n$ onto with $\varphi_k(\bar p)=0$, 
such that $\varphi_k^* (g_k\vert_{V_k})$ converges
to $\bar g$ in $C^\infty$ topology, uniformly on compact subsets of $\bar M^n$.
We now choose $k_0$ large enough so that $B_{4R}^{\bar g}(\bar p)  \subset \bar U_k $
for all $k\geq k_0$. Then for $k$ large enough we have by Lemma \ref{lem_contract} that
\[
  \varphi_k(\gamma) \subset B_r^{g_k}(0) \subset
  \Omega_R \subset B_{2R}^{g_k}(0) \subset \varphi_k(B_{4R}^{\bar g}(\bar p))\,.
\]
Since $\Omega_R \subset \RR^n$ is contractible, $\gamma$ must be contractible as well, and this is a contradiction.
The same argument shows that all higher homotopy groups must vanish,
thus $\bar M^n$ is contractible. This shows the claim.
\end{proof}

In the next Lemma,
by contractibility of the set $\Omega_R$ we mean that there exists a continuous map
$\Psi:\Omega_R\times [0,1] \to \Omega_R$ with $\Psi|_{\Omega_R \times \{ 0\}} = \Id_{\Omega_R}$ and $\Psi(\Omega_R \times \{1 \}) = \{ w\} \in \Omega_R$.

\begin{lemma}\label{lem_contract}
For any $K, r >0$ there exists  $R=R(K,r,n)>0$, such that the following holds:
for any $n$-dimensional homogeneous space $(\RR^n,g)$ with sectional curvature bounded by $\vert K(g)\vert \leq K$, there is a contractible open subset $\Omega_R \subset \RR^n$ with $B_r^g(0) \subset \Omega_R \subset B_{2R}^g(0)$. 
\end{lemma}
\begin{proof}
We use induction on $n\in \NN$,
 the case $n=1$ being clear. Let $n \geq 2$ and write $\RR^n=G/H$, with $H$ the connected,
 compact isotropy subgroup at the origin $0\in \RR^n$, and $G \subset \Iso(\RR^n, g)$ 
admissible. 

In a first step,
we consider the closed subgroup $\widetilde G \leq G$ acting on $(\RR^n, g)$ with cohomogeneity one, and with orbits diffeomorphic to $\RR^{n-1}$, as given by Lemma \ref{lem_orbitsRn-1}. Notice that this yields
a foliation of $\RR^n$ by equidistant $\widetilde G$-orbits $(O_t)_{t\in \RR}$, and $\RR^n \simeq \RR \times O_0$, where $O_0 = \widetilde G \cdot 0$.  
Let $N_t$ be the unit normal vector field to each orbit $O_t$
 along a normal unit speed geodesic $\gamma(t)$ through $0$, and 
 $\tilde  L(t)=\nabla_{(\cdot)}  N_t$ be the corresponding shape operator.  
We denote by $\tilde  g(t)$ the metric on the orbit $O_t$
induced by the metric $g=dt^ 2 +\tilde  g(t)$ on $\RR \times O_0$. 
Using the background metric 
$\tilde g_0=\tilde g(0)$ on the fixed vector space $W=T_0 O_0$  we consider 
$\tilde g(t)$ a positive definite endomorphism  on $W$
by setting  $\tilde g(t)(X,X)=\tilde g_0(\tilde g(t)X,X)$ for all $X \in W$.
Then by  \cite{EscWng00} we have
\begin{eqnarray}
   \tilde  g'(t)&=& 2 \tilde  g(t)\cdot \tilde  L(t). \label{shapeode}
\end{eqnarray} 
Recall now the  Riccati equation  
 \begin{equation}\label{eqn_riccati}
 \tilde  L'(t)+\tilde  L^2(t)+\Riem_{N_t}=0 \,,
 \end{equation}
where $\Riem_{ N_t}(X)=\Riem_{X,N_t}^{g}N_t$ for 
$X$ tangent to $O_t$: see \cite[$\S$2]{EscWng00}. 
Setting
\[
 \varphi(t):= \max \left\{ \tilde g(t) \big(\tilde L(t)\cdot v,v \big) \mid \tilde g(t)(v,v)=1  \right  \}\,
\]
it follows that
\[
  \varphi(t) =\max \left\{ \tilde g_0\big(\tilde L_t^0 \cdot v_0,v_0\big) \mid \tilde g_0(v_0,v_0)=1\right\}\,.
\]
where
$\tilde L_t^0:= \sqrt{\tilde g(t)}\cdot \tilde L(t) \cdot \sqrt{\tilde g^{-1}(t)}$.
Now by \cite{libro}, p. 531, for the Dini derivative
\[
  \tfrac{d^+\varphi}{dt}(t):=\limsup_{s \to 0,s>0} \frac{\varphi(t+s)-\varphi(t)}{s}
\]
we have
\[
  \tfrac{d^+\varphi}{dt}(t) =\max \left\{  \tfrac{d}{dt}\,\,
  \tilde g_0\big( \tilde L_t^0  \cdot v_M,v_M\big) \mid 
    \tilde g_0\big( \tilde L_t^0  \cdot v_M,v_M\big) =  \varphi(t) \right\} \,.
\]
By \eqref{eqn_riccati} this implies
$$
 \tfrac{d^+\varphi}{dt}(t) = -( \varphi^2(t) + r(t) )
$$
with $\vert r(t)\vert \leq K$, since
by assumption 
the absolute value of the largest eigenvalue of $\Riem_{N_t}$ is bounded by $K$.
Suppose now, that there exists $t_0\in \RR$ such that 
$\vert \varphi(t_0) \vert \geq (K+1)^2$. Then
using that
$$
    \tfrac{3}{2} \cdot \varphi(t_0)^2 \geq \varphi(t_0)^2 + r(t_0) \geq 
      \tfrac{1}{2} \cdot \varphi(t_0)^2
$$
and the fact that no solution to $y'(t)=c \cdot y^2(t)$, $c \in [\tfrac{1}{2},\tfrac{3}{2}]$,
with $y(t_0) \neq 0$ exists for all times, we obtain a contradiction.

By the first step we deduce that ${\rm tr} \tilde  L^2(t) \leq (n-1) \cdot (K+1)^4$
for all $t \in \RR$
and  the Gau{\ss}  equation  yields now a uniform bound on the sectional curvature of the hypersurface  $(O_0=\widetilde G\cdot 0,\tilde  g_0)$, 
say $\vert K(\tilde  g_0) \vert \leq \tilde  K = \tilde K(K, n)$, where $\tilde K$ is independent of $g$.
  
For $r>0$ let  $c:[0,r]\to \RR^n$ be a unit speed geodesic with $c(0)=0$.
Using  $\RR^n =\RR \times   O_0$ we write
$c(s)=(t(s),\tilde  c(s))$, where $\tilde  c:[0,r]\to  O_0$ is smooth.
By the above, for $X \in W$ with $\tilde g(t)(X,X)=1$
we have $\vert \tilde g(t)(\tilde L(t)X,X) \vert \leq  \sqrt{(n-1)}\cdot (K+1)^2$.
A short computation shows,
that $\tilde L_t^0$ is  $\tilde g_0$-self-adjoint and that its eigenvalues
obey the very same estimate. Now by \eqref{shapeode}
and Cauchy-Schwarz we deduce
\[
 (\tr \tilde g(t))'
  =
    2\tr \big( \tilde g(t) \cdot \tilde L_t^0 \big)
    \leq 
      2\sqrt{\tr \tilde g^2(t)}\cdot \sqrt{\tr (\tilde L_t^0)^2}
    \leq 
      2\sqrt{n-1} \cdot (K+1)^2 \cdot \tr \tilde g(t)\,,
\]
since $\tilde g(t)$ is positive definite.  By \eqref{shapeode} this gives uniform upper and lower bounds for $\tilde g(t)$
on the time intervall $[0,r]$,  depending only on $K,r$ and $n$.
As a consequence, there exists $\tilde  r=\tilde r(K,r,n)>0$ with 
$L^{\tilde  g_0}(\tilde  c) \leq \tilde  r-1$ for all such curves $c$. 
Finally, by  induction hypothesis for $\tilde  R := R(\tilde  K,\tilde  r,n-1)$  
there exists a contractible set $\tilde  \Omega_{\tilde  R} \subset  O_0$
with $B_{\tilde  r}^{\tilde  g_0}(0) \subset 
\tilde  \Omega_{\tilde  R} \subset  B_{2\tilde  R}^{\tilde  g_0}(0)$, since
$\tilde G/\tilde H=O_0$ is diffeomorphic to $\RR^{n-1}$. 
We set 
$\tilde  R :=\max(\tilde  R,r)$, $R:=2\tilde  R$ and
$\Omega_R:=[-\tilde  R,\tilde  R] \times  \Omega_{\tilde  R}$.
Then we  have  $B_r^g(0) \subset \Omega_R \subset B_{2R}^g(0)$, $\Omega_R$ is contractible and $R$ does only depend on $K,r$ and $n$ but 
not on $g$.
\end{proof}

To conclude, we now prove the Lie-theoretic result 
needed in  Lemma \ref{lem_contract}:

 \begin{lemma}\label{lem_orbitsRn-1}
 Let $g$ be a homogeneous metric on $\RR^n$ and
 $G \leq \Iso(\RR^n,g)$ admissible.
  Then, there exists a closed, connected subgroup 
  $\widetilde G \leq G$ acting on $(\RR^n,g)$ with cohomogeneity one, with orbits diffeomorphic to $\RR^{n-1}$. 
 \end{lemma}

 \begin{proof}
It suffices to prove the claim for an admissible $ G$ of minimal dimension.
If $G$ is solvable, then by minimality and \cite[Lemma 1.2]{GrdWls} it must act simply transitively on $\RR^n$. It is well-known that solvable Lie groups have codimension-one normal subgroups. In our case, since $H = e$ is trivial, $G$ is simply connected, hence such a subgroup must be diffeomorphic to $\RR^{n-1}$ and it  must act freely on $(\RR^n,g)$.

For a non-solvable $G$ let $\ggo = \ug \ltimes\sg$ be a Levi decomposition, with $\ug = \ug_1\oplus \cdots \oplus \ug_r$ a sum of simple ideals. Consider the ideal $\ggo_2 := (\ug_2 \oplus \cdots \oplus \ug_r) \ltimes \sg$ in $\ggo$, and let $U_1, G_2$ be the connected Lie subgroups of $G$ with Lie algebras $\ug_1, \ggo_2$, respectively. For $U_1 = K A N$ an Iwasawa decomposition we have that $K, A$ and $N$ are connected. Set $G_3 := AN G_2$. Since $AN \leq G$ and $G_2 \vartriangleleft G$, it follows that $G_3$ is a subgroup of $G$, with Lie algebra $\ggo_3 = \ag \oplus \ngo \oplus \ggo_2$. Denote by $\widetilde G:=  \overline G_3$ its closure in $G$, with Lie algebra $\tilde \ggo$. Notice that $\widetilde G\lneq G$. Indeed, we have that $[\tilde \ggo, \tilde\ggo] = [\ggo_3, \ggo_3] \subset \ggo_3$ by \cite{Mal44} (cf.~ also \cite[Ch.~I, $\S$4, Thm.~3]{OniVin90}), thus $\tilde \ggo \neq \ggo$, otherwise $\ggo_3$ would contain $\ug_1 = [\ug_1, \ug_1]$. From this fact and the  minimiality of $G$ it follows that $\dim \widetilde G\cdot p \leq n-1$ for all $p\in \RR^n$. 

We now claim that $G = K \widetilde G$. To that end, first notice that the subset $U_1 G_2$ is in fact a subgroup of $G$. Since it contains a neighbourhood of the identity, it equals $G$. From $U_1 = K A N$ we conclude that $G = K G_3$, which implies the claim.

Recall that by the uniqueness (up to conjugation) of maximal compact subgroups of a connected Lie group, given any compact subgroup $\widetilde K \leq G$ there exists $p\in \RR^n$ such that $\widetilde K \subset G_p$. Indeed, for some $f\in G$ we have $f \widetilde K f^{-1} \leq H$, the isotropy at the origin $0\in \RR^n$, and then we simply set $p := f^{-1}\cdot 0$. 

Let now $\widetilde K$ be a maximal compact subgroup of $\widetilde G$, so that 
$\widetilde G$ is diffeomorphic to $\widetilde K \times \RR^m$ for some $m\in \NN$. We claim that $m = n-1$. To see that, let $p\in \RR^n$ be such that $\widetilde K \leq G_p$. Then the $\widetilde G$-isotropy $\widetilde G_p = G_p \cap \widetilde G$ is a compact subgroup of $\widetilde G$ containing $\widetilde K$, thus equal to $\widetilde K$, and hence $\widetilde G \cdot p \simeq \widetilde G / \widetilde K \simeq \RR^m$. By the above this yields $m\leq n-1$. On the other hand, we may write $K = K_{ss} Z(K)$ where $K_{ss} = [K,K]$ is semisimple and compact, and $Z(K)$ is the center. From the classification of real simple Lie groups one sees that $\dim Z(K) \leq 1$. Since $K_{ss}$ is a compact subgroup of $G$, there exists $f\in G$ such that $f K_{ss} f^{-1} \leq G_p$. We may moreover assume that $f\in \widetilde G$, since $G = K \widetilde G=\widetilde G K$ and $K$ normalizes $K_{ss}$. Therefore, $K_{ss} \leq G_q$ for $q := f^{-1}(p)$, and we have that $ \widetilde G \cdot q = \widetilde G \cdot p \simeq \RR^m$
and $\dim K \cdot q = \dim  Z(K) \cdot q \leq 1$.
Since $G = K \widetilde G$ is transitive, we must have $m \geq n-1$ and $\dim Z(K) = 1$. Thus, $m = n-1$. 
 
To finish the proof, we must show that for any $p\in \RR^n$ it holds that $\widetilde G \cdot p \simeq \RR^{n-1}$. Given $p\in \RR^n$, the isotropy $\widetilde G_p = G_p \cap \widetilde G$ is a compact subgroup, thus contained in some maximal compact subgroup $\widetilde K_p \leq \widetilde G$. Since $\widetilde G$ is diffeomorphic to $\widetilde K_p \times \RR^{n-1}$, the orbit $\widetilde G \cdot p$ is diffeomorphic to the product of the connected compact homogeneous space $\widetilde K_p / \widetilde G_p$ and $\RR^{n-1}$. By dimensional reasons, using that $\widetilde G$ is not transitive by minimality of $G$, we deduce that $\widetilde K_p = \widetilde G_p$ and $\widetilde G \cdot p \simeq \RR^{n-1}$.
 \end{proof}

\end{appendix}


 \bibliography{ramlaf2}
\bibliographystyle{amsalpha}

\end{document}